\documentclass[11pt]{article}%
\usepackage[colorinlistoftodos]{todonotes}

\usepackage{amssymb}
\usepackage{amsfonts}
\usepackage{amsmath}
\usepackage{mathtools}
\usepackage{dsfont}

\usepackage{comment}
\usepackage{amsthm}

\usepackage{hyperref}
\hypersetup{
    colorlinks=true,
    linkcolor=blue,
    citecolor=red,
    urlcolor=blue,
    pdfborder={0 0 0}
}

\usepackage{cleveref}

\usepackage{float}
\usepackage{graphicx}
\usepackage{wrapfig}
\usepackage{subfig}
\usepackage{fancybox}
\usepackage{framed}
\usepackage{caption}
\usepackage{epstopdf} 
\usepackage[all]{xy} 
\usepackage{tikz} 
\usepackage{tabularx}
\usepackage{afterpage}
\usepackage{arydshln}

\usepackage[titletoc, title]{appendix} 
\usepackage{titling} 
\usepackage{url} 
\usepackage{color}
\usepackage{enumitem}
\usepackage{multirow, longtable} 

\usepackage{pdfpages}

\usepackage[latin1, utf8]{inputenc}
\usepackage[a4paper]{geometry}
\usepackage{microtype}

\providecommand{\U}[1]{\protect\rule{.1in}{.1in}}
\newtheorem{theo}{Theorem}[section]
\newtheorem{conjecture}{Conjecture}
\newtheorem{prop}[theo]{Proposition}
\newtheorem{lem}[theo]{Lemma}
\newtheorem{cor}[theo]{Corollary}

\newtheorem{rem}[theo]{Remark}

\numberwithin{equation}{section}
\newtheorem{as}{Assumption}[section]

\newcommand{\CC}{\mathbb{C}}
\newcommand{\DD}{\mathbb{D}}
\newcommand{\D}{\mathbb{D}}

\newcommand{\EE}{\mathbb{E}}
\newcommand{\E}{\mathbb{E}}

\newcommand{\NN}{\mathbb{N}}

\newcommand{\PP}{\mathbb{P}}

\newcommand{\RR}{\mathbb{R}}

\newcommand{\ZZ}{\mathbb{Z}}

\def\E{\EE}
\def\P{\mathbb{P}}

\newcommand{\Ca}{ {\mathcal C }}

\newcommand{\cE}{ {\mathcal E }}

\newcommand{\cS}{ {\mathcal S }}
\newcommand{\cC}{ {\mathcal C }}

\newcommand{\Qb}{ {\mathbf Q }}
\newcommand{\Eb}{ {\mathbf E }}

\newcommand{\Bb}{ {\mathbf B }}

\newcommand{\Hb}{ {\mathbf H }}
\newcommand{\Pb}{ {\mathbf P }}
\newcommand{\Kb}{ {\mathbf K }}

\newcommand{\pb}{ {\mathbf p }}
\newcommand{\qb}{ {\mathbf q }}

\def\eps{\varepsilon}

\def\bl{\boldsymbol{\lambda}}

\newcommand{\R}{\mathbb{R}}

\def\bl{\boldsymbol{\lambda}}

\newcommand{\les}{ {<}}
\newcommand{\gre}{ {>}}

\newcommand{\dps}{\displaystyle}

\newcommand{\var}{\mathrm{Var}}
\newcommand{\cov}{\mathrm{Cov}}
\newcommand{\Hrad}{H_{\mathrm{rad}}}
\newcommand{\Hcirc}{H_{\mathrm{circ}}}
\newcommand{\hrad}{h_{\mathrm{rad}}}
\newcommand{\hcirc}{h_{\mathrm{circ}}}
\newcommand{\thrad}{{\tilde h}_{\mathrm{rad}}}
\newcommand{\thcirc}{{\tilde h}_{\mathrm{circ}}}

\newcommand{\red}[1]{{\color{red} #1}}

\allowdisplaybreaks

\newcommand{\nb}[1]{\textcolor{blue}{#1}}

\title{Spectral expansion of LQG heat trace and KPZ scaling}

\author{
Nathana\"el Berestycki\thanks{University of Vienna, \href{mailto:nathanael.berestycki@univie.ac.at}{nathanael.berestycki@univie.ac.at}}\:
and Jakob Klein\thanks{ University of Vienna,
\href{mailto:jakob.klein@univie.ac.at}{jakob.klein@univie.ac.at}}\: 
} 
\date{\today}
\begin{document}
\maketitle

\begin{abstract}
Let $h$ be a whole plane Gaussian free field, and let $\Omega$ be a bounded domain in two dimensions. We study the asymptotics as $t\to 0$ of the Liouville quantum gravity (LQG) heat trace, defined as the integral over $\Omega$ of the on-diagonal LQG heat kernel. Our main result is to show that the second term in the spectral expansion as $t\to 0$ of the expected heat trace is governed by a nontrivial exponent, given by the KPZ (Knizhnik--Polyakov--Zamolodchikov) relation. A similar but stronger (almost sure) result applies to the related notion of heat content. Along the way we obtain various results on the short-term behaviour of the heat kernel, notably solving a conjecture of \cite{BW} concerning its annealed asymptotics, and showing the finiteness of all moments of the properly rescaled heat kernel. 
\end{abstract}

\tableofcontents

\section{Introduction, main results and conjectures}

\subsection{Background}
In classical Riemannian geometry, the analysis of the \emph{heat trace} (the space integral of the on-diagonal heat kernel, see definitions below) is a key way of identifying spectral invariants, i.e., quantities that are spectrally determined, such as the volume or the length of the boundary. Its short term asymptotics are closely related to some fundamental parameters of the underlying manifold, including its Euler characteristic, eigenvalue counting function, and the zeta-regularised determinant of its Laplacian. 

In this paper we consider short time asymptotics of the heat trace in the context of \textbf{Liouville quantum gravity} (LQG), a certain canonical random geometry initially proposed by Polyakov \cite{polyakovstrings} for non-critical bosonic string theory in two dimensions. We are motivated not only by the programme initiated in \cite{BW} (which initiated the study of many questions in the spectral geometry of LQG, see also \cite{BerestyckiICMP}) but also by the fact that spectral geometry should provide an alternative way of defining and studying LQG; for instance, reweighing the law of an LQG surface corresponding to a certain coupling constant $\gamma \in (0,2)$ by an appropriate power -- the central charge -- of its zeta-regularised determinant (if it can be well-defined) should lead to an LQG surface with a different coupling constant $\gamma'$; see for instance \cite{AngParkPfefferSheffield} and \cite{BerestyckiICMP}. 

The leading order behaviour of the LQG heat trace $
\mathbf{H}_\Omega(t)$ as $t\to 0$ was described in \cite{BW} and was found to be proportional to $(1/t)$ times the LQG area of the surface, leading to an (almost sure) Weyl law for the eigenvalue counting function, albeit with a nontrivial, $
\gamma$-dependent proportionality constant $c_\gamma$. In this paper we consider the next term in this expansion, which, if the Riemannian template is to be believed, should be governed by boundary effects and should be of order $1/\sqrt{t}$. In fact, our main finding is that this next term is determined by boundary effects which are considerably more subtle in the LQG case than in the Riemannian situation. In particular, its behaviour is described by a nontrivial exponent which is related to the so-called \textbf{KPZ equation} (after Knizhnik, Polyakov and Zamolodchikov), the relation which can be used to transform Euclidean scaling exponents of a set into its LQG analogue \cite{KPZ, DS2011, RhodesVargas,  HKPZ, Aru, KPZas}; see also \cite[Chapter3]{BP} for an overview.

Let $h$ denote a whole plane Gaussian free field, normalised so its unit circle average vanishes (see \cite[Section 6.4.1]{BP} for a definition, and see also \eqref{eq:Greens function of full plane GFF} for some basic properties and reminders about this field). Let $\Omega \subset \mathbb{C}$ be a bounded, simply connected domain, and consider the restriction $h|_
\Omega$ of $h$ to $\Omega$.
Let $\mu$ denote the corresponding GMC (or Liouville area measure), which corresponds informally to the exponential of the field $h$: $\mu(dx) = \exp ( \gamma h(x) ) dx$. However, this definition is not rigorous as the GFF is not pointwise defined, and requires a regularisation procedure: 
$$
\mu(dx) = \lim_{\eps \to 0} \eps^{\gamma^2/2} e^{\gamma h_\eps(x) } dx
$$
as $\eps \to 0$ in probability with respect to the topology of vague convergence;  see e.g., \cite{Ber2017}, \cite{Shamov} or \cite{BP} for the existence of this limit. 

\medskip Let $(\mathbf{Z}_t, 0 \le t\le \sigma_\Omega )$ denote the associated Liouville Brownian motion (\cite{GRV, Ber2015}, killed upon leaving $\Omega$ (so $\sigma_{\Omega} = \sigma_\Omega(\mathbf{Z}) = \inf \{ t>0: \mathbf{Z}_t \notin \Omega\}$). We review its construction in Section \ref{SS:prelim}, but for now let us simply mention that Liouville Brownian motion can be viewed as the canonical diffusion in the geometry of LQG, and is for instance conjectured to describe the scaling limits of random walks on corresponding models of random planar maps (see e.g. \cite{BerestyckiGwynne} for results in this direction).

Let $\pb^\Omega_t(x,y)$ denote the corresponding \textbf{heat kernel}, so that, by definition \cite{GRV_hk}, $\pb^\Omega_t(x,y)$ is the Radon--Nikodym derivative of the law of $\mathbf{Z}_{t}$, restricted to the event $\{  \sigma_\Omega > t \}$, with respect to $\mu$, when $\mathbf{Z}$ is started from $x \in \Omega$. A jointly continuous modification in $\Omega^2 \times (0,t)$ is known to exist (\cite{MRVZ}) and we will only refer to this from now on. (Strictly speaking, these results were stated for Gaussian free fields with Dirichlet boundary conditions on $\partial \Omega$ rather than the whole plane GFF restricted to $\Omega$ which we consider here, but by local absolute continuity these results can easily be transported to this framework).

The main object of investigation in this paper is the \textbf{LQG heat trace}:
\begin{equation}
\label{eq:trace}
\mathbf{H}_\Omega(t) = \int_\Omega \mathbf{p}^\Omega_t(x,x) \mu(\mathrm{d}x). 
\end{equation}
Although this won't be needed in the paper, we recall that by the trace formula, $\mathbf{H}_\Omega(t)$ is the Laplace transform of the nonnegative measure associated to the increasing eigenvalue counting function: if
$\{\bl_n\}_{n\ge 1}$ denotes the sequence of eigenvalues associated to the corresponding Liouville Brownian motion in nondecreasing order, and $\mathbf{N}(\lambda) = \sum_{n} 1_{\{ \bl_n \le \lambda\}}$ is the eigenvalue counting function, then $
\mathbf{H}_\Omega(t) = \int_0^\infty e^{- \lambda t} \mathrm{d} \mathbf{N} (\lambda)$.

\medskip It was proved in \cite{BW} that the leading order asymptotics of the heat trace is of order $(1/t)$ and is proportional to the LQG area of $\Omega$:
\begin{equation}\label{eq:trace_vol}
t \mathbf{H}_\Omega(t) \to c_\gamma \mu(\Omega)
\end{equation}
in probability as $t\to 0$, where $c_\gamma = 1/ [ \pi (2- \gamma^2/2)]$. Using a small extension of Tauberian theory and the above trace formula, this implied the following result for the eigenvalue counting function:
\begin{equation}\label{eq:Weyl}
\lambda^{-1} \mathbf{N}(\lambda) \to c_\gamma \mu(\Omega), 
\end{equation}
 (Once again, strictly speaking these results were established for a Gaussian free field with Dirichlet boundary conditions on $\partial \Omega$ instead of a quantum cone restricted to $
\Omega$ as we do here, but there is no difficulty in extending that result to the setup of quantum cones considered here; see Section \ref{S:proof_heat} for more details.)
Thus, the LQG eigenvalues satisfy a (quenched) Weyl law, with spectral dimension $d_s =2$ but with anomalous proportionality constant $c_\gamma$ instead of the  constant $c_0 = 1/(2\pi)$ observed in the standard Riemannian framework.

\subsection{Heat trace results}

\medskip Our goal in this paper is on the one hand, to go further in the expansion of the heat trace near $t=0$ compared to \eqref{eq:trace_vol}, and on the other hand, to discuss the asymptotics of related quantities such as the \textbf{heat content}, discussed below. In the Riemannian setting, it is well known (see e.g., \cite{Kac1966}) that 
\begin{equation}\label{eq:heat_exp_Riemann}
    H_\Omega(t): = \int_\Omega p^\Omega_t(x,x) dx = \frac1t c_0 \text{Leb} (\Omega) - \frac1{\sqrt{t}} c'_0 |\partial \Omega|,
\end{equation}
at least when $\Omega$ is a smooth domain. (Here $H_\Omega(t)$ is the heat trace of standard Brownian motion in $\Omega$, and $p^\Omega_t(x,y)$ is the transition probability of standard Brownian motion, killed upon leaving $\Omega$, and $|\partial \Omega|$ denotes the boundary length of $\Omega$.)  While the first term above matches with 
\eqref{eq:trace_vol} (replacing $c_0$ by $c_\gamma$), we will see that the order of magnitude of the next term, averaged over the realisations of the field, is very different and is instead of order $t^{-(1-\Delta)}$, where $\Delta$ is the \textbf{KPZ exponent} of the boundary, as will be defined below.

This appears to confirm numerical simulations performed in \cite{BW}, where it was observed numerically that the leading order exponent for the second term in the expansion of $\mathbf{H}_\Omega(t)$ near $t =0$ is anomalous, i.e., deviates (substantially) from $1/2$.


\medskip To state our results precisely, we first introduce carefully the exponents which play a role. Suppose that $\Omega$ is a bounded domain, and let $\partial \Omega$ be its boundary. We will need to consider the fractal dimension (more precisely, Euclidean scaling exponent) of $\partial \Omega$, 
\emph{as viewed from the inside} of $\Omega$. (This corresponds to what is sometimes called the \emph{inner} Minkowski dimension \cite{Falconer}, more precisely the inner Minkowski dimension would be given by $d = 2- 2\mathsf{x}$ in the notations below.) 
As the domain above need not be simply connected (so that $\partial \Omega$ need not be connected, and could include Cantor-like sets of nontrivial dimensions), we also fix a connected component, denoted by $\partial^*\Omega$, of $\partial \Omega$, and which we assume is of positive diameter.

\begin{as} \label{As:innerMinkowski}
Suppose  there exists $\mathsf{x} \in [0,1]$ such that 
%
$\Omega_\eps$ denotes the inner $\eps-$neighbour\-hood of $\partial \Omega$ (i.e., $\Omega_\eps = \{ z \in \Omega: d(z, \partial \Omega) \le \eps\}$), then
\begin{equation}\label{eq:scaling_informal}
\mathrm{Leb} \left( \Omega_\eps \right) = \eps^{2 \mathsf{x}+o(1)}.
\end{equation}
We also suppose that if $\Omega^*_\eps = \{ z\in \Omega: d( z, \partial^*\Omega) \le \eps\}$, then also $ \text{Leb} ( \Omega^*_\eps ) = \eps^{ 2 \mathsf{x} + o(1)}$ as $\eps \to 0$.  
\end{as}

%

We call $\mathsf{x}$ defined by \eqref{eq:scaling_informal} the  \emph{inner} scaling exponent of $\partial \Omega$. Since $\partial^* \Omega$ is connected, note that $\mathsf{x} \in [0,1/2]$ necessarily, with $\mathsf{x}  = 1/2$ if $\partial \Omega$ is smooth.

\begin{rem}
    The reader may wonder if the inner Minkowski dimension (or equivalently inner scaling exponent) ever differs from its global version. Equivalently, can  the outer dimension of a domain $\Omega$ ever  be strictly larger than its inner dimension? 
It is in fact possible to construct examples of such sets. We will provide in Appendix \ref{App:Minkowski} an example of such a set.
\end{rem}

\medskip Thus, suppose that \eqref{eq:scaling_informal} holds. 
Let 
 $\Delta\in[0,1]$ be the unique solution of the KPZ equation
        \begin{equation}
\label{eq:KPZ}
\mathsf{x} = \frac{\gamma^2}{4} \Delta + (1- \frac{\gamma^2}{4})\Delta^2,
        \end{equation}
where $x$ is the inner scaling exponent associated to $\partial \Omega$, as in 
\eqref{eq:scaling_informal}.

\begin{theo}\label{T:qc} Let $h$ be as above. Then as $t \to 0$,
    \begin{equation}
    \EE[\mathbf{H}_\Omega(t)] = \frac1t c_\gamma \EE[ \mu(\Omega)] - t^{-1 + \Delta + o(1)}.
    \end{equation}
    \end{theo}

Again, we emphasise that this result gives information only at the level of exponents for the second term in $\E[\Hb_\Omega(t)]$. Concretely, it says 
$$
\frac{\log \left(\EE[\mathbf{H}_\Omega(t)] - \frac1t c_\gamma \EE[ \mu(\Omega)]\right)}{\log t} \to -1+ \Delta
$$
as $t\to 0$. 
It is unclear at this stage whether one should expect the same result in an almost sure sense, i.e., without taking the expectation (over the realisations of the field $h$). 

\medskip Beyond Theorem \ref{T:qc}, it would be particularly natural and interesting to be able to analyse where $(\Omega, h)$ is a so-called unit \textbf{quantum disc} (of some particular weight, and conditioned to have unit boundary length, say). 
Informally, this corresponds to the above case where $\Omega$ is bounded by an SLE$_\kappa$-type curve. Using the well known theorem of Beffara \cite{Beffara} for the value of the dimension of SLE (and a recent strengthening of this result by Powell and Sepulveda \cite{PowellSepulveda}, which shows that the dimensions of this curve as computed on each side separately agree), and solving the KPZ equation we would find $\Delta = 1/2$. From Theorem \ref{T:qc} it is thus natural to conjecture:

\begin{conjecture}\label{Conj:QD}
    Let $(\Omega, h)$ be a quantum disc as above. Then 
    $$
\E [ \mathbf{H}(t)] =  \frac1t c_\gamma \EE[ \mu(\Omega)] - t^{-1/2 + o(1)}
    $$
\end{conjecture}

Thus for quantum discs one should recover (at least the level of exponents) the same behaviour as in the Riemannian case, \eqref{eq:heat_exp_Riemann}. Although such a fact is heavily suggested by our results, we are currently unable to prove this at the moment,  for reasons which we believe to be only technical. We are nevertheless able to prove a statement which is ``morally equivalent'' and which concerns the \textbf{heat content}, as we now explain.

\subsection{Heat content}

 By definition, the \textbf{heat content} $\mathbf{K}_\Omega(t)$ is given by 
$$
\mathbf{K}_\Omega(t) = 
\int_\Omega \mathbf{P}_x[ \sigma_\Omega<  t] \mu (\mathrm{d}x),
$$
where we recall that $\sigma_\Omega = \inf\{ t>0: \mathbf{Z}(t) \notin \Omega\}$
is the exit time of $\Omega$ by our Liouville Brownian motion, and $\mathbf{P}_x$ denotes its (quenched) law starting from $x\in \Omega$. 

\medskip The heat content has the following {physical interpretation}: imagine that at time $t=0$ the domain $\Omega$ has temperature $T=0$ everywhere. Then starting from time $t=0^+$, heat up the boundary with a constant temperature $T=1$. The quantity $\mathbf{P}_x( \tau_\Omega> t)$ then measures the temperature at a point $x$ and time $t$ in $\Omega$, so the heat content measures the integrated temperature (``total heat'') in $\Omega$ at time $t>0$. Now, as time increases, the temperature inside $\Omega$ gradually increases from $T=0$ at time $t =0$ to $T=1$ at time $t=\infty$, and it is of particular interest to describe the asymptotics near $t=0$ of this behaviour. Beyond its physical interpretation and connection to heat trace (which we will detail below), the heat content is also a quantity of interest from the point of view of spectral geometry of Euclidean domains and more generally Riemannian manifolds; see, e.g., \cite{McD_Meyers, vdB2, vdB_Gittins2, vdB_Gittins1}, and references therein (note however that the heat content is not a spectral invariant).

\medskip Clearly, from the above description one might expect that the short time behaviour of $\mathbf{K}_\Omega(t)$ is dominated by boundary effects, and this may lead the reader to wonder if there is a connection between the second term asymptotics of $\mathbf{H}_\Omega(t)$ and the leading order of $\mathbf{K}_\Omega(t)$. The answer turns out to be yes: informally, for reasons that will be clarified later, one should expect that the second term in the expansion of $\E[\mathbf{H}_\Omega(t)]$ near $t =0$ is approximately $\mathbb{E}[\mathbf{K}_\Omega(t)]/t$. For this reason, a key step in our proof of Theorem \ref{T:qc} is to establish the asymptotic behaviour of the heat content $\mathbf{K}_\Omega(t)$ near $t =0$. 

\medskip Our results establish convergence in probability for $\mathbf{K}_\Omega(t)$. For this we require a stronger version of \eqref{eq:scaling_informal}, where we control the size of the annuli centered at points on the boundary intersected with $\Omega$, uniformly over all scales and over the boundary. In the language of fractal geometry \cite{Falconer}, this corresponds to an \emph{inner} and slightly weaker version of the notion of \textbf{local Assouad dimension}: 

\begin{as} \label{assumption 1}
There is a constant $C_\partial >0$ such that for all $0<r<R\le R_0$ we have
    \begin{align}
        \label{eq:assumption_upper_bound_global}\sup_{z\in\partial\Omega}\Big\vert \Omega_r\cap B_R(z)\Big\vert\le C_\partial r^{2\mathsf{x}}R^{2(1-\mathsf{x})}.
    \end{align}
\end{as}
The above gives an upper bound on the size of $r$-neighbourhood of the boundary, uniformly across the boundary, and at all mesoscopic scales. In particular, this implies
\begin{align} \label{Assumption 1 non local}
    \vert\Omega_r\vert\le \tilde  C_\partial r^{2\mathsf{x}},
\end{align}
for some constant $\tilde C_\partial $. We also require a complementary lower bound, which says that for every choice of scale $0<R<1$, we have a polynomial number of good points in $\partial^* \Omega$ next to which the volume growth of the boundary of $\Omega$ grows in a controlled way (recall that $\partial^* \Omega$ is some fixed component of $\partial\Omega$ of positive diameter). 

\begin{as} \label{Assumption 2}
    For all sufficiently small $\zeta>0 $ there exists a constant $c_\partial = c_\partial(\zeta)>0$ satisfying the following properties. 
    For every choice of $0<R<1$ there exists integer $N\ge 1\vee R^{-2(1-\mathsf{x})+\zeta}$ and points $z_1, \ldots, z_N \in  \partial^* \Omega$ with $B_R (z_i) \cap B_R(z_j) = \emptyset$ if $1\le i \neq j \le N$, and for all $0 < r <R$, 
    \begin{align}\label{eq:assumption_lowerbound_local}
        \min_{1\le i\le N}\Big\vert \Omega^*_r\cap B_{R}(z_i)\Big\vert\ge c_\partial r^{2\mathsf{x}}R^{2(1-\mathsf{x})}.
    \end{align}
    We call $\partial_R = \{ z_1, \ldots, z_N\}$ this set of points. 
\end{as}

Assumption \ref{assumption 1} can be viewed as an upper bound on the dimension, while Assumption \ref{Assumption 2} is a lower bound; we require here not only that there is a single point where the local behaviour matches the global upper bound  \eqref{eq:assumption_upper_bound_global} but an abundance of such points, which of course is typically the case for most natural random fractals. The requirement that $N\ge R^{-2(1-\mathsf{x})+\zeta}$ is added at little cost, and corresponds to the fact that the dimension of $\partial \Omega$ is assumed to be $2-2\mathsf{x}$. (Since $R$ is typically small, the presence of the factor $R^{\zeta}$ in the assumption makes the result below stronger, since it means that any polylogarithmic terms in the Minkowski gauge can be ignored). 

\begin{theo}
\label{T:heat} Let $h$ be as in Theorem \ref{T:qc}, i.e., $h$ is a full plane GFF normalized to have $0$ average on the unit circle. Let $\Omega$ be a bounded domain satisfying Assumptions \ref{assumption 1} and \ref{Assumption 2}. Then 
$
\mathbf{K}_\Omega(t) = t^{\Delta + o(1)}
$ where $o(1)$ converges to 0 in probability. That is, 
$$
\frac{\log \mathbf{K}_\Omega(t)}{\log t} \to \Delta
$$
as $t\to 0$ in probability. \end{theo}


Note that for this result there is no need to average over the realisations of the GFF; the convergence holds in probability. However, the assumption on the boundary of $\partial \Omega$ can be relaxed to \eqref{eq:scaling_informal} if all one desires is a control over $\E[ \Kb_\Omega(t)]$: indeed, as our results show, the estimate $\E [\Kb_\Omega(t) ] = t^{\Delta + o(1)}$ holds as soon as Assumption \ref{As:innerMinkowski} holds.

As already discussed above, the case of quantum discs (\cite{HRVdisk, DuplantierMillerSheffield, AruHuangSun}) is particularly interesting. In that setting it should be possible based on the same ideas used in the proof of Theorem \ref{T:heat} to show the following result.

\begin{conjecture}
   \label{T:qd_hc}
   Let $(\Omega, h)$ be a quantum disc, as defined in \cite{HRVdisk}. 
Then
$$
\frac{\log \mathbf{K}(t)}{\log t} \to \Delta = 1/2
$$
in probability.
\end{conjecture}

\subsection{Results pertaining to the heat kernel}

A key aspect of the work to prove Theorem \ref{T:qc} is to analyse the short-term behaviour of the heat kernel under the \emph{rooted} measure $\P^*_x$, which may be thought of as the law of the field viewed from a quantum-typical point $x$ (formally this is the Peyrière rooted measure arising from the Girsanov lemma and which may be described as the law of $h (\cdot) + \gamma G(x, \cdot)$, where $G(x, \cdot) $ denotes the covariance of the field $h$, see \eqref{eq:Greens function of full plane GFF}). In particular, we resolve a conjecture from \cite{BW}. We summarise our results as follows:

\begin{theo}\label{T:limit_law}
    Let $\Omega$ be any domain and let $x \in \Omega$. Consider the on-diagonal heat kernel $\pb_t^\Omega(x,x)$, viewed as a random variable under the law $\P^*_x$. Then, as $t\to 0$ we have the annealed convergence in law:
    \begin{equation} \label{eq:limit_law}
t \pb_t^\Omega (x,x) \Rightarrow \pb_1^{\cC}(0,0)
    \end{equation}
    where the limiting random variable describes the heat kernel at time one on the $\gamma$-quantum cone. Furthermore, this random variable satisfies:
    \begin{align*}
        \E[\pb_1^{\cC}(0,0)] & = c_\gamma = \frac{1}{\pi(2-\gamma^2/2)}\\
        \E[ (\pb_1^{\cC}(0,0))^n] &< \infty, \text{ for all } n \ge 1.
    \end{align*}
    
\end{theo}

See Section \ref{S:proof_heat} for a precise definition of the $\gamma$-quantum cone.  It is an interesting question as to whether one can compute explicitly the law or the moments of the limiting random variable, $\pb_1^{\cC}(0,0)$. We note that given \eqref{eq:limit_law}, it follows that the Laplace transform of this random variable is given by a somewhat explicit formula, described in Lemma 4.4 of \cite{BW}.




\subsection{The \emph{in-out} decomposition and heuristic connection to KPZ}

Here we describe informally a key idea of our proof or Theorem \ref{T:qc} for the heat trace; we call this the \emph{in-out} decomposition. We then discuss how this decomposition can be used to give a conceptual explanation for how the KPZ scaling emerges naturally. 

\label{SS:heuristics}


\paragraph{Reminders about KPZ.} We start by some reminders on the KPZ equation. Informally, the Euclidean  scaling exponent of a Borel set $A\subset \RR^2$ is such that $ \dim (A) = 2(1- \mathsf{x})$, where $\dim (A)$ is an appropriate notion of dimension for $A$ (i.e., typically, Hausdorff or Minkowski dimension, with respect to Euclidean distance, respectively Lebesgue measure). In other words, $1- \mathsf{x}$ is the \emph{relative dimension} of $A$ compared to that of the entire space. 

The quantum scaling exponent $\Delta$ is (at least informally) defined similarly, by the property that $1-\Delta$ is the dimension of $A$ in the intrinsic Liouville metric, relative to that of the entire space: \begin{equation}\dim_\gamma(A)=D_\gamma(1-\Delta),
\end{equation}
where $D_\gamma$ is the Hausdorff dimension of the whole space (see \cite{GwynnePfeffer} for a rigorous formulation of the above informal identity). On the other hand, the value of $D_\gamma$ is currently unknown (even nonrigorously), except in the case $\gamma=\sqrt{8/3}$ where it is known that we have $D_\gamma=4$. See \cite[Chapter 3]{BP} for a discussion.


\medskip While the proofs of Theorems \ref{T:qc} and \ref{T:heat} are based on a ``bare-hands'' analysis of the heat content and heat kernel, which does not explicitly use the connection to KPZ, we provide here a heuristic explanation for why such an exponent is in fact natural.






\paragraph{The ``in-out'' decomposition.} Fix $x\in \Omega$. The starting point of our argumentation is the following trivial but essential decomposition:
\begin{equation}\label{eq:decomposition}
\mathbf{p}_t^\Omega (x,x) = \mathbf{p}_t^{\CC}(x,x) - \mathbf{p}_t^{\CC \setminus \Omega} (x,x) 
\end{equation}
where the first term indicates that there is no restriction on the trajectory (other than it must start and end at the same point $x$, but is otherwise not constrained to stay in the domain $\Omega$), and the second term corresponds to subtracting all trajectories that leave $\Omega$. Integrating (with respect to the GMC measure $\mu$) over $x \in \Omega$ we obtain
\begin{equation}\label{eq:decomposition2}
\mathbf{H}(t) = \int_\Omega \mathbf{p}_t^{\CC}(x,x) \mu(dx) - \int_\Omega \mathbf{p}_t^{\CC\setminus \Omega} (x,x) \mu(dx).
\end{equation}
This identity is what we refer to in the following as the in-out decomposition. 

\medskip Let us explain how KPZ scaling emerges naturally from such a decomposition as well as basic scaling considerations. When $t$ is small, $\mathbf{p}_t^{\CC}(x,x)$ depends only on the local geometry of the field $h$ near $x$, which by results of \cite{DuplantierMillerSheffield} (see also Theorem 7.11  in \cite{BP}) is close (under the rooted law $\P^*_x$) to that of a $\gamma$-quantum cone rerooted at $x$ (let us call such a cone $\cC_x$ for now); see Section \ref{S:proof_heat} for definitions and reminders. It is thus natural to expect $\pb_t^{\CC}(x,x) \approx \pb_t^{\cC_x}(x,x)$. 

Now, a defining property of quantum cones is their invariance under scaling by a fixed amount of quantum (Liouville) area. 
Since in two dimensions ``volume=time'' (for Brownian motion) it is not hard to verify the following elementary but crucial scaling relation for the heat kernel on the quantum cone: \begin{equation}\label{eq:scaling_qc}
\mathbf{p}^{\Ca}_t(0,0) \overset{d}{=} \frac1t \mathbf{p}_1^{\Ca}(0,0). 
\end{equation}
A proof of this scaling relation will be given in Proposition \ref{P:scale_hk}.
As a consequence, plugging into \eqref{eq:decomposition2} and using ergodicity of the quantum cone, it is not hard to guess that the expectation of the first term on the right hand side of \eqref{eq:decomposition2} scales like $(1/t)\times \E[\mu(\Omega)]\times c_\gamma$ where the constant $c_\gamma$ corresponds to $\EE( \mathbf{p}_1^{\CC} (0,0))$. This first term thus corresponds to the volume scaling in the Weyl law. Comparing with $\cite{BW}$, it must be the case $c_\gamma = 1/[\pi ( 2- \gamma^2/2)]$; this explains Theorem \ref{T:limit_law}.

Given the above, the corrections to the Weyl law are thus captured by the second term in the in-out decomposition \eqref{eq:decomposition2}. Theorem \ref{T:qc} therefore boils down to estimating the ``out term'', namely it suffices to show: 
\begin{equation}\label{eq:conjecture_decomp}
    \E [\int_\Omega \mathbf{p}_t^{\CC\setminus \Omega} (x,x) \mu(dx)] = t^{  \Delta -1+ o(1)}.
\end{equation}
Intuitively, the dominant  contribution to the integral in the left hand side comes from points close to the boundary, since otherwise the probability to leave the domain is negligible.

Let $(\mathbf{Z}_t, t\ge 0) = (\mathbf{Z}_t(x), t\ge 0)$ denote the trajectory of a Liouville Brownian motion starting from $x \in \Omega$, and let $\sigma_\Omega = \inf \{ t \ge 0: \mathbf{Z}_t \notin \Omega\}$ 
be the first hitting time of the boundary of $\Omega$. 
Let $\mathbf{P}_x$ denote the law (conditional on the Gaussian free field, i.e., quenched) of $\mathbf{Z}$, given that $\mathbf{Z}_0 = x$. Let $\mathbf{d}$ denote the $\gamma$-intrinsic distance associated to our quantum cone, as constructed in the series of papers \cite{GM, GM2, DDDF}.

It is widely expected that when $t \to 0$, uniformly over $x \in \mathbb{C}$, 
\begin{equation}\label{eq:distance}
\mathbf{d}(\mathbf{Z}_t, x) \approx t^{1/ D_\gamma}.
\end{equation}
This generalises the well known fact that when $\gamma = 0$, standard Brownian motion is diffusive: that is, if $B$ denotes a standard Brownian motion starting from $x$, then $\|B_t - x\| \approx t^{1/2}$ as $t \to 0$. (In fact, \eqref{eq:distance} can be inferred from the diffusivity of standard Brownian motion and the KPZ equation itself, see Exercise 4.6 in \cite{BP}.) In other words, in terms of intrinsic distance, the behaviour of the Liouville Brownian motion is very regular (whereas in terms of Euclidean distance, its behaviour is wildly dependent on the local geometry of $h$ near $x$ and in particular its thickness, see \cite{Jackson}).  

Let us cover $\partial \Omega$ with balls $\mathbf{B}_i$ (with respect to the intrinsic metric $\mathbf{d}(\cdot, \cdot)$) of radius $r = t^{1/ D_\gamma}$, so $\partial \Omega \subset \cup_{i=1}^N \mathbf{B}_i.$ 
By definition, the number $N$ of balls required for such a covering is approximately 
$$N \approx  r^{ - \dim_\gamma(A)} = t^{ - \dim_\gamma(A) / D_\gamma} =  t^{-(1- \Delta)},$$ 
by definition of the quantum scaling exponent (see Section 3.12.1 in \cite{BP}), and where $\Delta = \Delta(\partial \Omega)$ satisfies \eqref{eq:KPZ}. From \eqref{eq:distance} it is natural to expect that if the starting point $x$ of Liouville Brownian motion is such that $x \notin \cup_{i=1}^N \mathbf{B}_i$ then $\mathbf{P}_x (\sigma_\Omega<t ) \approx 0$ and we should also have $\mathbf{p}_t^{\CC \setminus \Omega} (x,x) \approx 0$. On the other hand, if $x \in \cup_{i=1}^N \mathbf{B}_i$ then $\mathbf{P}_x ( \sigma_\Omega<t) \approx 1$, and thus, one should have (in  view of \eqref{eq:scaling_qc}): 
$$
\mathbf{p}_t^{\CC \setminus \Omega} (x,x) \asymp \mathbf{p}_t^{\CC} (x,x) \approx \frac1t.
$$
We deduce that the left hand side of \eqref{eq:conjecture_decomp} satisfies:
\begin{align*}
    \int_\Omega \mathbf{p}_t^{\CC\setminus \Omega} (x,x) \mu(dx) & \approx \frac1t  \mu(\cup_{i=1}^N \mathbf{B}_i)
\end{align*}
Now, by definition of the   Hausdorff dimension of the whole space, one should also expect that for each ball $\mathbf{B}_i$ ($i=1, \ldots, N$) one has $\mu(\mathbf{B}_i) \approx r^{ D_\gamma}=t$. 
(This is because the Hausdorff dimension $D_\gamma$ corresponds to the volume growth exponent). Therefore, since the balls $\mathbf{B}_i$ form an optimal cover of $\partial \Omega$ and are thus essentially disjoint, we get altogether:
\begin{align*}
    \int_\Omega \mathbf{p}_t^{\CC\setminus \Omega} (x,x) \mu(dx) & \approx N r^{D_\gamma}\frac1t  \\
    & \approx t^{ - (1- \Delta)} \times t \times t^{-1}  = t^{ - (1- \Delta)}.
\end{align*}
This justifies (at least heuristically) the exponent obtained in the right hand side of \eqref{eq:conjecture_decomp} and thus the scaling in Theorem \ref{T:qc}.



\begin{rem}
    The Euclidean analogue of the decomposition \eqref{eq:decomposition} and \eqref{eq:decomposition2} were already noticed by Kac in his seminal paper \cite{Kac1966}. A reasoning not too dissimilar to what is described above was also employed in the book by Berry \cite{Berry_book} to predict (still in the deterministic case) that in $d\ge 2$ dimensions and for a potentially rough bounded domain $\Omega$, the eigenvalue counting functions (with Dirichlet boundary conditions, say) has the following asymptotic behaviour: 
$$
N(\lambda) = c_0(d) \lambda^{d/2} \text{Leb} (\Omega) - c'(d,h) \lambda^{h/2} |\partial \Omega| + o(\lambda^{h/2})
$$
where $h$ denotes the Hausdorff dimension of $\partial \Omega$, $|\partial \Omega|$ its $h$-dimensional Hausdorff measure,  and $c'(d,h)$ a certain constant. This conjecture (and a number of natural variants of it, such as ones in which Hausdorff dimension and measure are replaced by Minkowski dimension and content) have now been disproved in their strictest sense, although they may remain broadly true in a sense that remains somewhat unclear (at least to us) at this stage. We recommend the article of Molchanov and Vainberg \cite{MolchanovVainberg} for a discussion of these results. In a sense, results such as our Theorem \ref{T:qc} can be viewed as providing some justification for these predictions, at least in the LQG setting.
\end{rem}


\subsection{Organization of the paper}
The paper is organised as follows. 

\begin{itemize}
\item 
In the next subsection (Section \ref{SS:prelim}), we start with some preliminaries (recalling the definitions of the whole plane GFF with unit circle average normalisation, and that of Liouville Brownian motion) and introduce notation that will be enforced throughout the paper. 

\item In Section \ref{sec:2} we prove Theorem \ref{T:heat}. 
Its proof splits into a first and second moment estimates, carried out respectively in Sections \ref{sec:2.1} and \ref{sec:2.2}, together with a ``zero-one law'' type argument to bootstrap the probability from not too small to order one.  This last aspect is the content of Section \ref{sec:2.3}. 

\item Section \ref{S:proof_heat} is dedicated to the proof of Theorem \ref{T:qc}. Key to the arguments in this section are estimates on the heat kernel and in particular a control on the moments of $\pb_t^{\CC} (x,x)$. This is carried out in Section \ref{SS:hkcone} (where we explain the scale invariance of the heat kernel of \emph{any} thick $\alpha$-quantum cone) and \ref{SS:moments}. We then use these estimates together with the in-out decomposition. We handle the bulk term (which provides the first order behaviour, corresponding to the Weyl law of \cite{BW}) in Section \ref{SS:bulk}, and along the way complete the proof of Theorem \ref{T:limit_law}.  Finally, the boundary term, which is the main novelty of Theorem \ref{T:qc}, is handled in Section \ref{SS:boundary}.

\item An appendix contains the proof of some technical lemmas, a proof that the heat content is intrinsic to LQG (i.e., satisfies the coordinate change formula, even though it is not expected to be spectrally determined), and the construction of an example of a domain having different Minkowski dimensions on both sides. 

\end{itemize}

\subsection{Acknowledgements}

N.B.  acknowledges the support from the Austrian Science Fund (FWF) grants 10.55776/F1002 on ``Discrete random structures: enumeration and scaling limits" and 10.55776/PAT1878824 on ``Random Conformal Fields''. J.K. was partly supported by FWF grant 10.55776/P36835.
For open access purposes, the
authors have applied a CC BY public copyright license to any accepted manuscript version arising from this submission. 

Both authors acknowledge Tomas Alcalde for some useful discussions at the start of this project. We also thank Jason Miller for some insights concerning the short-term behaviour of the heat kernel, and for asking us about how to interpret our results in Theorem \ref{T:heat} in an LQG intrinsic way: ultimately this led us to a more conceptual understanding of our results. We also thank Juhan Aru for discussions concerning Lemma \ref{lem:tailbound} and \cite{Arunotes}.

\medskip \noindent \textbf{Data statement}: this research contains no relevant data. 

\medskip \noindent \textbf{AI statement}: no generative AI was used for this research, except for proofreading purposes. 

\subsection{Preliminaries and notations}

\label{SS:prelim}

\paragraph{Whole plane Gaussian free field.}
A definition of the Gaussian free field $h$ on the whole plane may be found in \cite[Chapter 6.4]{BP}. Often, one sees this field as being defined modulo a global additive constant, but in this paper we choose a concrete normalisation (i.e., additive constant) by requiring that the average of this 
field on the unit circle is zero. Thus $h$ is a centered Gaussian field with covariance $G(x,y), (x,y) \in \CC^2$, in the sense that for arbitrary $f,g \in H^{-1}_{\mathrm{loc}} (\CC)$ we have 
$$
\E [ (h,f) (h,g) ] = \int_{\CC} f(x) g(y) G(x,y) dx dy,
$$
with a slight abuse of notations when $f,g$ are not pointwise defined. 

An explicit formula for $G$ is given by
\begin{align} \label{eq:Greens function of full plane GFF}
    G(x,y)=\log\frac{1}{\vert x-y\vert}+\log^+\vert x\vert +\log^+\vert y\vert,\qquad x,y\in\CC,
\end{align}
where $\log^+:=\max\{0,\log\}$; see, e.g., Lemma 5.30 in \cite{BP}. Note that for $x,y\in B_1(0)$ this gives $G(x,y)=-\log\vert x-y\vert$. One can check that that the field $h$ is invariant under all Möbius transforms of the extended complex plane \emph{modulo constants} (see, e.g., Corollary 5.7 in \cite{BP} or combine Corollary 6.5 and Theorem 1.57 in \cite{BP}).

\paragraph{Liouville Brownian motion.} Given the random field $h$ defined above, Liouville Brownian motion (\cite{GRV, Ber2015}) is the canonical diffusion in the geometry defined by $h$. Its definition starting from a given point $x$ proceeds in two steps. First, take an independent, standard Brownian motion $(B_t, t \ge 0)$ starting from $B_0 = x$. Then consider the quantum clock process associated to $h$ and $B$, namely:
\begin{align*}
\nu (s) &= \nu^{(h, B)}(s) = \int_0^s e^{\gamma h(B_u) } du\\
& = \lim_{\eps \to 0} \eps^{\gamma^2/2} \int_0^s e^{ \gamma h_\eps (B_u)} du,
\end{align*}
where $h_\eps$ is a smooth regularisation of $h$ at scale $\eps$. (The first line is a formal identity and is given meaning by the second line. The existence of the limit in the second line, and independence of this limit with respect to the choice of regularising convolution kernel, is justified by GMC theory, see \cite{Ber2017} and \cite{Shamov} or \cite[Chapter 3]{BP} for a systematic presentation). By definition, the Liouville Brownian motion $(\mathbf{Z}_t, t \ge 0)$ starting from $x$ is obtained by changing the time-parametrisation of $B$ through the quantum clock process: 
\begin{equation}
\mathbf{Z}_t = B_{\nu^{-1} (t)}; \quad \text{ where } \nu^{-1} (t) = \inf \{ s>0: \nu(s) > t \}.
\end{equation}
It is known that, with this time-parameterisation the process is continuous and has no interval of constancy, however it is highly singular with respect to the original parameterisation of $B$ (in particular, $d\nu$, viewed as a Stieltjès measure associated to the increasing process $\nu$, is mutually singular with respect to Lebesgue measure). Note also that if 
\begin{equation}\label{eq:notation_exit}
\begin{cases}
\tau_\Omega &= \inf \{ t>0: B_t \notin \Omega\}\\
\sigma_\Omega & = \inf \{ t > 0: \mathbf{Z}_t \notin \Omega \}
\end{cases}
\end{equation}
then $
\sigma_\Omega = \nu (\tau_\Omega)
.$ The notations \eqref{eq:notation_exit} will be enforced throughout the article as much as possible; however, in Section \ref{S:proof_heat} it becomes necessary to consider the path space $C( [0, T]; \CC)$, endowed with various measures (which may be the law of ordinary Brownian motion, Liouville Brownian motion, Brownian bridge, Liouville Brownian bridge, etc.). We then use generically $\sigma_\Omega$ and view it as a (continuous) functional defined on path space, which may lead to it being used for other processes than Liouville Brownian motion strictly speaking. To distinguish these laws, we will write $P_x$ for the law of ordinary Brownian motion starting from $x$, and $\mathbf{P}_x$ the law of Liouville Brownian motion starting from $x$. For instance, we have 
$$
\Pb_x( \sigma_\Omega>t) = P_x ( \nu(\tau_\Omega)>t),
$$
and both sides of the identity are actually random variables under the law $\P$ of the whole plane Gaussian free field $h$, say. 

\paragraph{Other notations.} If $x \in \CC$ and $r>0$ we call $B_r(x)$ the Euclidean ball of radius $r$ around $x$. We also let $Q = (\gamma/2) + (2/\gamma)$ be the \textbf{background charge} of LQG. If $x \in \CC$ and $h$ has the law $\P$ of a whole plane GFF with zero unit circle average, we call $\P^*_x$ the corresponding Peyrière or \textbf{rooted law}, which is simply the law of $h (\cdot) + \gamma G(x, \cdot)$, where $h$ has law $\P$ and $G$ is as in \eqref{eq:Greens function of full plane GFF}. 

In order to conveniently keep track of errors that can be safely neglected when we care about exponents (as in Theorems \ref{T:qc} and \ref{T:heat}) we introduce the notion of convergence to zero \textbf{faster than polynomial} (FTP): we say that $f(t) \to 0$ FTP if for any $k>0$, $|f(t) | \le t^{k}$ for $t$ sufficiently small. In that case we also write $f(t) = o_{\mathrm{FTP}}(1)$. 

We will also sometimes need the related (but slightly more subtle) notion of convergence to zero \textbf{slower than polynomial} (STP): say that $f(t) \to 0$ STP if $f(t) \to 0$ as $t\to 0$, but $f(t) t^{-\eta} \to \infty$ for any $\eta>0$. In that case we also write $f(t) = o_{\mathrm{STP}}(1)$.

\section{Heat content asymptotics} \label{sec:2}
This section is devoted to a proof of Theorem \ref{T:heat}. We start by giving an overview of the proof, which is divided into three parts.

Overall the proof follows a familiar first and truncated second moment argument. Although the KPZ heuristics discussed in Section \ref{SS:heuristics} strongly suggests breaking up the integral as a function of the intrinsic distance to the boundary, it turns out to be technically easier to analyse things if we break it up in terms of Euclidean distance to the boundary. We then estimate the probability of leaving the domain by time $t>0$ by taking into account the thickness of the point $z$, and optimise over this thickness. It is implicit in the proof of the KPZ relation that the dominant contribution to the Liouville area of the neighbourhood of a fractal set $A$ (measured in an essentially intrinsic way) comes from points with atypical thickness, even with respect to ordinary LQG where the points are typically $\gamma$-thick. Here we find a similar phenomenon: for instance, if $\Omega$ is \emph{smooth} the dominant contribution comes from points of thickness $\alpha = Q - \sqrt{Q^2 - 2} < \gamma$. (More generally the value of $\alpha_0$ is specified in \eqref{eq:alpha0}.) 

\medskip We will first consider the case where $h$ is a full plane GFF normalized to have zero average on the unit circle, and explain at the end how to deduce the result for quantum cones. We start in Section \ref{sec:2.1} to establish the first-moment asymptotics
$$
\EE[\Kb_\Omega(t)] = t^{\Delta+o(1)} \quad \text{as } t \to 0,
$$
see Proposition \ref{prop:first moment Heat content fractal boundary}. We apply Girsanov’s theorem to swap the expectation with the integral against the Liouville measure which adds a $\gamma \log$ singularity to the field. The problem is then reduced to understanding how the expected exit time from the domain  scales with $t$, as a function of the starting point. A scaling argument (Lemma \ref{lem:scaling without}) shows that, starting from a Liouville-typical (i.e., $\gamma$-thick point), an LBM travels a Euclidean distance of order $t^{2-\gamma^2/2}$ in time $t$. 

However, as mentioned above, the dominant contribution comes from atypical $\alpha$-thick points; more precisely points with circle average $h_r(z) \approx \alpha \log(1/r)$, where $r = d(z,\partial\Omega)$ and $\alpha < \gamma$ is the result of an optimisation problem  (which turns out to match the exponent $\alpha_0 = \gamma(1- \Delta)$ arising in the proof of the KPZ relation, see \cite{DS2011}). 
Lemma \ref{cor:scaling} gives the relevant time scale under the assumption that the point has a given thickness. Combined with an estimate on the size of the neighbourhood of the boundary (at a given Euclidean distance) using Assumption \ref{Assumption 1 non local} and the probability that a given point is of a given thickness we are led to an explicit optimisation problem whose solution gives us the expectation of the heat content.


A key input in this chapter is to obtain good concentration bounds for the exit time of the unit disc by a Liouville Brownian motion from a given open set, say the unit disc $\DD$. We obtain the following log-normal tail:
\begin{equation}\label{eq:lognormal}
\PP^*_0\otimes {P}_0[\nu(\tau_{\DD}) \le \epsilon] \le c e^{-c' (\log \epsilon)^2},
\end{equation}
for some constants $c, c' > 0$, and where $\PP^*_0$ indicates that the field may have an additional logarithmic singularity at the origin, and $\PP^*_0\otimes \mathbf{P}_0$ indicates we take the expectation both with respect to Brownian motion starting at the origin and GFF (i.e., it is an \emph{annealed} probability). In particular the right hand side of 
\eqref{eq:lognormal} converges to zero FTP as $\epsilon \to 0$. 
This is the content of Lemma \ref{lem:tailbound}, whose proof is technical and inspired by Duplantier–Sheffield \cite{DS2011} (or, more precisely, by a version of this argument discussed in \cite{Arunotes}), and which may be of independent interest.

Section \ref{sec:2.2} concerns second-moment estimates for the heat content. Jensen’s inequality provides a lower bound on the second moment, so we only need an upper bound on the second moment. One would hope for
$$
\EE[\Kb_\Omega(t)^2] \le t^{2\Delta+o(1)} \quad \text{as } t \to 0,
$$
but this fails outside of the $L^2$-phase (i.e., when $\gamma \in [\sqrt{2},2]$), since already the second moment of the Liouville measure diverges. We need to identify a good  event which carries most of the mass in an $L^1$ sense, and prevents the second moment from blowing up. The problem is nonstandard, owing to the fact that on the one hand, the points we consider are $\gamma$-thick at an infinitesimal level (an unavoidable consequence of the use of Girsanov's lemma) but are also conditioned to have a different thickness, namely $\alpha_0$, at scale $r$.

Finally, in Section \ref{sec:2.3} we invoke the Paley–Zygmund inequality to obtain convergence in probability. Because our first- and second-moment estimates carry $o(1)$ errors, they initially yield a statement only on an event whose probability may tend to $0$. To amplify this to probability tending to $1$, we use a $0$–$1$-type argument: we partition the boundary into many small pieces, establish the local first- and second-moment bounds on each piece, and exploit their approximate independence. For this step we also need Assumption \ref{Assumption 2}, in order to control the $o(1)$ terms uniformly over all small pieces. This completes the proof of Theorem \ref{T:heat}.

\subsection{First moment computation} \label{sec:2.1}
Let us now start with the first moment estimate of the heat content. We assume throughout Section \ref{sec:2.1} that $h$ is a full plane GFF normalized to have $0$ average on the unit circle and, without loss of generality, $\Omega\subseteq B_{1/2}(0)$. In particular, for any $x,y \in \Omega$, $G(x,y) = - \log | x- y |$. 

\begin{prop} \label{prop:first moment Heat content fractal boundary}
    Let 
 $\Omega$ satisfy Assumption \ref{As:innerMinkowski}, i.e., $\Omega$ is a  domain whose boundary has inner Euclidean scaling exponent $x\in[0,1/2]$. Then,
    \begin{align*}
\EE[\Kb_\Omega(t)]=t^{\Delta+o(1)} \text{ as }t\to0,
    \end{align*}
    where $\Delta$ is as in \eqref{eq:KPZ}. 
\end{prop}

First, using Girsanov's lemma (see Lemma 2.5 in \cite{BP}) we can rewrite the expected heat content as 
\begin{align*}
    \E[\Kb_\Omega(t)] & = \E\Big[\int_\Omega\Pb_z[\sigma_\Omega\leq t]\mu_\gamma(dz)\Big]
    = \int_\Omega \P^*_z\otimes {P}_z[\nu(\tau_\Omega)\leq t]dz,
\end{align*}
where, under $\P^*_z$ the field can be written as $h=h'+\gamma G(\cdot,z)$, with $h'$ being a whole plane GFF normalised to have zero unit circle average, and $G$ being as in \eqref{eq:Greens function of full plane GFF}.


We need to understand the quantum clock $\nu$ under a GFF with a $\gamma-$log singularity. For this, we want to perform a rescaling argument, to see how far a LBM goes in $t$ units of time. Choose $z\in \Omega$ and suppose $B_{2r} (z) \subset \Omega$. We apply the domain Markov property to $h'$ in this ball, to get by \eqref{eq:Greens function of full plane GFF} that under $\PP_z^*$,
$$
h = h^{2r} + u^{2r} + \gamma \frac{1}{\vert\cdot-z\vert},
$$
where $u^{2r}$ is a.s. harmonic in $B_{2r}(z)$, $h^{2r}$ has the law of a Gaussian free field in $B_{2r} (z) $ with Dirichlet boundary conditions, and these are independent. Let $\tau_r$ denote the exit time of the planar Brownian motion $B$ from $B_r(z)$, and let $\sigma_r = \nu (\tau_r)$ the associated exit time of the ball of radius $r$ around $z$ for the corresponding Liouville Brownian motion. We prove the following bounds: 

\begin{lem}[Scaling the exit time of a ball] \label{lem:scaling without} 
Suppose $B_{2r}(z) \subset \Omega$. Then, we have the almost sure inequalities: \begin{align} \label{eq:scaling without}
 r^{2- \gamma^2/2} e^{\gamma \underline{u}_r} \hat \sigma_1 \le     \sigma_r \le  r^{2-\gamma^2/2}e^{\gamma \bar u_r} \hat \sigma_1,
\end{align}
where $\bar u_r = \sup_{B_{r}(z) } u^{2r}$ (resp. $\underline{u}_r = \inf_{B_{r}(z)}  u^{2r}$), and under the law $\P^*_z$, $\hat \sigma_1$ is independent of $h|_{B_{2r}(z)^c}$ and has the same law as that of the exit time $\sigma_1$ from a unit ball by a Liouville Brownian motion starting at its center, associated to a Dirichlet Gaussian free field on $B_2(0)$ plus a $\gamma-$log singularity at the origin.
\end{lem}
\begin{proof}
    Denote $\hat B_t = (1/r) B_{tr^2}$, then $(\hat B_t)_{t\ge 0}$ has the law of a Brownian motion starting from $z$ which is independent of $h$. Let $\hat h$ denote the image of $h^{2r} $ under the maps $w \mapsto (w-z)/r$ (which by conformal invariance has the law of a GFF with Dirichlet boundary conditions in $B_2(0)$). Note that, by  \eqref{eq:Greens function of full plane GFF}, we have for $w\in B_{2r}(z)$
$
G(z,w) = - \log |z-w|.
$
Hence
    \begin{align*}
    \nu(\tau_r) &= \lim_{\epsilon\to0}\ \int_0^{\tau_r}e^{\gamma h_{\epsilon}(B_s)}\epsilon^{\gamma^2/2}ds\\
    &= \lim_{r\epsilon\to 0}\ \int_0^{\tau_{r}}e^{\gamma (h^{2r}_{r\epsilon}(B_s)+u^{2r}_{r\epsilon}(B_s))}(r\epsilon)^{\gamma^2/2} e^{\gamma^2 G(0, B_s)}ds\\
    & \le
\limsup_{\epsilon\rightarrow0} \int_0^{r^2\hat \tau_1}e^{\gamma (h^{2r}_{r\epsilon}(r\hat B_{s/r^2})+ \bar u_r)}(r\epsilon)^{\gamma^2/2}\vert rB_{s/r^2}\vert^{-\gamma^2} ds\\
    & \le \limsup_{\epsilon\rightarrow0} e^{\gamma \bar u_r} \int_0^{\tau_1}r^2e^{\gamma h^{2r}_{r\epsilon}(r \hat B_t)}(r\epsilon^{\gamma^2/2})\vert r \hat B_t\vert^{-\gamma^2}dt\\
    & = e^{\gamma \bar u_r} \ r^{2-\gamma^2/2} \limsup_{\epsilon\rightarrow0} \int_0^{\hat \tau_1}e^{\gamma \hat{h}_{\epsilon}( \hat B_{t})}\epsilon^{\gamma^2/2}\vert \hat B_{t}\vert^{-\gamma^2}dt\\
    &= e^{\gamma \overline{u}_r}r^{2-\gamma^2/2}\hat \nu(\hat \tau_1),
\end{align*}
where $ \hat \nu(\cdot)$ denotes the clock process associated to $\hat B$ in the environment $\hat h$. This gives the upper bound, the lower bound follows by a similar argument. 
\end{proof}

From this computation, it is natural to expect that, starting from a $\gamma$-typical point, a Liouville Brownian motion typically travels to a Euclidean distance of order $t^{2-\gamma^2/2}$ in time $t$. Note also that this computation tells us a bit more since we have isolated the dependence on the overall value of the field on $B_{2r}(z)$ through the values $\bar u_r$ and $\underline{u}_r$, thereby allowing us to condition on the value $h_{2r}(z)$. 

We will need to make this precise by providing quantitative bounds. This boils down to proving  suitable concentration bounds for both $\underline{u}_r$ and $\hat\sigma_1$. It is important to us that the control is ``better than polynomial" in order to not lose in the computation of the exponent. 
The first lemma below (Lemma \ref{lem:tailbound}) gives a lognormal estimate for the small ball probability of $\hat \sigma_1$, which will be sufficient for us. 
Concerning $\underline{u}_r$ and $\bar{u}_r$ we will need to condition on the value of $h_{2r}(z)$. 
Concretely, we further decompose
\begin{align*}
\underline{u}_r=u(z)+\bigl(\underline{u}_r-u(z)\bigr).
\end{align*}

We have the following tail bounds for $\hat\sigma_1$ and $\underline{u}_r-u(z)$, which we will prove later, as they are quite technical, see Section \ref{SS:intermediary}.

\begin{lem} \label{lem:tailbound}
    Let $\alpha \ge 0$, and let $h$ have the law under $\P$ of  $h=h^{2\DD} + \alpha \log 1/|\cdot|$ where $h^{2\DD}$ is a Dirichlet GFF on $2\DD$. Then, for some fixed constants $c_1,c_2\in(0,\infty)$ and any $\epsilon>0$ small enough, we have that 
    \begin{align}
        \PP\otimes P_0[\nu(\tau_{\DD})\leq \epsilon]\leq c_1e^{-c_2(\log\epsilon)^2}.
    \end{align}
\end{lem}

The next result is a small extension of the classical Markov property of the GFF (in which we point out that the harmonic function is furthermore independent of its overall circle average), together with 
a variation on classical concentration results for Gaussian processes. 

\begin{lem} \label{lem:decaymean}
    Let $h$ be a whole plane GFF normalized to have $0$ average on the unit circle. Fix a constant $N>1$. Let $z_0\in \Omega $ and $r\gre0$ such that $\bar B_{Nr}(z_0)\subseteq B_1(0)$. We can write (using the Markov property)$$h(z) =h'(z)+u(z_0) + \big( u(z) - u(z_0)\big),$$where $h'$ is a Dirichlet GFF on $B_{Nr}(z_0)$, and $u$ is harmonic on $B_{Nr}(z_0)$ (so $u(z_0)$ coincides with the circle average $h_{Nr}(z_0)$). Then, we have that $h', u(0)$  and $u(\cdot)-u(z_0)$ are independent. Furthermore, fix $\lambda_0>0$. Then there are constants $c,c'>0$, $c$ only depending on $N$ and $\lambda_0$ and $c'$ depending only on $N$ such that
    \begin{align} \label{eq:decay mean}
        \PP[\sup_{z\in B_r(z_0)}\vert u(z)-u(z_0)\vert\ge\lambda]\le ce^{-c'\lambda^2},\qquad\lambda\ge\lambda_0.
    \end{align}
    Furthermore, the constant $c' = c'(N)$ can be chosen arbitrarily large by choosing $N$ to be large enough. 
\end{lem}

We emphasise that aspects of this lemma remain true for a GFF with Dirichlet or Neumann boundary conditions on an arbitrary domain $U$ such that   $B_{Nr}(z_0) \subset U$. In this case $u$ and $h$ are of course still independent, while $u(z_0)$ and $u(z)-u(z_0)$ are not in general. One can easily check that the proof of \eqref{eq:decay mean} remains valid for this case. 


When analyzing the exit time of a Liouville Brownian motion started at $z$ from the ball $B_r(z)$, we condition on the value of $u(z)$ and rescale time accordingly. This changes the effective exponent. The next corollary identifies the resulting scaling exponent. In particular, the thickness of $z$ (more precisely, the circle average at scale $r$) alters the typical Euclidean distance a Liouville Brownian motion started at $z$ can cover, as follows:

\begin{cor}[Scaling with conditioning] \label{cor:scaling}
    Recall the notations of Lemma \ref{lem:scaling without}.  Suppose $z \in \Omega$ and $B_{2r}(z) \subset \D$. 
            For $\alpha \in \R$, consider the event 
    \begin{align*}
      \cE_{z,r}(\alpha) := \{  h_{2r}(z)=\alpha\log1/2r\}.
    \end{align*}
Then a.s. under $\P^*_z ( \cdot | \cE_{z,r} (\alpha))$, we have by Lemma \ref{lem:scaling without}
    \begin{align} \label{eq:scaling with conditioning}
         r^{2+\gamma^2/2-\gamma\alpha} e^{\gamma(\underline{u}_r-u(z))} \hat \sigma_1 \le     \sigma_r \le  r^{2+\gamma^2/2-\gamma\alpha}e^{\gamma (\bar u_r-u(z))} \hat \sigma_1,
    \end{align}
    where $\hat \sigma_1$ is as in Lemma \ref{lem:scaling without}, and in particular, $\hat\sigma_1$ and $\underline{u}_r-u(z)$ are independent of $\cE_{z,r}(\alpha)$.
    As a consequence we obtain the following bound: 
    let
\begin{align} \label{eq:definition of kappa}
    \kappa(\alpha):=\frac{1}{2+\gamma^2/2-\gamma\alpha}.
\end{align}

If $\alpha < Q$ (so that $\kappa(\alpha))>0$), then for any $\delta>0$, 
    \begin{align}\label{C:exittime_alpha<Q}
    \dps\lim_{t\rightarrow0}\ \PP^*_z\otimes P_z[\nu(\tau_{r}) \leq t |\cE_{z,r} (\alpha)]=\begin{cases}0\qquad&\textnormal{if } r> t^{\kappa(\alpha)(1- \delta)}\\ 
    1 &\textnormal{if } r < t^{\kappa(\alpha)(1+ \delta)}
    \end{cases}\qquad\textnormal{FTP},
\end{align}
where in both of these situations, the convergence to the limit holds FTP and, crucially, \emph{uniformly} in $z$ and $\alpha<Q$. On the other hand, if $\alpha \ge Q$ and $r\le1$, then
\begin{align*}
    \PP^*_z\otimes P_z[\nu(\tau_{r}) \leq t |\cE_{z,r} (\alpha)] \to 0,\qquad\textnormal{FTP}.
\end{align*}
\end{cor}

\begin{proof} Under $\PP^*_z$, we have $h=h'+\gamma \log\vert z-\cdot\vert$, so conditioned on $\cE_{r,z}(\alpha)$, we get by \eqref{eq:Greens function of full plane GFF}
\begin{align*}
    h'_{2r}(z)=(\alpha-\gamma)\log1/2r.
\end{align*}
Plugging that into \eqref{eq:scaling without}, we obtain \eqref{eq:scaling with conditioning}.

Now suppose that $\delta>0$ and $\alpha <Q$ so that $\kappa(\alpha)>0$. Then, we get
\begin{align*}
\PP^*_z\otimes P_z[\nu(\tau_{r}) \leq t| \cE_{z,r} (\alpha)]&\le\PP^*_z\otimes P_z[t^{\kappa(\alpha)(1-\delta)(2+\gamma^2/2-\gamma\alpha)} e^{\gamma(\underline{u}_r-u(z))} \hat \sigma_1\le t | \cE_{z,r} (\alpha)]\\
&\le\PP^*_z\otimes P_z[t^{(1-\delta)} e^{\gamma(\underline{u}_r-u(z))} \hat \sigma_1\le t | \cE_{z,r} (\alpha)]\\
&\le\PP^*_z\otimes P_z[e^{\gamma(\underline{u}_r-u(z))} \hat \sigma_1\le t^{\delta} | \cE_{z,r} (\alpha)]\\
&\le\PP^*_z\otimes P_z[\hat \sigma_1\le t^{\delta/2}| \cE_{z,r} (\alpha)]+\PP^*_z[e^{\gamma(\underline{u}_r-u(z))}\le t^{\delta/2} | \cE_{z,r} (\alpha)].
\end{align*}
Now, $\hat \sigma_1$ is independent of $\cE_{z,r} ( \alpha)$ so the first term vanishes FTP by Lemma \ref{lem:tailbound}, uniformly in $\alpha<Q$. The exact same is true for the second term, as by Lemma \ref{lem:decaymean} $u - u(z)$ and $u(z)$ are independent and the tail can be bounded by \eqref{eq:decay mean}. The second case $r<t^{\kappa(\alpha)(1+\delta)}$ follows in an analogous manner. 

Now suppose $\alpha \ge Q$ so $\kappa(\alpha ) \le 0$, and let $\delta>0$ and $r\le 1$. Then $r^{2+\gamma^2/2-\gamma\alpha}\ge1$ so we have
\begin{align*}
    \PP^*_z\otimes P_z[\nu(\tau_r)\le t\vert\cE_{r,z}(\alpha)]&\le\PP^*_z\otimes P_z[ e^{\gamma(\underline{u}_r-u(z))} \hat \sigma_1\le t | \cE_{z,r} (\alpha)]\\
    &\le\PP^*_z\otimes P_z[\hat \sigma_1\le t^{1/2}| \cE_{z,r} (\alpha)]+\PP^*_z[e^{\gamma(\underline{u}_r-u(z))}\le t^{1/2} | \cE_{z,r} (\alpha)].
\end{align*}
Again, we conclude by Lemma \ref{lem:tailbound} and Lemma \ref{lem:decaymean} that the right hand side tends to zero FTP in $t$,  uniform in $\alpha\ge Q$.
\end{proof}

We now explain the proof of Proposition \ref{prop:first moment Heat content fractal boundary} assuming Lemmas \ref{lem:tailbound} and \ref{lem:decaymean}.
\begin{proof}[Proof of Proposition \ref{prop:first moment Heat content fractal boundary}]
    As already indicated at the beginning of this section, we start applying Girsanov's theorem, to exchange the expectation and the integral. For every point $z\in \Omega$, under $\PP^*_z$ our field can be written as $h=h'+\gamma G ( z, \cdot)$ where $h'$ has again the law $\P$ (see \eqref{eq:Greens function of full plane GFF}).  Note that this is equal to $h(x) = h'(x) - \gamma \log\vert x-z\vert$, if $x \in B_1(0)$. 
    We will show that $\gamma$-thick points do not contribute uniformly to the integral (cf. Lemma \ref{cor:scaling}). The contribution of a point is governed by the circle average of $h'$ at the scale of its distance to the boundary. Writing $r=d(z,\partial\Omega)$, we therefore condition on the value of $h_{2r}(z)$ and rewrite the integral accordingly. 

    We have to show that
    \begin{align*}    I=\int_{\Omega}\P^*_z\otimes P_z[\nu(\tau_\Omega)<t]dz=t^{\Delta+o(1)}.
    \end{align*}
We start with the upper bound. We condition with respect to the value of $h_r(z)/ \log (1/r)$, a value which we will denote by $\alpha$ (so we will integrate with respect to $\alpha$). As 
\begin{align} \label{eq:compare h to h'}
\frac{h_r(z) }{\log (1/r)} = \gamma + \frac{h'_r (z)}{\log (1/r)}
\end{align}
and as $h'_r(z)/ \log (1/r)$ is a centered Gaussian random variable under $\P^*_z$ and variance $1/\log (1/r)$, we deduce (after lower-bounding $\tau_\Omega$ crudely by $\tau_r$) the following: 
    \begin{align}\label{eq:boundI}
I&\le\int_{\alpha\in \R}\int_{z\in \Omega}\P^*_z\otimes P_z\Bigl[\nu(\tau_r)<t\Big\vert h_{2r}(z)=\alpha\log1/r\Bigr]f_{r,z}(\alpha)dz d\alpha ;\\
 & \text{ with } f_{r,z}(\alpha) = r^{(\alpha-\gamma)^2/2+o(1)}; \text{ and } \int_\alpha f_{r,z} (\alpha) d\alpha = 1 \text{ for all } z \in \Omega\nonumber
    \end{align}

    As we will see, the dominant contribution (even for a given $\alpha \in \R$) to the second integral on the right hand side turns out to come from a thin ``strip'' of points at a certain distance from the boundary of $\Omega$, namely points whose distance to the boundary is of order $t^{\kappa (\alpha)}$, where $\kappa(\alpha)$ is given by \eqref{eq:definition of kappa}.

    Let us first comment on the case $\alpha\ge Q$ (i.e. $\kappa(\alpha)\le 0$), and show that the portion of the first integral coming from such values of $\alpha$ can be neglected. By Corollary \ref{cor:scaling}, the probability in \eqref{eq:boundI} can be bounded by $Ct^m$ for any $m>0$ (in particular any $m>\Delta$), for some $C=C(m,\gamma)>0$. Thus, we get
    \begin{align*}
        \int_Q^\infty&\int_{z\in \Omega}\P^*_z\otimes P_z\Bigl[\nu(\tau_r)<t\Big\vert h_{2r}(z)=\alpha\log1/r\Bigr]f_{r,z}(\alpha) d\alpha\\
        &\le\int_Q^\infty\int_{z\in \Omega}Ct^m f_{r,z}(\alpha)dz d\alpha\\
        &\le Ct^m \int_{z\in \Omega}\int_Q^\infty f_{r,z}(\alpha)d\alpha dz\\
        &\le C\vert\Omega\vert t^m
    \end{align*}
    and obtain that this case contributes less than $t^{\Delta+o(1)}$ to the integral.
    
    For $\alpha< Q$, we pick a small parameter $\delta>0$, let $\rho = t^{\kappa(\alpha)(1-\delta)}$ and decompose the domain $\Omega$ by setting 
\begin{align*}
    \Omega_{\alpha, n}^\delta&: = \{ z \in \Omega: \rho^{n+1} < d ( z, \partial \Omega) \le \rho^{n}\},\qquad n\ge1\\
    \Omega^\delta_{\alpha,\mathrm{in}}&:=\{z\in\Omega:d(z,\partial\Omega)>\rho\}
\end{align*}
so that 
\[
\Omega=\Omega_{\alpha,\mathrm{in}}^\delta\cup\bigcup_{n\ge0}\Omega_{\alpha,n}^\delta.
\]
    Let us first show that $\Omega_{\alpha,\mathrm{in}}^\delta$ only contributes a negligible amount to the integral. Note that for $z\in \Omega_{\alpha,\mathrm{in}}^\delta$, 
    we have $r = d(z, \partial \Omega) \ge \rho = t^{\kappa (\alpha) ( 1- \delta)}$. So by \eqref{C:exittime_alpha<Q} in Corollary \ref{cor:scaling}, we see that for any $m>0$, there is a constant $C = C(m, \delta, \gamma)$ (but independent of $\alpha<Q$) such that $\P^*_z\otimes P_z[\nu(\tau_\Omega)<t] \le C t^m$. Hence
    \begin{align} \label{eq: upper estimate first moment of the heat content}
 \int_{\alpha< Q} \int_{\Omega_{\alpha,\mathrm{in}}^\delta} \P^*_z\otimes P_z[\nu(\tau_r)<t] f_{r,z} (\alpha) dz d\alpha 
\le & C t^m \int_{\Omega_{\alpha,\mathrm{in}}^\delta} \int_{\alpha<Q} f_{r,z} (\alpha) d\alpha  \le C t^m |\Omega|.
    \end{align}
Since $m>0$ is arbitrary this shows the left hand side is in particular less than $t^{\Delta + o(1)}$ by picking $m>\Delta$. 
    
Now to the case $\alpha<Q$ and $z\in\Omega^\delta_{\alpha,n}$ for some $n\ge1$. It will turn out to be convenient to first show that the values of $\alpha$ outside of a compact set (say $\alpha< -2$) contribute a negligible amount overall to the integral. To see this, we bound crudely $\P^*_z\otimes P_z[\nu(\tau_r)<t] \le 1$, and recall that 
$$
\int_{\alpha< -2} f_{r, z} ( \alpha) d\alpha \le C \exp \Bigl( - \frac12 (\gamma+2)^2 \log (1/r) \Bigr)  
$$
using standard results for the tail of a standard Gaussian random variable (where $C$ is a universal constant). Furthermore, if $\alpha< -2$, $$\bigcup_{n\ge0} \Omega_{\alpha,n}^\delta \subseteq \{ z \in \Omega: d( z, \partial \Omega) \le t^{\eta}\}, \text{ with } \eta = \kappa(-2) /2$$
so that uniformly for $z \in \cup_{n\ge1} \Omega_{\alpha,n}^\delta$, we have $r \le t^{\eta}$. Therefore, using Fubini's theorem, 
$$
\int_{\alpha< -2} \int_{z\in  \cup_{n\ge0} \Omega_{\alpha,n}^\delta } f_{r,z} (\alpha) d\alpha \le |\Omega| \exp \Bigl( - \frac12 (\gamma+2)^2 \eta \log (1/t)\Bigr).
$$
One can easily check that $\eta ( \gamma+2)^2/2 > \Delta$, so that the integral on the left hand side is $t^{\Delta + o(1)}$.

We will thus henceforth be able to assume without loss of generality that $-2 \le \alpha <Q$. Let us call $I_n(\alpha)$ the portion of the integral in the right hand side of \eqref{eq:boundI} on $\Omega_{\alpha, n}^\delta$. 
By the assumption on the inner Minkowski dimension of the boundary \ref{assumption 1}, we have
\begin{align*} 
    \vert\Omega^\delta_{\alpha,n}\vert=( \rho^n)^{2\mathsf{x} (1+ o(1))}
\end{align*}
Since $\kappa(\alpha) \ge \kappa (-2)$ uniformly over $-2 \le \alpha \le Q$, and since $n\ge 1$ we deduce that 
\begin{align} 
    \vert\Omega^\delta_{\alpha,n}\vert\le t^{2\mathsf{x}n\kappa(\alpha)(1-\delta)+o(1)},
\end{align}
where now $o(1)$ is uniform in $\alpha\in[-2,Q)$ and $0<\delta<1/2$ and only depends on $t$. Again, on $\Omega^\delta_{\alpha,n}$, we can crudely bound $\P^*_z\otimes P_z\Bigl[\nu(\tau_r)<t\Big\vert h_r(z)=\alpha\log1/r\Bigr]\le 1$, and we estimate $f_{r,z}(\alpha)$ by
$$r^{(\alpha-\gamma)^2/2} \le \rho^{(n+1)(\alpha- \gamma)^2/2} = t^{(n+1)\kappa(\alpha)(1-\delta)(\alpha-\gamma)^2/2}.$$ Thus, we get the following upper bound for $I_n(\alpha)$ when $-2\le \alpha<Q$:
$$
I_n(\alpha) \le t^{2\kappa(\alpha)(1-\delta)(n+1)\mathsf{x}}t^{(n+1)\kappa(\alpha)(1-\delta)(\alpha-\gamma)^2/2+o(1)},
$$
Where the $o(1)$-term is uniform in $z,\alpha,\delta$ and $n$.
Therefore,
    \begin{align*}
        \int_{\alpha<Q}\sum_{n\ge0} I_n(\alpha) d\alpha&\le t^{\Delta+o(1)}+\int_{-2}^Q\sum_{n\ge0}\int_{\Omega_{\alpha,n}^\delta}t^{(n+1)\kappa(\alpha)(1-\delta)(\alpha-\gamma)^2/2+o(1)}dzd\alpha\\
        &\le t^{\Delta+o(1)}+\int_{-2}^Q\sum_{n\ge0}t^{2\kappa(\alpha)(1-\delta)(n+1)\mathsf{x}}t^{(n+1)\kappa(\alpha)(1-\delta)(\alpha-\gamma)^2/2+o(1)}d\alpha\\
        &\le t^{\Delta+o(1)}+\int_{-2}^Q 
      t^{o(1)}\sum_{n\ge 0} \left(t^{(1-\delta)  f(\alpha)}\right)^n d\alpha,
    \end{align*}
    where 
    $$
f(\alpha )= \kappa(\alpha)(2\mathsf{x}+(\alpha-\gamma)^2/2).
    $$
    Some elementary but tedious calculations show that the minimum of $\alpha\mapsto f(\alpha) $ 
    is attained at \begin{equation}\label{eq:alpha0}
        \alpha_0=Q-\sqrt{Q^2-4(1-\mathsf{x})}.
     \end{equation}   
        The corresponding minimum is the value
    \begin{align}
        f(\alpha_0) =  \frac{2\mathsf{x}+(\alpha_0-\gamma)^2/2}{2+\gamma^2-\alpha_0\gamma}=\Delta,
        \label{eq:Delta_alt}
    \end{align}
    where one can check that $\Delta$ agrees with \eqref{eq:KPZ}.

    As a consequence, the sum in the integral above can be bounded (uniformly over $-2 \le \alpha<Q$) by
    \begin{align*}
        \sum_{n\ge 1}\Bigl(t^{(1-\delta) \Delta}\Bigr)^n \le \frac{t^{(1-\delta)\Delta}}{1- t^{\Delta/2}}.
    \end{align*}
    if we assume without loss of generality that $\delta <1/2$ and $0<t<1$. 
Integrating, we thus obtain 
$$
I\le C t^{(1-\delta) \Delta + o(1)} 
$$
with $C = 2(Q+2)$ and $t$ sufficiently small. Since $0<\delta<1/2$ is arbitrarily small, this shows $I \le t^{\Delta+o(1)}$, as desired. 
%

    The lower bound is easier. We want to restrict to an event and a region which together carry the main contribution. Let $\alpha_0$ be as in \eqref{eq:alpha0}
    be the optimal value, $\kappa(\alpha_0)$ given by \eqref{eq:definition of kappa} (note that $\alpha_0 < Q$, so $\kappa(\alpha_0)>0$) and let $\delta>0$. Under the rooted law $\PP^*_z$ we have by \eqref{eq:Greens function of full plane GFF}, $h=h'-\gamma\log\vert\cdot-z\vert$, where $h'$ is again a full plane GFF normalized to have zero average on the unit circle. Let 
    \begin{align*}
        \cE_{z,r}(\alpha_0,\delta)&:=\{h_{2r}(z)\in[\alpha_0\log1/r,(\alpha_0+\delta)\log1/r]\}\\
        \Omega^*_{\delta,t}&:=\Bigl\{z\in\Omega\Big\vert t^{\kappa(\alpha_0)(1+2\delta)}\le d(z, \partial^*\Omega)\le t^{\kappa(\alpha_0)(1+\delta)}\Bigr\}.
    \end{align*}
    (We recall that $ \partial^*\Omega$ denotes a fixed connected component of positive diameter of the boundary).  Let also $z\in\Omega^*_{\delta,t}$. Let $n\ge1$. By scaling (for ordinary planar Brownian motion) and Beurling's estimate, since $\partial^* \Omega$ is connected, we have for $t>0$ small enough
    \begin{align*}
         P_z[\tau_\Omega<\tau_{r\log1/t}]\ge\frac{n-1}{n}.
    \end{align*}
    On $\Omega^*_{\delta,t}$ we have $r\le t^{\kappa(\alpha_0)(1+\delta)}$. Thus, we get by Corollary \ref{cor:scaling}
    \begin{align*}
        \PP^*_z\otimes P_z[\nu(\tau_{r\log1/t})<t\vert\cE_{z,r}(\alpha_0,\delta)]\ge\frac{n-1}{n},
    \end{align*}
    for $t>0$ small enough. Combining these estimates, we get
    \begin{align*}
        \PP^*_z\otimes P_z[\nu(\tau_\Omega)<t\vert \cE_{z,r}(\alpha_0,\delta)]&\ge\PP^*_z\otimes P_z[\nu(\tau_\Omega)<\nu(\tau_{Cr})<t\vert\cE_{z,r}(\alpha_0,\delta)]\\
        &=\PP^*_z\otimes P_z[\nu(\tau_\Omega)<\nu(\tau_{Cr})\vert\cE_{z,r}(\alpha_0,\delta)]\\
        &\qquad-\PP^*_z\otimes P_z[\nu(\tau_\Omega)<\nu(\tau_{Cr}),\nu(\tau_{Cr})\ge t\vert\cE_{z,r}(\alpha_0,\delta)]\\
        &\ge P_z[\tau_\Omega<\tau_{Cr}]-\PP^*_z\otimes P_z[\nu(\tau_{Cr})\ge t\vert\cE_{z,r}(\alpha_0,\delta)]\\
        &\ge1-\frac{1}{n}-\frac{1}{n}=\frac{n-2}{n},
    \end{align*}
   so the probability tends to $1$ uniformly in $z\in\Omega^*_{\delta,t}$, when $t\to0$. Next, we estimate the size of $\Omega^*_{\delta,t}$. By 
   \eqref{eq:assumption_lowerbound_local} and Assumption \ref{As:innerMinkowski}, we have
    \begin{align*}
        \vert\Omega^*_{\delta,t}\vert \ge t^{2\mathsf{x}\kappa(\alpha_0)(1+\delta)+o(1)}-t^{2\mathsf{x}\kappa(\alpha_0)(1+2\delta)+o(1)}=t^{2\mathsf{x}\kappa(\alpha_0)(1+\delta)+o(1)},
    \end{align*}
    where $\mathsf{x}$ is the Euclidean scaling exponent.
    
    Finally, we bound the probability of $\cE_{z,r}(\alpha_0,\delta)$ from below for a point $z\in\Omega^*_{\delta,t}$. Note that $h'_r(z)$ (where we recall that 
    $h=h'-\gamma\log\vert\cdot-z\vert$ under $\P^*_z$) is a centered Gaussian random variable with variance $\var (h'_r(z))= \var (h_r(z))=\log(1/r)$. Thus, using \eqref{eq:compare h to h'}, we get
    \begin{align*}
        \PP^*_z[\cE_{z,r}(\alpha_0,\delta)]&\ge\int_{\alpha_0\log1/r}^{(\alpha_0+\delta)\log1/r}r^{(\gamma-\alpha)^2/2+o(1)}d\alpha\\
        &\ge r^{(\gamma-\alpha_0)^2/2+o(1)}\\
        &\ge t^{\kappa(\alpha_0)(1+2\delta)(\gamma-\alpha_0)^2/2+o(1)},
    \end{align*}
    where again the $o(1)$-term is uniform in $\alpha,z$ and $\delta$. So we get in total
    \begin{align*}
        I&\ge\int_{\Omega^*_{\delta,t}}\PP^*_z\otimes P_z[\nu(\tau_\Omega)<t\cap \cE_{z,r}(\alpha_0, \epsilon)]dz\\
        &=\int_{\Omega^*_{\delta,t}}\PP^*_z\otimes P_z[\nu(\tau_\Omega)<t\vert \cE_{z,r}(\alpha_0, \delta)]\PP^*_z[\cE_{z,r}(\alpha_0, \delta)]dz\\
        &\ge \frac{\vert\Omega^*_{\delta,t}\vert}{2} t^{\kappa(\alpha_0)(1+2\delta)(\gamma-\alpha_0)^2/2+o(1)}\\
        &=t^{2\mathsf{x}(\kappa(\alpha_0)(1+\delta)+\kappa(\alpha_0)(1+2\delta)(\gamma-\alpha_0)^2/2+o(1)}.
    \end{align*}
As $\delta>0$ was arbitrary, we have in total, using \eqref{eq:Delta_alt},
\begin{align*}
    I_1&\ge t^{\kappa(\alpha_0)(2\mathsf{x}+(\gamma-\alpha_0)^2/2)+o(1)}=t^{\Delta+o(1)}.
\end{align*}
    All together, we proved
    \begin{align*}
        \E[\Kb_\Omega(t)]=t^{\Delta+o(1)},
    \end{align*}
    which completes the proof of Proposition \ref{prop:first moment Heat content fractal boundary} assuming Lemmas \ref{lem:tailbound} and \ref{lem:decaymean}.
\end{proof}

\subsection{Proof of Lemmas \ref{lem:tailbound} and \ref{lem:decaymean} }
\label{SS:intermediary}

To complete the proof of Proposition \ref{prop:first moment Heat content fractal boundary}, it thus remains to verify the estimates for $\hat\sigma_1$ and for $u(0)-\underline{u}_r$ that were invoked above, namely Lemma \ref{lem:tailbound} and Lemma \ref{lem:decaymean}. These proofs are somewhat technical, and the rest of this subsection is devoted to developing them. We also note that the arguments are inspired by Duplantier and Sheffield (see Lemma 4.5 in \cite{DS2011}). We start by proving Lemma \ref{lem:decaymean}, as we actually also need it for the proof of Lemma \ref{lem:tailbound}. We present a self-contained exposition tailored to our framework and notation, expanding a few steps that were only sketched in the original source.

\begin{proof}[Proof of Lemma \ref{lem:decaymean}]
    The independence of $h'$ and $u$ is just given by the spatial Markov property of the GFF. As $u$ is harmonic on $B_{Nr}(z_0)$, we can write for all $z,w\in B_{r}(z_0)$
    \begin{align*}
        u(z)-u(w)=(h,\nu_z)-(h,\nu_w)=\frac{1}{2\pi}\int_0^{2\pi}h(z_0+Nr e^{i\phi})(\underbrace{f_z(\phi)-f_w(\phi)}_{=:f_{z,w}(\phi)})d\phi,
    \end{align*}
    where $\nu_z,\nu_w$ are the harmonic measures on $\partial B_{Nr}(z_0)$ of $z,w$ respectively and $f_z,f_w$ are the densities of $\nu_z,\nu_w$ to the uniform measure on $\partial B_{Nr}(z_0)$. In particular $f_{z,w}$ induces a signed measure on $\partial B_{Nr}(z_0)$ with total mass $0$. By estimating the first derivative of $f_{z,w}$ for $z,w\in B_{r}(z_0)$, it can be easily shown that
    \begin{align} \label{eq:estimate on f_{z,w}}
        \vert f_{z,w}(\phi)\vert\le L\frac{\vert z-w\vert}{(N-1)r},
    \end{align}
    where $L>0$ is a universal constant (note that for this bound $z,w\in B_r(z_0)$ is crucial); this corresponds to \textbf{Harnack's inequality}. We will write $L_N = L / (N-1)$; the only property we will need is that $L_N$ depends only on $N$ and tends to 0 as $N\to \infty$. By \eqref{eq:Greens function of full plane GFF}, and since $\bar B_{Nr}(z_0) \subset \D$, we have
    $
    G(z',w')=\log (1/{\vert z'-w'\vert})
    $ for any $z', w' \in \bar B_{Nr} (z_0)$. 
    Since for any fixed $x\in\CC$ the function $z\mapsto G(z,x)$ is harmonic on $\CC\setminus\{x\}$, we get for $z,w\in B_r(z_0)$,
    \begin{align*}
        \cov(u(z)-u(w),u(z_0))&=\frac{1}{4\pi^2}\int_0^{2\pi}\int_0^{2\pi}G(z_0+Nre^{i\phi},z_0+Nre^{i\psi})f_{z,w}(\phi)d\psi d\phi\\
        &=\frac{1}{2\pi}\int_0^{2\pi}G(z_0,z_0+N re^{i\phi})f_{z,w}(\phi)d\phi\\
        &=\frac{1}{2\pi}\int_0^{2\pi}\log\frac{1}{N r}f_{z,w}(\phi)d\phi=0.
    \end{align*}
    As $u$ is centered Gaussian, we obtain independence of $u(z)-u(w)$ and $u(z_0)$. We also want to estimate the variance of $u(z)-u(w)$. Using that $f_{z,w}$ integrates to $0$ and \eqref{eq:estimate on f_{z,w}}, we get
    \begin{align*}
        \var(u(z)-u(w))&=\frac{1}{4\pi^2}\int_0^{2\pi}\int_0^{2\pi}f_{z,w}(\phi)G(z_0+ Nre^{i\phi},z_0+ Nre^{i\psi})f_{z,w}(\psi)d\psi d\phi\\
        &=\frac{1}{4\pi^2}\int_0^{2\pi}\int_0^{2\pi}f_{z,w}(\phi)\log\frac{1}{Nr\vert e^{i\phi}-e^{i\psi}\vert}f_{z,w}(\psi)d\psi d\phi\\
        &\le\frac{1}{4\pi^2}\int_0^{2\pi}\int_0^{2\pi}\vert f_{z,w}(\phi)\vert \Big\vert\log\frac{1}{\vert e^{i\phi}-e^{i\psi}\vert}\Big\vert\vert f_{z,w}(\psi)\vert d\psi d\phi\\
        &\qquad+\frac{1}{4\pi^2}\int_0^{2\pi}\int_0^{2\pi}\log\frac{1}{N r}f_{z,w}(\phi)f_{z,w}(\psi)d\psi d\phi\\
        &\le L_N^2\frac{\vert z-v\vert}{4\pi^2r^2}\int_0^{2\pi}\int_0^{2\pi} \Big\vert\log\frac{1}{\vert e^{i\phi}-e^{i\psi}\vert}\Big\vert d\psi d\phi\\
        &\le CL_N^2 \frac{\vert z-w\vert^2}{r^2},
    \end{align*}
 where $C>0$ is a universal constant.
    
    Let us use this to deduce the desired Gaussian tail for the supremum of $\vert u(z)-u(z_0)\vert$ over $z\in B_r(0)$. For $n\ge1$, we define the lattice
    \begin{align*}
        A_n:=\{z\in B_r(z_0):z=(r2^{-n}k,r2^{-n}l), k,l\in\ZZ\}
    \end{align*}
    Then we have $\vert A_n\vert\le 4r^22^{2n}$. Furthermore, for any given $z\in B_r(z_0)$ we define a sequence $(z_n)_{n\in\NN}$, by setting $z_n$ to being the closest point in $A_n$ to $z$ 
    Then note that $\vert z_n-z_{n-1}\vert\le \sqrt{2}r2^{-(n-1)}$ and thus
    \begin{align} \label{eq:variance estimate u}
        \var(u(z_n)-u(z_{n-1}))\le CL_N^2 2^{-2n},
    \end{align}
     for some universal constant $C>0$. As $u$ is harmonic and therefore also continuous, we can write
    \begin{align*}
        u(z)-u(z_0)=\sum_{n\ge1}u(z_n)-u(z_{n-1}).
    \end{align*}
    Let $\lambda\ge\lambda_0$ for some fixed $\lambda_0>0$. Then using \eqref{eq:variance estimate u} and the estimate on the variance of $u(z)-u(w)$ above, we have
    \begin{align}\label{eq:estimate in proof of decay mean}
        \PP[\sup_{z\in B_r(z_0)}\vert u(z)-u(z_0)\vert\ge\lambda]&\le\PP[\sup_{z\in B_r(z_0)}\vert \sum_{n\ge1}u(z_n)-u(z_{n-1})\vert\ge\frac{6\lambda}{\pi^2}\sum_{n\ge1}n^{-2}]\nonumber\\
        &\le\sum_{n\ge1}\PP[\sup_{z\in B_r(z_0)}\vert u(z_n)-u(z_{n-1})\vert\ge\frac{6\lambda}{\pi^2}n^{-2}]\nonumber\\
        &\le\sum_{n\ge1}8\sum_{z_n\in A_n}\PP[\vert u(z_n)-u(z_{n-1})\vert\ge\frac{6\lambda}{\pi^2}n^{-2}]\nonumber\\
        &\le C_1\sum_{n\ge1}2^{2n}\exp\Bigl(-\frac{C_2 L_N^{-2} \lambda^22^{2n}}{n^4}\Bigr),
    \end{align}
    for some universal constant $C_1,C_2>0$. To conclude, we first notice that there is a universal constant $C>0$ such that
    \begin{align*}
        \frac{C_2L_N^{-2} \lambda^2 2^{2n}}{n^4} & \ge CL_N^{-2} \lambda^2(n+2^n)\\
        & \ge C L_N^{-2}\lambda^2 n + C  \lambda^2 2^{2n} .
    \end{align*}
    
    Using this, we obtain
    \begin{align*}
        \sum_{n\ge1}2^{2n}\exp\Bigl(-\frac{C_2\lambda^22^{2n}}{n^4}\Bigr)\le\max_{n\ge1}\exp\Bigl(2\log(2)n-C\lambda^22^n\Bigr)\sum_{n\ge1}e^{-CL_N^{-2}\lambda^2n}.
    \end{align*}
    By maximizing over any $n\in[1,\infty)$, we get that the maximum can be estimated by
    \begin{align} \label{eq:proof of decaymean 1}
        \max_{n\ge1}\exp\Bigl(2\log(2)n-C\lambda^22^n\Bigr)\le\Bigl(\frac{2}{C\lambda^2}\Bigr)^2e^{-2}\le C' <\infty,
    \end{align}
    assuming $\lambda\ge\lambda_0$, for some constant $C'$ only depending on $\lambda_0$. Furthermore, the sum gives
    \begin{align} \label{eq:proof of decaymean 2}
        \sum_{n\ge1}e^{-C\lambda^2n}=\frac{e^{-C L_N^{-2}\lambda^2}}{1-e^{-C L_N^{-2} \lambda^2}}<\frac1{1- e^{-C}}e^{-CL_N^{-2}\lambda^2}.
    \end{align}
     By putting \eqref{eq:proof of decaymean 1} and \eqref{eq:proof of decaymean 2} together, we obtain
    \begin{align*}
        \PP[\sup_{z\in B_r(z_0)}\vert u(z)-u(z_0)\vert\ge\lambda]\le ce^{-c'\lambda^2},\qquad\lambda\ge\lambda_0
    \end{align*}
    for constants $c,c'>0$, where $c$ depends only on $N$ and $\lambda_0$ and where $c' = C L_N^{-2}$ uniform in $\lambda_0$ can be made arbitrarily big by choosing $N$ large enough since $L_N \to 0$ as $N\to \infty$. This completes the proof of Lemma \ref{lem:decaymean}.
\end{proof}

Now we are ready to prove Lemma \ref{lem:tailbound} and complete this section.

\begin{proof}[Proof of Lemma \ref{lem:tailbound}]
First we note that we can assume without loss of generality that $\alpha = 0$, as when $\alpha>0$ the exit time becomes larger (since $\log (1/ |z|) \ge 0$ in $\D$). We therefore assume $\alpha = 0$ in the rest of the proof. The strategy is inspired by Duplantier and Sheffield’s approach in \cite{DS2011} (where the Liouville area measure of a ball is used in place of the quantum clock accumulated by an independent Brownian motion until its exit time from the unit disc). Roughly, the idea is as follows. We estimate the exit time from below by two smaller copies of itself, by considering two balls of a slightly smaller radius (but still same scale overall). For each of them, by the spatial Markov property, the GFF splits into an independent part plus a harmonic remainder, and Lemma \ref{lem:decaymean} controls the harmonic contribution. This yields a lower bound on the original exit time by the maximum of two independent, smaller copies of itself, up to a small error term. By the change of coordinate formula of the Liouville measure, these smaller exit times get transported back to the original scale, resulting in a recursive inequality for the distribution function of the exit time. Estimating this recursion gives the desired bound.\\

\textbf{Setup.} To make things rigorous, fix a path of the underlying (ordinary) Brownian motion $(B_t, t \ge 0)$ up to its exit from $B_2(0)$. Let 
\begin{align*}
    \tau_{1/2}=\inf\{t>0:\vert B_t\vert=1/2\},\qquad \text{ and let } z_0=B_{\tau_{1/2}}
\end{align*}
be the first point where the (Euclidean) Brownian motion attains modulus 1/2. By Lemma \ref{lem:decaymean}, by choosing $N$ sufficiently large we suppose without loss of generality that the constant $c'$ in this Gaussian bound is such that $c'>\log2$. 

For notational convenience, set
\begin{align*}
    D^+&=B_{1/4}(0),\qquad D^-=B_{1/4}(z_0),\qquad D=D^+\cup D^-,\\
    B^+&=B_{1/(4N)}(0)\subseteq D^+,\qquad B^-=B_{1/(4N)}(z_0)\subseteq D^-.
\end{align*}
The argument runs as follows. Given the path of the ordinary Brownian motion $(B_t, t \ge 0)$, we bound $\sigma_\D = \nu(\tau_\DD)$ from below by two ingredients: 
-- The exit time $\nu(\tau_{B^+})$ of the smaller ball $B^+$ -- and the clock accumulated by $B$ between $\tau_{1/2}$ and $\tau_{B^-}$, defined as the first time after $\tau_{1/2}$ when $B$ leaves the ball $B^-$. Equivalently, this can be viewed as the clock $\nu^{(h, \tilde B)}(\tilde \tau_{B^-})$ where the underlying planar Brownian motion is  $\tilde B_s = B_{s+\tau_{1/2}}, s\ge 0$, and $\tilde \tau_{B^-} = \inf \{ s > 0: \tilde B_s \notin B^-\}$.

\medskip To make them independent, we use the spatial Markov property of the GFF $h^{2\DD}$ to get $h^{2\DD}=h^{D^+}+h^{D^-}+u$, where $h^{D^+}, h^{D^-}$ are independent Dirichlet GFFs in $D^+$ and $D^-$ respectively and $u$ is independent of the two other terms and harmonic in $D^+\cup D^-$. We define
\begin{align*}
    A &= \log\nu^{2\DD}(\tau_{\DD}),\\
    A^+ &= \log\nu^{(h^{D^+}, B)}(\tau_{B^+}),\\
    A^- &= \log\nu^{(h^{D^-}, \tilde B)}(\tilde \tau_{B^-}),
\end{align*}
where $\nu^{D^+},\nu^{D^-}$ are the quantum clocks associated to $h^{D^+}$ and $h^{D^-}$ respectively. Now, by the strong Markov property of planar Brownian motion, given $\mathcal{F}_{\tau_{1/2}} = \sigma( B_s: 0 \le s \le \tau_{1/2})$,  $(\tilde B_{s}, s \ge 0)$ is a Brownian motion starting in $z_0$. Since $h^{D^+}$ and $h^{D^-}$ are also independent, we deduce that $A^+$ and $A^-$ are independent.\\

\textbf{Construction of the recursion.} By the change of coordinate formula of LQG, we have
\begin{align}
    A^+, A^- \overset{(d)}{=} A - (2+\gamma^2/2)\log(8N). 
\end{align}
We will write $b = ( 2+ \gamma^2/2) \log (8N)$. Its value for the rest of the argument is of little relevance but we note that $b >0$. Moreover, we have that
\begin{align*}
    \nu^{2\DD}(\tau_{\DD})\geq \nu^{2\DD}(\tau_{B^+})\geq e^{\gamma\underline{u}}\nu^{D^+}(\tau_{B^+}),
\end{align*}
and similarly for $A^-$, so we conclude that 
\begin{align} \label{eq:boundAs}
    A\geq\max\{A^-,A^+\}+\gamma\underline{u},
\end{align}
with $\underline{u}=\inf_{z\in B^+\cup B^-}u(z)$. Our goal is to show that, for sufficiently negative $x$, 
\begin{align*}
    p_A(x):=\PP\otimes P_0[A\leq x]\leq e^{-c'x^2}.
\end{align*}

{
For $\rho\in(0,1)$ that we will specify later, we split this probability into two terms, 
\begin{align*}
    p_A(x)=\underbrace{\PP\otimes P_0[A\leq x,\underline{u}\leq\rho x/\gamma]}_{=:p_1(x)}+\underbrace{\PP\otimes P_0[A\leq x,\underline{u}>\rho x/\gamma]}_{=:p_2(x)}.
\end{align*}
Now, Lemma \ref{lem:decaymean} allows us to immediately bound the first term, so that for sufficiently negative $x$,
\begin{align}\label{eq:sufficiently negative}
    p_1(x)\leq e^{-c'x^2},
\end{align}
where we made $c'$ smaller to get rid of the factor $\rho^2/\gamma^2$ in the exponent and also the constant $c>0$. Using \eqref{eq:boundAs}, and the independence of $A^+$ and $A^-$, we get 
\begin{align*}
    p_2(x)
    \le \P(\max (A^-, A^+) \le (1- \rho) x) \le p_A\Bigl((1-\rho)x+\underbrace{(2+\gamma^2/2)\log8N}_{=b}\Bigr)^2,
\end{align*}
which in total gives 
\begin{align*}
    p_A(x)\le   p_A\Bigl((1-\rho)x+b\Bigr)^2+e^{-c'x^2},
\end{align*}
for some constant $c'>0$. From here the idea is to use a recursive argument to obtain the desired Gaussian bound. Such an idea is sketched in \cite{DS2011}; however, we could not follow the details of their argument, and indeed our numerical simulations suggest that the induction alone as used in \cite{DS2011} is not enough to guarantee a Gaussian bound without additional nontrivial information on the start of this recursion. The same question is also studied by Aru in \cite[Lemma 4.3]{Arunotes}, following a similar idea but with a cleaner approach, and we follow this instead. While some of the details of the argument there were not entirely written (and indeed the crucial iteration step cannot be carried out all the way up to $\nu_0$ in Aru's notations) we were nevertheless able to carry out the argument and provide a complete proof following his outline. The details are as follows.


Note that if we choose $N=2$ so $b = (2+ \gamma^2/2) \log (8N) $ then $b\le 16 \log 2$, so if $x \le -1000$ and $\rho =0.15$ then $x' := ((1- \rho) x + b) \le 0.8 x = (4/5) x$. Using monotonicity of $p_A$ we deduce
\begin{equation}\label{eq:Induction_Aru}
p_A(x) \le e^{ - c' x^2} + p_A( 4x/5)^2; \quad \text{ for all } x \le - 1000.
\end{equation}
Fix $a = c'/2$ and let $x_0 \le -1000$ be such that if $x \le x_0$ then \eqref{eq:sufficiently negative} holds true and also
\begin{equation}\label{eq:x0}
e^{ - a x^2 } - e^{ - c' x^2} \ge e^{ - 0.9 a x^2}.
\end{equation}

Consider $x_m = -x_0 \cdot (5/4)^m$. By monotonicity of $p_A$ again it suffices to show that $p_A(x_m) \le e^{ - a x_m^2} $ for all sufficiently large $m\ge 0$. Suppose for contradiction that: 
\begin{equation}\label{eq:contrad}
p_A(x_m) > e^{ - a x_m^2}, \quad \text{ for infinitely many values of $m$.}
\end{equation}
If $m\ge 1$ is any such value we also have from \eqref{eq:Induction_Aru}
$$
e^{ - a x_m^2} \le p_A(x_m) \le e^{ - c' x_m^2} + p_A(x_{m-1})^2
$$
so by \eqref{eq:x0} (since $x_m \le x_0$),
$$
p_A(x_{m-1})^2 \ge e^{ - a x_m^2} - e^{ - c' x_m^2} \ge e^{ - 0.9 a x_m^2}.
$$
This implies (since $0.45 \cdot (5/4)^2 \approx 0.7 \le 0.9$)
\begin{equation}\label{induction_k=1}
p_A(x_{m-1}) \ge e^{ - 0.45 a x_m^2} \ge e^{ - 0.9 a x_{m-1}^2}.
\end{equation}
This can be iterated by induction: for $0 \le k \le m$, let 
$$
y_k = 0.9^{k/2} x_{m-k}.
$$ 
Thus $y_0 = x_m$, and as $k$ increases, $y_k$ gradually shrinks closer to 0 at a geometric rate which is faster than $x_{m-k}$ itself; in particular $y_k \ge x_{m-k}$. As a consequence, if 
$$
k_m = \min \{ k \ge 0: y_k \ge x_0\}
$$
is the first time for which 
$y_k$ gets above the value $x_0$, then $k_m< m$, and in fact $m - k_m \to \infty$ as $m\to \infty$ (linearly with $m$, in fact). 
We will show by induction on $k\ge 0$ that for all $0 \le k < k_m$, 
\begin{equation}\label{eq:iterated}
p_A (x_{m-k}) \ge e^{- 0.9^k a x_{m-k}^2} = e^{ - a y_k^2}.
\end{equation}
For completeness we write out the induction carefully. The case $k =0$ was the assumption on $m$, and we have just done by hand the case $k = 1$ in \eqref{induction_k=1}. Let $k \ge 1$ and suppose that \eqref{eq:iterated} holds. Then  from \eqref{eq:Induction_Aru}, 
\begin{align*}
p_A(x_{m-k-1})^2 & \ge e^{ - a y_{k}^2 } - e^{ - c' x_{m-k}^2} \\
&\ge  e^{ - a y_{k}^2 } - e^{ - c' y_{k}^2} \quad \text{ (since $x_{m-k} \le y_k$)}\\
& \ge e^{ - 0.9 a y_{k}^2} = e^{ - 0.9^{k+1} a x_{m-k}^2}, 
\end{align*}
where we used \eqref{eq:x0}, thanks to the fact that $k< k_m$ and thus $y_k \le x_0$. Therefore, taking square roots, 
\begin{align*}
p_A(x_{m-k-1})& \ge e^{ - (1/2) 0.9^{k+1} a x_{m-k}^2} \\
& = e^{ - 0.45 \cdot (5/4)^2 \cdot 0.9^k a x_{m-k-1}^2} 
\\ 
& \ge e^{ - 0.9^{k+1} a x_{m-k-1}^2} = e^{ - 0.9^{k+1} y_{k+1}^2}.
\end{align*}
This completes the induction. 

Now let us evaluate \eqref{eq:iterated} at $k = k_m-1$. This reads: 
\begin{equation}\label{eq:contrad2}
p_A ( x_{m- k_m+1} ) \ge e^{ - a y_{k_m-1}^2} \ge \rho: =  e^{ - a \frac{x_0^2}{ 0.9 (4/5)^2}}.
\end{equation}
Note that \eqref{eq:contrad2} holds for the infinitely many values of $m$ for which \eqref{eq:contrad} holds. In particular, since $m-k_m \to \infty$ as $m\to \infty$ we have a subsequence $n_m \to \infty$ such that $p_A ( x_{n_m}) \ge \rho$. This contradicts the fact that $p_A (x) \to 0$ as $x \to - \infty$ (which follows from the fact that $\nu^{2\DD}( \tau_\DD) >0$ a.s., see, e.g., Theorem 3.42 in \cite{BP}). Therefore, \eqref{eq:contrad} cannot hold, and this concludes the proof of Lemma \ref{lem:tailbound}.
}
\end{proof}

As an immediate consequence of Lemma \ref{lem:tailbound} we obtain finiteness of all negative moments for the exit time. We write this below as this may be of independent interest.

\begin{cor} \label{cor:finite negative moment of time change}
    In the same setting as Lemma \ref{lem:tailbound}, we have for all $p>0$
    \begin{align*}
        \EE\otimes P_0\Bigl[\frac{1}{\sigma_{\DD}^p}\Bigr]<\infty.
    \end{align*}
\end{cor}


Note that Theorem 3.42 in \cite{BP} shows finiteness of \emph{all} negative moments for Gaussian multiplicative chaos. It seems initially tempting to want to deduce Corollary \ref{cor:finite negative moment of time change} from this result: indeed, when we condition on the trajectory of the Brownian path, then $\sigma_\DD = \nu(\tau_\DD)$ can be viewed as a GMC with respect to a reference measure given by the occupation measure of the Brownian path. Then Theorem 3.42 in \cite{BP} would apply and give finiteness of all negative moments but only conditionally given the Brownian path; one would still need to take expectation with respect to the Brownian path. Unfortunately, this last step is far from obvious, even if we keep track of the explicit dependence of the moments on the Brownian path itself, which, following the work of \cite{BerestyckiSheffieldSun} (which itself relies on \cite{GHSS}), boils down to the  Riesz energies of the stopped Brownian path, i.e., quantities of the form
$\int_0^{\tau_\DD}\int_0^{\tau_\DD} 1/ |B_t - B_s|^q ds dt$ for some $q>0$. 


\subsection{Second moment computation} \label{sec:2.2}

In this section we complement our estimates on $\E( \mathbf{K}_\Omega(t))$ (which, in particular, gives an upper bound valid with high probability on the heat content itself via Markov's inequality) with an upper bound on the second moment. Ultimately, the goal is to show that the heat content is concentrated so that we get the lower bound also with high probability.

However, in common with many problems involving Gaussian multiplicative chaos, a straightforward bound on the second moment of the type
$
    \EE[\Kb_\Omega(t)^2]\le t^{2\Delta+o(1)}
$    
is too much to hope for. Indeed, the total mass $\mu(\Omega)$ itself does not even have finite second moment when $\gamma \notin [0, \sqrt{2})$. We will therefore need to apply a truncation argument in the spirit of \cite{Ber2017} (see also \cite[Chapter 3]{BP}). In essence, we have identified in the previous section an optimal distance from the boundary (roughly $r = t^{\kappa(\alpha_0)}$, where $\alpha_0=Q-\sqrt{Q^2-4(1-\mathsf{x})} = \gamma(1-\Delta)$, see \eqref{eq:Delta_alt}) which carries most of the contribution to the heat content. In addition, we know that the dominant contribution comes from points whose circle average is of order $h_r(z) \approx \alpha_0 \log (1/r)$, where $r = \text{dist}(z, \partial \Omega)$, consistent with the KPZ heuristics. At the same time, we know that points which contribute to the heat content are themselves by definition sampled from Liouville measure and are thus $\gamma$-thick (\cite[Chapter 3]{BP}), not $\alpha_0$-thick, i.e., at scales $\eps\ll r$, one expects $h_\eps(z) \approx \gamma \log (1/\eps)$. Finding the good event which carries all the mass in an $L^1$ sense, and does not make the second moment explode, is therefore more subtle than in the standard case discussed, say, in \cite{Ber2017}.

Our solution is as follows. Recall the set $\partial_R = \{ z_1, \ldots, z_N \}$ from Assumption \ref{Assumption 2}, associated to a scale $R>0$ and a small error $\zeta>0$. For each $t>0$, choose a radius $R(t)>0$ depending only on $t>0$, such that $R(t)\to0$ when $t\to0$ but slower than any polynomial (STP), i.e. $R(t) = t^{o(1)}$ or equivalently, $t^{-\eps}R(t)\to\infty$ for all $\eps>0$ when $t\to0$. We also
 fix $\zeta>0$ arbitrarily (later $\zeta$ and $R(t)$ will be fixed by \eqref{eq:R_N}). Given $R(t)$ and $\zeta$, we consider the set $\partial_{R(t)}$ coming from Assumption \ref{Assumption 2}. For $z_i \in \partial_{R(t)}$, we define the truncated heat content as
\begin{align} \label{eq:truncated heat content}
    \Kb_{\Omega,z_i}(t):=\displaystyle\int_{\Omega^t\cap B_{R(t)}(z_i)}\Pb_z[\sigma_\Omega<t]\textbf{1}_{G_t}\mu_{\gamma}(dz),
\end{align}
where the good region $\Omega^t$ is defined by
\begin{align} \label{eq:def of Omega^t}
\Omega^t:=\Bigl\{z\in \Omega:\Bigl(\frac{C_\partial}{2c_\partial}\Bigr)^\frac{1}{2\mathsf{x}}t^{\kappa(\alpha_0)}\le d(z,\partial^*\Omega)\le t^{\kappa(\alpha_0)}\Bigr\}.
\end{align}
(This should not be confused with our notation $\Omega_r = \{z \in \Omega: d(z, \partial \Omega) \le r\}$.) Here, we recall that the constants $c_\partial, C_\partial$ are those appearing in the boundary assumptions \eqref{assumption 1} and \eqref{Assumption 2}). The event $G_t = G_t(z)$ is a good event defined for $z \in \Omega^t$, defined as follows. Let us fix $\beta > \gamma$ arbitrarily (later it will be taken suitably close to $\gamma$). Fix also $N>0$ large enough such that (by scaling for ordinary Brownian motion and Beurling's estimate)
\begin{align} \label{eq:Beuerling's estimate}
P_z[\tau_\Omega>\tau_{Nd(z,\partial\Omega)}]\le p_1 = 0.1,
\end{align}
uniformly over all $z\in\Omega^t$. Writing $r=Nd(z,\partial^*\Omega)$, we impose the following conditions:
\begin{align} \label{conditions}
    &G_t = G_t(z):=\Bigl\{(C1)\text{ and }(C2) \text{ are fulfilled}\Bigr\}\nonumber\\
    (C1)&\qquad h_r(z)\in[\alpha_0 \log (1/r)-2M,\alpha_0 \log (1/r)-M]\\
    (C2)&\qquad h_{r^{1+q}}(z)\le(\alpha_0+q\beta)\log(1/r),\qquad q\ge0
\end{align}
for some $\beta>\gamma$ sufficiently close to $\gamma$, and where $M$ is a sufficiently large constant which will be fixed below depending solely on $\beta$ and $\gamma$. Roughly speaking, condition (C1) imposes the correct order of magnitude for the circle average at the scale corresponding to the distance to the boundary (which is the correct intrinsic thickness of points on the boundary which actually contribute to the heat content), while condition (C2) interpolates between the thickness at scale $r$, which is $\alpha_0$, and that at infinitesimal scales, which is $\gamma$.


Note that because of our assumptions on the boundary and by \eqref{Assumption 1 non local}, we have
\begin{align*}
    \vert\Omega^t\cap B_{R(t)}(z_i)\vert&\le\vert \Omega^t\vert\le\vert \Omega_{t^{\kappa(\alpha_0)}}\vert\le \tilde C_\partial t^{2\mathsf{x}\kappa(\alpha_0)},\\
    \vert\Omega^t\cap B_{R(t)}(z_i)\vert&\ge\vert\Omega^*_{t^{\kappa(\alpha_0)}}\cap B_{R(t)}(z_i) \vert-\vert \Omega_{(c_\partial/2C_\partial)^{1/2\mathsf{x}}t^{\kappa(\alpha_0)}}\cap B_{R(t)}(z_i)\vert\\
    &\ge c_\partial t^{2\mathsf{x}\kappa(\alpha_0)}R(t)^{2(1-\mathsf{x})}-C_\partial \frac{c_\partial}{2C_\partial}t^{2\mathsf{x}\kappa(\alpha_0)}R(t)^{2(1-\mathsf{x})}\\
    &=\frac{c_\partial}{2}t^{2\mathsf{x}\kappa(\alpha_0)}R(t)^{2(1-\mathsf{x})},
\end{align*}
for all $t>0$ and for all $z_i \in \partial_{R(t)}$.
In total we get
\begin{align} \label{eq:size of local version domain}
    \vert\Omega^t\cap B_{R(t)}(z_i)\vert=t^{\mathsf{x}\kappa(\alpha_0)+o(1)},
\end{align}
where $o(1)\to0$ when $t\to0$ is an error term whose magnitude is bounded by a quantity tending to 0 uniformly in $z_i \in \partial_{R(t)}$. 


The plan for this section is as follows. First, we show that the truncation does not remove a significant contribution: $$\E[\Kb_{\Omega,z_i}(t)] \ge t^{o(1)} \E [ \Kb_\Omega(t)].$$
Second, we estimate the second moment of $\Kb_{\Omega,z_i}(t)$ and show that $$\E[\Kb_{\Omega,z_i}(t)^2] \le \E[\Kb_{\Omega,z_i}(t)]^{2+ o(1)}.$$
This then shows (via a standard Paley--Zygmund argument) that $\Kb_{\Omega,z_i}(t)$ is not much smaller than its expectation with reasonable probability, i.e.,  with a probability at least of order $t^{o(1)}$. This could however still tend to zero, and we use a decoupling argument to boost this probability from $t^{o(1)}$ to close to 1. This last step is the subject of Section \ref{sec:2.3}.

Before starting with the proofs, we describe the two conditions $(C1),(C2)$ under Girsanov’s transform.  Recall that $X_s:=h_{e^{-s}}(z)$ is a Brownian motion in $s$,  so the conditions constrain the long-time behavior of $X$, when $t$ gets small. By the Girsanov--Cameron--Martin lemma,
\begin{align}\label{eq:GCMKtilde}
    \EE[\Kb_{\Omega,z_i}(t)]=\displaystyle\int_{\Omega^t\cap B_{R(t)}(z_i)}\EE^*\Bigl[\Pb_z[\sigma_\Omega<t]\textbf{1}_{G_t}\Bigr]dz,
\end{align}
where under $\PP^*$, the field decomposes as $h=h'+\gamma G(z,\cdot)$ with $h'$ being a whole plane GFF (normalised to have zero unit circle average). Consequently, by \eqref{eq:Greens function of full plane GFF} $X_s$ becomes a Brownian motion with drift $\gamma s$. Writing $X'_s:=h'_{e^{-s}}(z)$ (which is a Brownian motion under $\P^*$), the conditions become
\begin{align} \label{conditions'}
        (C1)'&\qquad X'_s\in \bigl[(\alpha_0-\gamma)s-2M,(\alpha_0-\gamma)s -M\bigr],\\
        (C2)'&\qquad X'_{s(1+q)}\le\bigl(\alpha_0-\gamma+(\beta-\gamma)q\bigr)s,\qquad q\ge0,
\end{align}
where $s=\log(1/r)$, $M$ is still to be fixed, and as before, $r = Nd(z, \partial \Omega)$. In this form, the conditions are well suited for the first-moment analysis.

\begin{lem} \label{lem:modified expectation}
For $\Kb_{\Omega,z_i}(t)$ defined as above, we have for all $z_i\in \partial_{R(t)}$,
\begin{align*}
    \E[\Kb_{\Omega,z_i}(t)]=t^{\Delta+o(1)}\ \ 
\end{align*}
where the $o(1)$-term is bounded by a quantity which tends to zero as $t \to 0$, uniformly in $z_i \in \partial_{R(t)}$. In fact, there exists $\epsilon (t) \to 0$ such that the following holds. For any choice of $R(t) \to 0$: 
\begin{align}\label{eq:1stmoment_strong}
    \E[\Kb_{\Omega,z_i}(t)]\ge t^{\Delta+\epsilon(t)} R(t)^{2 (1- \mathsf{x})},
\end{align}
where $\Delta$ is as in Proposition \ref{prop:first moment Heat content fractal boundary}, uniformly in $z_i \in \partial_{R(t)}$.
\end{lem}
\begin{proof}
    Clearly, we have $\EE[\Kb_{\Omega,z_i}(t)]\le\EE[\Kb_\Omega(t)]$, so it suffices to establish the  lower bound. Our goal will be to establish that 
    \begin{equation}\label{eq:1st_trunc_goal}
\P^* \otimes P_z ( \nu (\tau_\Omega)< t ; G_t) \ge c \P^*(G_t) 
    \end{equation}
    for some constant $c>0$, uniformly in $z \in \Omega^t \cap B_{R(t)} (z_i)$ and uniformly in $z_i \in \partial_{R(t)}$. Note that this implies the strong lower bound \eqref{eq:1stmoment_strong}.
    Indeed, given this, we can write by the Girsanov--Cameron--Martin lemma (more specifically \eqref{eq:GCMKtilde}):
     \begin{align*}
        \EE[\Kb_{\Omega,z_i}(t)]&\ge\int_{\Omega^t \cap B_{R(t)} (z_i)}\PP^*_z\otimes P_z[\nu(\tau_\Omega)<t; G_t]dz\ge c\int_{\Omega^t\cap B_{R(t)}(z_i)}\PP^*_z[G_t]dz
        \end{align*}
    Recall that $G_t = (C1)' \cap (C2)'$. We can estimate the probability of $G_t$ simply by that of $(C1)'$: more precisely, there exists $p_0$ depending solely on $\beta $, $\gamma$ and $M$ such that 
    \begin{equation}\label{eq:p0}
\P^*_z ( (C2)' (z) | (C1)'(z) ) \ge \P [ B_t \le M + ( \beta - \gamma) t, \forall t \ge 0] \ge p_0, 
    \end{equation}
    where $(B_t)_{t\ge 0}$ is a standard one-dimensional Brownian motion starting from $B_0 = 0$. Furthermore, if $\beta$ and $\gamma$ are fixed, by choosing $M$ large enough we can ensure that $p_0 \ge 0.9$.
    
    Thus, 
    $$
\EE[\Kb_{\Omega,z_i}(t)] \ge c p_0 \int_{\Omega^t}\PP^*_z [ (C1)'(z)] dz.
    $$
    As $X'_s$ is a centered Gaussian random variable with variance $\log1/r$ and by using $r=Nd(z,\partial\Omega)$ and \eqref{eq:def of Omega^t}, we have
    \begin{align*}
        \PP^*_z[(C1)'(z)]\ge C\frac{t^{\kappa(\alpha_0)(\alpha_0-\gamma)^2/2}}{\log1/t},
    \end{align*}
    where $C>0$ is a constant which is uniform in $R(t),z$ and $z_i$. 
We deduce that 
     \begin{align*}
        \EE[\Kb_{\Omega,z_i}(t)]&\ge cp_0 \left|\Omega^t\cap B_{R(t)}(z_0)\right| \PP^*_z[(C1)'(z)]\\
        &\ge c_\partial t^{2\mathsf{x}\kappa(\alpha_0)}R(t)^{2(1-\mathsf{x})}\frac{C}{\log1/t}t^{\kappa(\alpha_0)(\alpha_0-\gamma)^2/2}\\
        &\ge c\frac{t^{\Delta} R(t)^{2(1-\mathsf{x})}}{\log1/t}=t^{\Delta+\epsilon(t)}R(t)^{2(1-\mathsf{x})}.
    \end{align*}
where $c>0$ is a constant which is uniform in $R(t)$ and $z_i$ and $\epsilon(t)$ is tending to 0 as $t\to 0$, uniformly in $z_i \in \partial_{R(t)}$, and uniformly in the choice of $R(t)$.
    
It therefore suffices to prove \eqref{eq:1st_trunc_goal}. Now, observe that 
\begin{align}
    \P^*_z \otimes P_z [ \nu (\tau_\Omega) < t ; G_t] & \ge  \P^*_z \otimes P_z [ \nu (\tau_\Omega) < t ; (C1)'] -  \P^*_z \otimes \Pb_z [ \nu (\tau_\Omega) < t; (C1)' ; \overline{(C2)'}]
  \nonumber  \\
    & \ge  \P^*_z \otimes P_z [ \nu (\tau_\Omega) < t ; (C1)']  - \P^*_z \otimes P_z [  (C1)' ; \overline{(C2)'}] \nonumber \\
    & \ge  \P^*_z \otimes P_z [ \nu (\tau_\Omega) < t ; (C1)']  - (1-p_0) \P^*_z [ (C1)'] \label{eq:GtC1}
\end{align}
where $\overline{(C2)'}$ denotes the complement of the event $(C2)'$, and $p_0$ is as in \eqref{eq:p0}. 

To estimate the first term in the right hand side, we will use a Beurling estimate to compare $\nu (\tau_\Omega)$ with $\nu (\tau_r)$ (cf. \eqref{eq:Beuerling's estimate}) as well as a scaling argument,     where we recall that $r=Nd(z,\partial\Omega)$. 
More precisely, 
\begin{align} 
\P^*_z \otimes P_z [ \nu (\tau_\Omega) < t ; (C1)']  & \ge \P^*_z \otimes P_z [ \nu (\tau_r) < t ; (C1)'] - \P^*_z \otimes P_z [ \nu (\tau_r) < \nu (\tau_\Omega) ; (C1)'] \nonumber \\
& \ge \P^*_z \otimes P_z [ \nu (\tau_r) < t ; (C1)'] - \P^*_z \otimes P_z [ \tau_r < \tau_\Omega ; (C1)'] \nonumber \\
& \ge \P^*_z \otimes P_z [ \nu (\tau_r) < t ; (C1)'] - p_1 \P^*_z [ (C1)']. \label{eq:C1tau_r}
\end{align}
Now, by Lemma \ref{lem:scaling without}, 
we have for $z \in \Omega^t \cap B_{R(t)} (z_i)$, 
 \begin{align*}
     \nu(\tau_{r})\le r^{2-\gamma^2/2}e^{\gamma\bar u_r}\hat\sigma_1=r^{2-\gamma^2/2}e^{\gamma u(z)}e^{\gamma(\bar u_r-u(z))}\hat\sigma_1,
 \end{align*}
 where all three factors $e^{\gamma u(z)}$, $e^{\gamma(\bar u_r-u(z))}$ and $\hat\sigma_1$ are mutually independent. 
When $(C1)'$ holds, we therefore deduce (after some algebra) that 
\begin{align*}
 \nu(\tau_{r})& \le r^{1/\kappa(\alpha_0)} e^{ - \gamma M}  e^{\gamma(\bar u_r-u(z))} \hat \sigma_1\\
 & =  N^{1/\kappa ( \alpha_0)} t e^{ - \gamma M} e^{\gamma(\bar u_r-u(z))} \hat \sigma_1.
\end{align*}
We may assume without loss of generality that 
\begin{equation}
\label{eq:MN}
N^{ 1/ \kappa (\alpha_0)} e^{ - \gamma M/2} \le 1,
\end{equation}
so that in total, canceling the factors of $t$ on either side of the inequality, 
\begin{align*}
\P^*_z \otimes P_z [ \nu (\tau_r) < t ; (C1)']     \ge \P^*_z \otimes P_z [e^{\gamma(\bar u_r-u(z))} \hat \sigma_1 < e^{ \gamma M /2} ; (C1)'].
\end{align*}
Now, observe that $(C1)'$ depends only on $u(z)$ and is thus independent of both random variables appearing in the right hand side. We deduce 
\begin{align*}
\P^*_z \otimes P_z [ \nu (\tau_r) < t ; (C1)']
& \ge \P^*_z [ (C1)'] \P^*_z \otimes P_z [e^{\gamma(\bar u_r-u(z))} \hat \sigma_1 < e^{ \gamma M /2}].
\end{align*}
The two random variables above are furthermore independent of each other, and their law does not depend on $r$ or $z$ (by scaling invariance of the Gaussian free field in the plane modulo constants). By choosing $M$ sufficiently large and using the monotone convergence theorem, we can therefore ensure that the second term in the right hand side is as close to 1 as desired, say at least $0.9$. Hence 
$$
\P^*_z \otimes P_z [ \nu (\tau_r) < t ; (C1)'] \ge 0.9\  \P^*_z [ (C1)'] 
$$
Plugging  this into \eqref{eq:C1tau_r} and recalling that $p_1 =0.1$, we deduce 
$$
\P^*_z \otimes P_z [ \nu (\tau_\Omega) < t ; (C1)'] \ge 0.8 \  \P^*_z [ (C1)'] .
$$
In turn, plugging this into \eqref{eq:GtC1}, and recalling that $p_0 \ge 0.9$, we get 
$$
\P^*_z \otimes P_z [ \nu (\tau_\Omega) < t ; G_t] \ge 0.7 \  \P^*_z[ (C1)'] \ge  0.7 \ \P^*_z [ G_t].
$$
Let us comment quickly on the order in which we need to choose constants in order to make this argument work: we first pick $N$ large enough that \eqref{eq:Beuerling's estimate} holds with $p_1 = 0.1$. Then we choose $M$ large enough so that \eqref{eq:MN} holds and \eqref{eq:p0} holds with $p_0=0.9$.

Altogether, this proves \eqref{eq:1st_trunc_goal} with $c = 0.7$, and thus concludes the proof of Lemma \ref{lem:modified expectation}.
\end{proof}

We remark that the constants $M$ and $N$ were chosen in a specific way in order to make the above proof work;  the choice of these constants to define $G_t$ (and thus the restricted heat content $\Kb_{\Omega, z_i}(t)$) is now fixed and will not be changed when computing the second moment of the restricted heat content in Proposition \ref{thm:2moment} below.  

\begin{prop} \label{thm:2moment}
There exists $\epsilon (t) \to 0$ such that the following holds for all $R(t) \to 0$. Consider the restricted heat contents $\Kb_{\Omega,z_i}(t), (1 \le i \le N(t))$, defined above. Then for all $z_i\in \partial_{R(t)}$, 
    \begin{align*}
        \EE[\Kb_{\Omega,z_i}(t)^2]\le t^{2\Delta+\epsilon(t)} R(t)^{4(1-\mathsf{x}) - \alpha_0^2/2}\ \ \text{as} \ \ t\to0,
    \end{align*}
uniformly in $z_i \in \partial_{R(t)}$, and  where $\Delta$ is as in Proposition \ref{prop:first moment Heat content fractal boundary}, 
\end{prop}

The dependence on $R(t)$, in particular the presence of the power $-\alpha_0^2/2$ in the exponent (which represents a loss compared to the square of the first moment in \eqref{eq:1stmoment_strong}), plays an important role in subsequent aspects of the proof of Theorem \ref{T:heat}.

\begin{proof}
    The lower bound is obvious from Jensen's inequality and Lemma \ref{lem:modified expectation}. The hard part is to prove the corresponding upper bound. Using Girsanov again, we get
    \begin{align*}
        \EE[ \Kb_{\Omega,z_i}(t)^2]= \displaystyle\iint_{(\Omega^t \cap B_{R(t)} (z_i))^2} \EE^{**}\Bigl[\Pb_z[\sigma_\Omega<t]\Pb_w[\sigma_\Omega<t]\textbf{1}_{G_t(z)}\textbf{1}_{G_t(w)}\Bigr]\vert z-w\vert^{-\gamma^2}dzdw,
    \end{align*}
    where under the tilted measure $\PP^{**}$ we have $h=h'+\gamma G(z,\cdot)+\gamma G(w,\cdot)$ with $h'$ being a whole plane Gaussian free field, normalised to have zero unit circle average. (Indeed, recall that on $B_1(0)$ we have $G(z,w)=-\log\vert z-w\vert$ and $\Omega \subset B(0,1/2)$). To make the computation easier, we split the integral in the variable $w$ in the following way: we fix $z\in \Omega^t \cap B_{R(t)}(z_i)$ and define for $n\in\NN$
    \begin{align}
        A_n(z):=\Bigl\{w\in \Omega^t \cap B_R(z_i):\vert z-w\vert\in[e^{-(n+1)},e^{-n})\Bigr\}.
    \end{align}
    Thus, we get, bounding both $\Pb_z[\sigma_\Omega<t],\Pb_w[\sigma_\Omega<t]$ crudely by $1$,
    \begin{align} \label{eq:proof of second moment}
        \EE[\Kb_{\Omega,z_i}(t)^2]\lesssim\int_{\Omega^t\cap B_{R(t)(z_i)}}\sum_{n\ge n(t)}\int_{A_n(z)}\EE^{**}\Bigl[\textbf{1}_{G_t(z)}\textbf{1}_{G_t(w)}\Bigr]e^{n\gamma^2}dwdz
    \end{align}
    where $n(t)$ is such that $e^{ - n(t)} \asymp R(t)$. 
    The key step is to control the probability of the good events. To this end, we reformulate the conditions \eqref{conditions} once more; now in the presence of two logarithmic insertions rather than a single one.

    The approach is as follows. Let us write $r_0:=t^{\kappa(\alpha_0)}>0$. We decompose the double integral into two regimes.
    \begin{itemize}
        \item First we look at the \textbf{mesoscopic case}, where $\vert z-w\vert\ge 100r_0$. In this regime the two balls $B_{r_0}(z),B_{r_0}(w)$ are disjoint. The objective is to show that, at bigger separations, the restriction of the GFF to the two balls are roughly independent, so the $\P^{**}$-probability of $G_t(z)$ and $G_t(w)$ (which concern the behaviour of the GFF inside balls of radius roughly $r_0$ around both points) factorises, yielding the correct exponent. As the distance between the two points gets closer to the bound $100 r_0$, the contribution to the integral coming from points at a given distance depends in a delicate way on the geometry of the domain near the boundary, in particular the size of balls of a given radius. This can be controlled by using Assumption \ref{Assumption 2}.
        
        \item The second case is the \textbf{microscopic} one, where $\vert z-w\vert<100r_0$. Here the two balls $B_{r_0}(z),B_{r_0}(w)$ may overlap, and the two logarithmic insertions from the Girsanov transform are close, yielding (almost) a $2\gamma$-log singularity. This proximity makes it very unlikely for both events $G_t(z)$ and $G_t(w)$ to hold - this is for similar reasons as in the GMC case (\cite{Ber2017}).
    \end{itemize}
    This mesoscopic/microscopic splitting corresponds to partitioning the sum over $n$ into $n\le N_t:=\lfloor\log(1/r_0)-\log(100)\rfloor -1$ and $n>N_t$. Let us first comment on the size of $A_n(z)$. If $z \in \Omega^t$, there exists $\tilde z\in\partial\Omega$ with $d(\tilde z,z)<e^{-n}$. So we can estimate $A_n(z)$ using Assumption \ref{assumption 1}:
    \begin{align} \label{eq:estimate on A_n mezoscopic case}
        \vert A_n(z)\vert\le\vert \Omega_{t^{\kappa(\alpha_0)}}\cap B_{2e^{-n}}(\tilde z)\vert\le 4C_\partial t^{2\mathsf{x}\kappa(\alpha_0)}e^{-2n(1-\mathsf{x})}.
    \end{align}
    This will be useful in the mesoscopic case. On the other hand, in the microscopic case we will simply use the trivial bound $\vert A_n(z)\vert\le\vert B_{e^{-n}}(z)\vert=\pi e^{-2n}$. 
    
    Now let us start with the estimate of \eqref{eq:proof of second moment}. We begin with the mesoscopic case.\\
    
    \textbf{Mesoscopic case.}
    Let $n \le N_t$, and suppose $w\in A_n(z)$. Under $\P^{**}$, we may write by the Girsanov--Cameron--Martin lemma:
    \begin{align} \label{eq:h with two singularities}
        h=h'+\gamma G(z,\cdot)+\gamma G(w,\cdot),
    \end{align}
    where $h'$ is a full plane GFF, normalised to have zero unit circle average. 
    Let $s_z=\log(1/Nd(z,\partial\Omega))$ and $s_w=\log(1/Nd(w,\partial\Omega))$ (where $N$ was chosen in \eqref{eq:Beuerling's estimate}), and let $X_{s_z}':=h'_{e^{-s_z}}(z)$ and $Y_{s_w}':=h'_{e^{-s_w}}(w)$. 
    So $(X_{s_z}',Y'_{s_w})$ is a centered Gaussian vector, with respective variances and covariance given by
    \begin{align} \label{eq:variance and covariance of (X',Y')}
    \sigma_z^2:=\EE[(X_{s_z}')^2]=s_z+O(1),\quad\sigma_w^2:=\EE[(Y_{s_w}')^2]=s_w+O(1),\quad \rho:=\EE[X_{s_z}'Y_{s_w}']=n+O(1),
    \end{align}
    where all $O(1)$ terms are bounded by a constant $C>0$ which is uniform in $z,w,n,t$ and $z_i$. (Note that the estimate on the covariance relies on the fact that we are in the mesoscopic regime).
    
    Because of \eqref{eq:proof of second moment}, our goal is to estimate $\P^{**} [G_t(z), G_t(w)].$ In fact we will ignore the condition $(C2)$ and simply estimate these events using the conditions $(C1)$. Note that the conditions $(C1)(z) \cap (C1)(w)$ imply
    \begin{align*}
        X_{s_z}'&\in\Bigl[(\alpha_0-\gamma)s_z-\gamma n-C,(\alpha_0-\gamma)s_z-\gamma n+C\Bigr],\\
        Y_{s_w}'&\in\Bigl[(\alpha_0-\gamma)s_w-\gamma n-C,(\alpha_0-\gamma)s_w-\gamma n+C\Bigr],
    \end{align*}
    where $C=2M$ (and where $M$ was chosen in the proof of Lemma \ref{lem:modified expectation}, and hence is uniform in $z,w,n$ and $t$). Note that $s_z = \log ( 1/ r_0) + O(1)$, and analogously for $s_w$. For simplicity, we write $\lambda = \lambda_n:=n/(\log1/r_0)$ (note that $\lambda \in [0,1]$ since $n \le N_t$). Thus, it suffices to compute the joint density of $(X'_{s_z}, Y'_{s_w})$ in the box
    \begin{align} \label{eq:box}
        \Bigl[(\alpha_0-\gamma(1+\lambda))\log(1/r_0)-C,(\alpha_0-\gamma(1+\lambda))\log(1/r_0)+C\Bigr]^2,
    \end{align}
    for some $C>0$ which is uniform in $z,w$ and $t$. To describe the joint law of $(X'_{s_z}, Y'_{s_w})$ we use the fact that we can explicitly describe the conditional law of $Y'_{s_w}$ given $X'_{s_w}$, from which it is not hard to see that there exists an independent centered normal random variable $Z$ such that
    \begin{align*}
        \left(X'_{s_z}, Y'_{s_w}\right)\overset{d}{=}\left(X'_{s_z},\frac{\rho}{\sigma_z^2}X'_{s_z}+Z\right).
    \end{align*}
    To make this work, by \eqref{eq:variance and covariance of (X',Y')} we have to choose
    \begin{align}
        \var(Z)&=\sigma_w^2-\frac{\rho^2}{\sigma_z^2}=\log\frac{1}{r_0}-\frac{n^2}{\log1/r_0}+O(1)
        =\log\frac{1}{r_0}(1-\lambda^2+o(1))\nonumber \\
        &\in\Bigl[\log\frac{1}{r_0}\Bigl(1-\lambda^2\frac{\log1/r_0}{\log1/r_0-C}\Bigr)-C,\log\frac{1}{r_0}\Bigl(1-\lambda^2\frac{\log1/r_0}{\log1/r_0+C}\Bigr)+C\Bigr]\label{eq:variance of Z}
    \end{align}
    For some $C>0$, which is uniform in $z,w,n$ and $z_i$. Observe also that
    \begin{align}\label{eq:rho over sigma^2}
    \frac{\rho}{\sigma_z^2} = \lambda + o(1)\in\Bigl[\lambda\frac{\log1/r_0}{\log1/r_0+C}-C,\lambda\frac{\log1/r_0}{\log1/r_0-C}+C\Bigr],
    \end{align}
    again for some $C>0$, which is uniform in $z,w,n$ and $z_i$. Hence, conditionally given
    \begin{align} \label{eq:mesoscopic case (C1) for z}
        X'_{s_z}\in\Bigl[(\alpha_0-\gamma(1+\lambda))\log(1/r_0)-C,(\alpha_0-\gamma(1+\lambda))\log(1/r_0)+C\Bigr],
    \end{align}
    by \eqref{eq:box}, the condition for $Y'_{s_w}$ can be equivalently reformulated in terms of $Z$ as
    \begin{align*}
        Z\in(1-\lambda(1 + o(1)))\biggl[(\alpha_0-\gamma\Bigl(1+\lambda(1+o(1))\Bigr)\log(1/r_0),(\alpha_0-\gamma\Bigl(1+\lambda(1+o(1))\Bigr)\log(1/r_0)\biggr],
    \end{align*}
    for an $o(1)$-term, which is of the same form as in \eqref{eq:variance of Z} and \eqref{eq:rho over sigma^2}. In particular the $o(1)$-term is uniform in $z_i\in\partial_{R(t)}$ and uniform in the choice of $R(t)$. As $(X'_{s_z},Z)$ is a Gaussian vector with independent entries, we get by the estimates before
    \begin{align*}
        \P^{**}\Bigl[{G_t(z)};{G_t(w)}\Bigr]\le r_0^{\frac{(\gamma(1+\lambda)-\alpha_0)^2}{2}}r_0^{\frac{(1-\lambda)^2(\gamma(1+\lambda)-\alpha_0)^2}{2(1-\lambda^2)}+\epsilon(t)},
    \end{align*}
    where the $\epsilon(t)$ is uniform in $z,w,n$ and $z_i$ and $\epsilon(t)\to0$ when $t\to0$. Using \eqref{eq:proof of second moment}, \eqref{eq:estimate on A_n mezoscopic case} and $e^{-n}=r_0^{\lambda}$, we conclude
    \begin{align*}
    \int_{\Omega^t\cap B_{R(t)}(z_i)}&\sum_{n= n(t)}^{N_t}\int_{A_n(z)}\PP^{**}\Bigl[{G_t(z)};{G_t(w)}\Bigr]e^{n\gamma^2}dwdz\\
    &\le \int_{\Omega^t\cap B_{R(t)}(z_i)}\sum_{n= n(t)}^{N_t}\int_{A_n(z)}\PP^{**}\Bigl[G_t(z)\cap G_t(w)\Bigr]e^{n\gamma^2}dwdz\\
    &\le\int_{\Omega^t\cap B_{R(t)}(z_i)}\sum_{n= n(t)}^{N_t}\int_{A_n(z)}r_0^{\frac{(\gamma(1+\lambda)-\alpha_0)^2}{2}}r_0^{\frac{(1-\lambda)^2(\gamma(1+\lambda)-\alpha_0)^2}{2(1-\lambda^2)}+\epsilon(t)}e^{n\gamma^2}dwdz\\
    &\le C_\partial  \sum_{n\ge n(t)}\int_{\Omega^t\cap B_{R(t)}(z_i)}r_0^{\frac{(\gamma(1+\lambda)-\alpha_0)^2}{2}+\frac{(1-\lambda)^2(\gamma(1+\lambda)-\alpha_0)^2}{2(1-\lambda^2)}+\epsilon(t)}e^{n\gamma^2}e^{-2n(1-\mathsf{x})}r_0^{2\mathsf{x}}dz\\
    &\le (C_\partial)^2 \sum_{n\ge n(t)}r_0^{\frac{(\gamma(1+\lambda)-\alpha_0)^2}{2}+\frac{(1-\lambda)^2(\gamma(1+\lambda)-\alpha_0)^2}{2(1-\lambda^2)}+\epsilon(t)}e^{n\gamma^2}e^{-2n(1-\mathsf{x})}r_0^{4\mathsf{x}} R(t)^{2(1-\mathsf{x})}\\
    &=r_0^{4\mathsf{x}+\epsilon(t)}R(t)^{2(1-\mathsf{x})}\sum_{n\ge n(t)}r_0^{f(\lambda_n)},
    \end{align*}
    where the $o(1)$-term is uniform in $z,w$ and $n$ and
    \begin{align*}
        f(\lambda):=\frac{(\gamma(1+\lambda)-\alpha_0)^2}{2}+\frac{(1-\lambda)^2(\gamma(1+\lambda)-\alpha_0)^2}{2(1-\lambda^2)}+\lambda(2(1-\mathsf{x})-\gamma^2).
    \end{align*}
    We estimate $f(\lambda)$ naively: 
    \begin{align*}
        f(\lambda)&=\frac{(\gamma(1+\lambda)-\alpha_0)^2}{1+\lambda}+\lambda(2(1-\mathsf{x})-\gamma^2)\\
        &=(\gamma-\alpha_0)^2+\lambda\gamma^2-\frac{\lambda}{1+\lambda}\alpha_0^2+\lambda(2(1-\mathsf{x})-\gamma^2)\\
        &=(\gamma-\alpha_0)^2+2\lambda(1-\mathsf{x})-\frac{\lambda}{1+\lambda}\alpha_0^2\\
        &\ge(\gamma-\alpha_0)^2+\lambda\Bigl(2(1-\mathsf{x})-{\frac{\alpha_0^2}{2}}\Bigr).
    \end{align*}
    Plugging that into our estimate of the mesoscopic part and using $\lambda_n=n/(\log1/r_0)$ results in
    \begin{align*}
        \int_{\Omega^t\cap B_{R(t)}(z_i)}\sum_{n=n(t)}^{N_t}\int_{A_n(z)}&\P^{**}\Bigl[{G_t(z)};{G_t(w)}\Bigr]e^{n\gamma^2}dwdz\\&\le t^{4\mathsf{x}\kappa(\alpha_0)+\epsilon(t)} R(t)^{2 (1-\mathsf{x})}\sum_{n\ge n(t)}t^{\kappa(\alpha_0)f(\lambda_n)}\\
        &\le t^{\kappa(\alpha_0)(4\mathsf{x}+(\gamma-\alpha_0)^2)+\epsilon(t)} R(t)^{2 (1-\mathsf{x})} \sum_{n\ge n(t)}t^{\kappa(\alpha_0)\lambda_n(2(1-\mathsf{x})-\alpha_0^2/2)}\\
        &=t^{2\Delta+\epsilon(t)} R(t)^{2 (1-\mathsf{x})} \sum_{n\ge n(t)}e^ {-n [ (2(1-\mathsf{x}) - \alpha_0^2/2]}\\
        &=t^{2\Delta+\epsilon(t)} R(t)^{4(1-\mathsf{x}) - \alpha_0^2/2} .
    \end{align*}
    This concludes the first part of the argument. \\

    \textbf{Microscopic case.} When $B_{r_z}(z)$ and $B_{r_w}(w)$ may overlap, the joint covariance structure of the circle-average processes $X'_{s_z}$ and $Y'_{s_w}$ becomes pointless. We therefore discard the constraint at $w$ and impose only a single good event for $z$. Heuristically, as $z$ and $w$ approach each other faster than the radius shrinks, the configuration behaves like a single ball with a double logarithmic insertion at (nearly) the same center. This substantially increases the cost of enforcing the circle average to be no more than $\alpha_0\log (1/r_0)$ (in fact what will be useful here more specifically will be condition $(C2)$). 
    This can be used to balance the Girsanov factor $\vert z-w\vert^{-\gamma^2}$ which otherwise leads to a blowup on the diagonal when $\gamma\in[\sqrt{2},2)$. 

    Let us start commenting on the event $(C1)\cap(C2)$ for some $z\in \Omega^t$ and $w\in A_n(z)$ such that $|z-w|<100t^{\kappa(\alpha_0)}$, i.e. $n>N_t$. As before, we can rephrase this in terms of the Brownian motion $X'_s:=h'_{e^{-s}}(z)$. Since $|z-w|\asymp e^{-n}$, this means that $s\mapsto h_{e^{-s}}(z)$ is a Brownian motion with two additional drifts. The first one is given by $\gamma s$  and comes from the $\gamma-\log$-singularity at $z$. The second one is  coming from the $\gamma-\log$-singularity at $w$, and is thus given by $\gamma s$ so long as $s <n$, whereas it is roughly zero when $s>n$. 
    
    
    Our strategy will thus be to estimate from above $G_t(z) \cap G_t(w)$ by $(C1)(z)$ and $(C2)(z)$ where we only enforce the condition $(C2)(z)$ at a scale corresponding to the separation between $z$ and $w$. 

In fact, we will distinguish an additional ``intermediate regime'', where $r_0 /2 \le |z-w| \le 100 r_0$. 
In this case, essentially $X'_s$ has drift $2\gamma$ until roughly $s=s_z$, so it will in fact suffice to only use $(C1)(z)$.
Reasoning as in the mesoscopic case (see \eqref{eq:mesoscopic case (C1) for z} with $\lambda=1$), but using only the condition $(C1)(z)$, we get
    \begin{align} \label{eq:microscopic case (C1) estimate}
        \P^{**}\Bigl[{G_t(z)};{G_t(w)}\Bigr]\le Cr_0^{\frac{(2\gamma-\alpha_0)^2}{2}},
    \end{align}
    where as before $r_0=t^{\kappa(\alpha_0)}$ and $C>0$ is uniform in $z,w$ and $z_i$. Thus, we estimate the contribution of the part $r_0/2\le\vert z-w\vert\le 100r_0$ by
    \begin{align}
        \int_{\Omega^t\cap B_{R(t)}(z_0)}&\int_{r_0/2\le\vert z-w\vert\le100r_0}\EE^{**}\Bigl[\textbf{1}_{G_t(z)}\textbf{1}_{G_t(w)}\Bigr]\vert z-w\vert^{-\gamma^2}dwdz \nonumber \\
        &\le C\int_{\Omega^t\cap B_{R(t)}(z_0)}\int_{r_0/2\le\vert z-w\vert\le100r_0}r_0^{\frac{(2\gamma-\alpha_0)^2}{2}}r_0^{-\gamma^2}dwdz \nonumber \\
        &\le C\int_{\Omega^t\cap B_{R(t)}(z_0)}r_0^{\frac{(2\gamma-\alpha_0)^2}{2}+2-\gamma^2}dz\\
        &\le C r_0^{\frac{(2\gamma-\alpha_0)^2}{2}+2-\gamma^2+2\mathsf{x}}. \label{eq:intermediate}
    \end{align}
    We want to show that this gives a larger exponent than $2\Delta/\kappa(\alpha_0)$. To prove this we rely on the following two elementary inequalities: 
\begin{equation}\label{eq:controlalpha0}
\begin{cases}
       \frac{\alpha_0^2}{2} < 2 (1- \sqrt{\mathsf{x}})^2\\
        2 (1- \sqrt{\mathsf{x}})^2 < 2(1-\mathsf{x}).
       \end{cases} 
\end{equation}
The second inequality is easier to check: we cancel the factors of two and take the square root of both sides (since they are positive) and we need to check that $ 1- \sqrt{x} < \sqrt{1-x}$, or, equivalently, $1 < \sqrt{x} + \sqrt{1-x}$, which is straightforward to check after squaring both sides.

For the first of the inequalities we observe that since $\alpha_0 = Q - \sqrt{Q^2 - 4 (1-\mathsf{x})}$ from \eqref{eq:alpha0}, 
$\alpha_0$ is maximized at $\gamma=2$ by monotonicity 
(as it is monotone decreasing in $Q$).
Hence $\alpha_0\le 2-\sqrt{4-4(1-x)}=2(1-\sqrt{x})$. Using that, we get
    \begin{align*}
        \frac{(2\gamma-\alpha_0)^2}{2}+2-\gamma^2+2\mathsf{x}> 4\mathsf{x}+(\alpha_0-\gamma)^2 = 2\Delta/\kappa(\alpha_0).
    \end{align*}
    Therefore, the contribution of the "intermediate regime" $r_0/2\le\vert z-w\vert\le 100r_0$ is negligible: that is, the right hand side of \eqref{eq:intermediate} is negligible compared to $t^{2\Delta}$. 
    
    Thus, it remains to treat the case $\vert z-w\vert\le r_0/2$. Fix $z,w$ such that $|z-w| \le r_0/2$, $w \in A_n(z)$. Then if $s \le n-1$, we have
     by \eqref{eq:h with two singularities}
    \begin{align} \label{eq:microscopic part}
        \Big| h_{e^{-s}}(z)- h'_{e^{-s}}(z)-2\gamma s\Big| \le C ,
    \end{align}
    for some explicit constant $C>0$. In other words, for points $y$ at distance $e^{-s}$ from $z$, $G(z,y) = s$, and $|G(w,y) - s| \le C$. In particular, applying this to $s = s_z \le n -1$, the condition $(C1)(z)$ implies
    \begin{align} \label{eq:microscopic part C1"}
        (C1)''&\qquad X'_{s_z}\in[(\alpha_0-2\gamma+o(1))s_z-C,(\alpha_0-2\gamma+o(1))s_z+C],
    \end{align}
    where the $o(1)$-terms vanishes when $t\to0$ and are uniform in $z,w,n$.


    
     We now take $s = n-1$ in \eqref{eq:microscopic part} and in condition $(C2)$ from \eqref{conditions} to obtain a condition on the Brownian motion $X'_s = h'_{e^{-s}}(z)$. However, it is more natural to phrase that condition in terms of the parameter $q\ge 0 $ defining Condition $(C2)$ in \eqref{conditions}, i.e., $(1+q) s_z = s$.

    \begin{align} \label{eq:microscopic part C2"}
        (C2)''\qquad X'_{s} & \le (\alpha_0 - \beta) \log (1/r_0) + s ( \beta - 2\gamma) + C \nonumber 
        \\ 
        & = (\alpha_0-2\gamma+o(1))s_z + (\beta-2\gamma+o(1))qs_z+C,
    \end{align}
    where $C>0$ is uniform in $z,w,n$ and $t$ and the $o(1)$-terms vanishes when $t\to0$ and is also uniform in $z,w,n$. So conditional on \eqref{eq:microscopic part C1"}, \eqref{eq:microscopic part C2"} implies
    \begin{align*}
        X'_{(1+q)s_z}-X'_{s_z}\le(\beta-2\gamma+o(1))qs_z+C,
    \end{align*}
    where the $o(1)$-terms vanishes when $t\to0$ and is uniform in $z,w,n$. Now, the left hand side is just an increment of a standard Brownian motion over an interval of duration $qs_z$. As $s_z=(1+o(1))\log (1/r_0)$, we have
    \begin{align*}
        \PP^{**}[(C2)\vert(C1)]\le \exp\Bigl(-\frac{(\beta-2\gamma+o(1))^2q^2(\log1/r_0)^2+o(1)}{2q\log1/r_0}\Bigr)=r_0^{\frac{(\beta-2\gamma)^2q}{2}+o(1)},
    \end{align*}
    where all $o(1)$-terms are uniform in $z,w,n$. As the probability of $(C1)''$ can be estimated as in \eqref{eq:microscopic case (C1) estimate}, we get in total
    \begin{align*}
        \PP^{**}\Bigl[(C1)\cap (C2)\Bigr]
        &\le r_0^{\frac{(2\gamma-\alpha_0)^2}{2}+o(1)}r_0^{\frac{(2\gamma-\beta)^2q}{2}+o(1)}.
    \end{align*}
    Now we are able to compute the contribution of the microscopic part. Using $$\vert A_n(z)\vert\le \vert B_{e^{-n}(z)}\vert\le \pi s_z^{2(1+q)} \le C r_0^{2 (1+q)},$$
    (recall that $q = q_n$ depends implicitly on $n$, since $s_z (1+q) = s= n-1$), 
    we get:
    \begin{align*}
        &\displaystyle\iint_{z,w \in \Omega^t:|z-w|< r_0/2}\P^{**}\Bigl[{G_t(z)};{G_t(w)}\Bigr]\vert z-w\vert^{-\gamma^2}dzdw\\
        &\le C\int_{\Omega^t}\sum_{n=\log1/r_0-O(1)}^{\infty}\PP^{**}\Bigl[(C1)(z)\cap(C2)(z)\Bigr]\vert A_n(z)\vert r_0^{-(1+q)\gamma^2}dz\\
        &\le\int_{\Omega^t}\sum_{n=0}^{\infty}r_0^{\frac{(2\gamma-\alpha_0)^2}{2}+o(1)}r_0^{\frac{(2\gamma-\beta)^2q}{2}+o(1)}r_0^{(2-\gamma^2)(1+q)}dz\\
        &\le \sum_{n=0}^{\infty}r_0^{2\mathsf{x}+\frac{(2\gamma-\alpha_0)^2}{2}+\frac{(2\gamma-\beta)^2q}{2}+(2-\gamma^2)(1+q)+o(1)}=\sum_{n=0}^{\infty}r_0^{f(q_n)+o(1)},
    \end{align*}
    where all $o(1)$-terms are uniform in $z,w,n$ and
    \begin{align*}
        f(q):=2\mathsf{x}+\frac{(2\gamma-\alpha_0)^2}{2}+\frac{(2\gamma-\beta)^2q}{2}+(2-\gamma^2)(1+q).
    \end{align*}
     One can easily check that $\alpha_0^2/2<2-2\mathsf{x}$.  To do so note that $\alpha_0$ is maximized at $\gamma=2$ which gives $2(1-\sqrt{\mathsf{x}})$ and then use $2(1-\sqrt{\mathsf{x}})^2<2-2\mathsf{x}$. Choose $\beta>\gamma$ so close to $\gamma$ and $\eps>0$ so small such that
    \begin{align*}
        2-\gamma^2+\frac{(2\gamma-\beta)^2}{2}> 2-\frac{\gamma^2}{2}-\eps,\qquad 2-\frac{\gamma^2}{2}-2\eps>0,\qquad 2-3\eps&\ge2\mathsf{x}+\frac{\alpha_0^2}{2}.
    \end{align*}
    Then, we get
    \begin{align*}
        2\mathsf{x}+&\frac{(2\gamma-\alpha_0)^2}{2}+\frac{(2\gamma-\beta)^2q}{2}+(2-\gamma^2)(1+q)\\
        &=2\mathsf{x}+\frac{(2\gamma-\alpha_0)^2}{2}-\frac{(2\gamma-\beta)^2}{2}+\Bigl(2-\gamma^2+\frac{(2\gamma-\beta)^2}{2}\Bigr)(q+1)\\
        &\ge2\mathsf{x}+\frac{(2\gamma-\alpha_0)^2}{2}-\frac{(2\gamma-\beta)^2}{2}+2-\frac{\gamma^2}{2}-2\eps+(1+q)\eps\\
        &\ge4\mathsf{x}+\gamma^2-2\gamma\alpha_0+\eps+\alpha_0^2+(1+q)\eps\\
        &=4\mathsf{x}+(\gamma-\alpha_0)^2+\eps+(1+q)\eps.
    \end{align*}
    As $r_0^{(1+q_n)\eps}=e^{-\eps n(1+o(1))}$, where the $o(1)$-term is uniform in $z,w$ and $n$, we get for $t>0$ small enough
    \begin{align*}
        \sum_{n=0}^{\infty}r_0^{f(q_n)+o(1)}\le\sum_{n=0}^{\infty}r_0^{4\mathsf{x}+(\gamma-\alpha_0)^2+\eps}e^{-\eps n(1+o(1))}=t^{2\Delta+\eps}\sum_{n=0}^{\infty}e^{-\eps n(1+o(1))}\le t^{2\Delta+\eps/2},
    \end{align*}
    which is negligible compared to the contribution coming from the mesoscopic part and thus finishes the proof.
\end{proof}


\subsection{Proof of Theorem \ref{T:heat}} \label{sec:2.3}

As already explained, by this point we know from Lemma \ref{lem:modified expectation}, Proposition \ref{thm:2moment} and the Paley--Zygmund inequality that 
with probability at least $t^{o(1)}$, $\Kb_{\Omega,z_i}(t)$ is bounded from below by $t^{\Delta + o(1)}$, whenever $z_i \in \partial_{R(t)}$. We need to find an argument to boost this probability (which could still tend to zero) to a probability of order 1. This is achieved using an argument with the flavour of a $0$-$1$ law.

More precisely, by Lemma \ref{lem:modified expectation}, Proposition \ref{thm:2moment} and Paley--Zygmund inequality, we have
\begin{align} \label{eq:paley-zygmund}
    a(t):=\PP\Bigr[\Kb_{\Omega, z_i}(t)\ge \E[\Kb_{\Omega, z_i}(t)]\Bigl]\ge t^{\epsilon(t)} R(t)^{\alpha_0^2/2},
\end{align}
where $\epsilon(t) \to 0$ as $t\to 0$, uniformly in $x_i \in \partial_{R(t)}$, and, crucially, \emph{uniformly in the choice of $R(t)$}.


Let us write $N(t) = \#\partial_{R(t)}$. By assumption \eqref{eq:assumption_lowerbound_local}, $N(t) \ge R(t)^{-2(1-\mathsf{x})+\delta}$ for all small enough $\delta$. By considering a smaller subset if necessary, we may assume without loss of generality that $N(t) \le 2 R(t)^{-2(1-\mathsf{x})+\delta}$ and thus $N(t) \to \infty$ STP.

\textbf{Almost independence.} First, we show that all the $\Kb_{\Omega, z_i}(t)$ are almost independent, by comparing their values to those associated with fields $h_i$ where we turn off the common harmonic parts in balls of radius $4R(t)$ around each $z_i$. 

More precisely, we use the Markov property on the union of the $N(t)$ disjoint balls $B_{4R(t)}(z_1),...,B_{4R(t)}(z_{N(t)})$. We may thus write:
\begin{align*}
    h=h^0+\displaystyle\sum_{i=1}^N h_i,
\end{align*}
where these terms are independent of one another, $h^0$ is harmonic on $\cup_{i=1}^{N(t)} B_{4R(t)}(z_i)$, and every $h_i$ is a Dirichlet GFF on $B_{4R(t)}(z_i)$. 
We first control the fluctuations and the overall value of the harmonic parts in each ball, showing it is not too large.

\begin{lem}\label{lem:G}
    We define the good event $G:=G_1\cap G_2$, where
\begin{align*}
    G_1&:=\Bigl\{\displaystyle\max_{1 \le i \le N(t)}\displaystyle\sup_{z\in B_{2R(t)}(z_i)}\vert h^0(z)-h^0(z_i)\vert\le\sqrt{\log(1/t)}\Bigr\},\\
    G_2&:=\Bigl\{\displaystyle\max_{1\le i \le N(t)}\vert h^0_{4R(t)}(z_i)\vert\le 10 \log(N(t))\Bigr\}.
\end{align*}
Then for some constant $C>0$, $\P[G^c]\le CN(t)^{-49} \to 0$ as $t\to0$.
\end{lem}
\begin{proof}
Let us start showing that $G_1$ has probability tending to $1$. By Lemma \ref{lem:decaymean} we get uniformly on all balls
\begin{align*}
    \PP\Bigl[\displaystyle\sup_{z\in B_{2R(t)}(z_i)}\vert h^0(z)-h^0(z_i)\vert\ge \lambda\Bigr]\le ce^{-c'\lambda^2},
\end{align*}
for $\lambda>1$ and constants $c,c'>0$. Using that, we conclude
\begin{align*}
    &\PP[G^c]=\PP\Bigl[\displaystyle\max_{1\le i \le N(t)}\displaystyle\sup_{z\in B_{2R(t)}(z_i)}\vert h^0(z)-h^0(z_i)\vert> \sqrt{\log (1/t)}\Bigr]\\
    &\le\displaystyle\sum_{i=1}^{N(t)}\PP\Bigl[\displaystyle\sup_{z\in B_{2R(t)}(z_i)}\vert h^0(z)-h^0(z_i)\vert>\sqrt{\log (1/t)}\Bigr]\\
    &\le cN(t)e^{-c'\log (1/t)}=N(t)t^{c'}\rightarrow0,
\end{align*}
as $t \to 0$, since $N(t)$ diverges to infinity slower than any polynomial in $t$. Hence $\P( G_1) \to 1$. 

Next, we need to control the circle averages $h^0_{4R(t)}(z_i)$. Using that $h^0_{4R(t)}(z_i)=h_{4R(t)}(z_i)$ is a centered Gaussian with variance $-\log(4R(t))$, we get
\begin{align*}
    \PP[G_2^c]&=\PP\Bigl[\displaystyle\max_{1\le i \le N(t)}\vert h^0(z)_{4R(t)}(\omega)\vert>10 \log(N(t))\Bigr]\\
    &\le\sum_{i=1}^{N(t)}\PP\Bigl[\vert h^0_{4R(t)}(z_i)\vert>10\log(N(t))\Bigr]\\
    &\le CN(t)\exp\Bigl(-\frac{100\log(N(t))^2}{2\log(\frac{1}{4R(t)})}\Bigr)\\
    &\le CN(t)\exp\Bigl(-50\log(N(t))\Bigr)\\
    &=CN(t)^{-49}\to 0,\qquad\textnormal{ as }t\to0.
\end{align*}
Thus, $\P(G_2) \to 1$ and the lemma is proved.
\end{proof}

We now want to compare the restricted heat content $\Kb_{\Omega, z_i}(t)$ to a version depending only on  $h_i$. We define
\begin{align} \label{eq:definition local independent version}
    \Kb_i^{h_i}(t):=\int_{\Omega^t\cap B_{R(t)}(z_i)}\Pb_z[\sigma^i_\Omega<t]\textbf{1}_{G_t}\mu_\gamma^{h_i}(dz),
\end{align}
where $\sigma^i_\Omega$ is the first hitting time of $\partial\Omega\cup\partial B_{2R(t)}(z_i)$ by a LBM associated to the Dirichlet GFF $h_i$ on $B_{4R(t)}(z_i)$ and $\mu_\gamma^{h_i}$ is the Liouville area measure associated to the Dirichlet GFF $h_i$ on $B_{4R(t)}(z_i)$. As $h_i$ are all independent, also the $\Kb_i^{h_i}(t)$ are independent of each other. 

\begin{lem}\label{lem:comparison_local}
    Let $G = G_1\cap G_2$ be as in Lemma \ref{lem:G}. There exists $\epsilon(t)\to 0$ such that the following holds for all choice of $R(t) \to 0$. 
    There exist random variables $(X_i(t))_{1\le i \le N(t)}$ such that $X_i(t) \ge 0 $ for all $1\le i \le N(t)$ and $\E( X_i(t) \mathbf{1}_G) \to 0$ FTP (uniformly over $z_i \in \partial_{R(t)}$). Furthermore, on the event $G$,
    \begin{align*}
    \Kb_{\Omega, z_i}(t)\ge \Kb_i^{h_i}(t^{1+\epsilon(t) })t^{\epsilon(t)} - X_i,
\end{align*}
where the $o(1)$-term is uniform in $i$ (independently of $R(t)$) and tends to 0 as $t \to 0$.
\end{lem}
\begin{proof}
Let us now have a look at what the good event means for $\Kb_i^{h_i}(t)$. On the event $G_1\cap G_2$ and in $B_{2R(t)}(z_i)$, we have $h=h^0+h_i$, where $\vert h^0\vert\le\vert h^0-h^0(z_i)\vert+\vert h^0(z_i)\vert\le \sqrt{\log(1/t)}+10\log(N(t))$. That means, the harmonic part shifts the Liouville area measure at most by a factor $e^{\gamma(\sqrt{\log1/t}+10\log(N(t)))}=t^{o(1)}$, i.e. for all $A\subseteq B_{2R(t)}(z_i)$
\begin{align*}
    \mu_\gamma(A)=\lim_{\eps\to0}\int_Ae^{\gamma h_\eps(z)}\eps^{\gamma^2/2}dz\le e^{\gamma(\sqrt{\log1/t}+10\log(N(t)))}\mu^{h_i}_\gamma(A)\le t^{o(1)}\mu^{h_i}_\gamma(A),
\end{align*}
Similarly, we get $\mu_\gamma(A)\ge t^{o(1)}\mu^{h_i}_\gamma(A)$. 
By a similar argument, we can deduce that, on $G$, we have
\begin{align*}
    \nu(t)=t^{o(1)}\nu^{h_i}(t),
\end{align*}
where $\nu^{h_i}$ denotes the time change given by \eqref{eq:quantum clock} with respect to the field $h_i$. By Lemma \ref{lem:scaling without}, we get that $\Pb_z[\sigma^i_{B_{2R(t)}(z_i)}<t]$ vanishes FTP in $t$, as $R(t)$ vanishes STP. Thus, we get
\begin{align*}
    \sigma_{\Omega}=t^{o(1)}\sigma^i_\Omega,
\end{align*}
as $t\to0$. Therefore, (still on the event $G$)
\begin{align*}
    \Kb_{\Omega, z_i}(t)&=\int_{\Omega^t\cap B_{R(t)}(z_i)}\Pb_z[\sigma_\Omega<t]\textbf{1}_{G_t}\mu_\gamma(dz)\\
    &=\int_{\Omega^t\cap B_{R(t)}(z_i)}\Pb_z[\sigma^i_\Omega t^{o(1)}<t]\textbf{1}_{G_t}\mu_\gamma(dz) - \int_{\Omega^t\cap B_{R(t)}(z_i)}\Pb_z[\sigma^i_\Omega <t, \sigma_{\Omega}>t]\textbf{1}_{G_t}\mu_\gamma(dz) \\
    &\ge\int_{\Omega^t\cap B_{R(t)}(z_i)}\Pb_z[\sigma^i_\Omega t^{o(1)}<t]\textbf{1}_{G_t}\mu^{h_i}_\gamma(dz)t^{o(1)} -\int_{\Omega^t\cap B_{R(t)}(z_i)}\Pb_z[\sigma_{R(t)}<t ]\mu_\gamma(dz)\\
    &=t^{o(1)}\Kb^{h_i}_i(t^{1+o(1))})  - X_i,
\end{align*}
where 
$$
X_i := \int_{\Omega^t\cap B_{R(t)}(z_i)}\Pb_z[\nu ( \tau_R(t)) <t ]\mu_\gamma(dz).
$$
Taking expectations, we see that 
\begin{align*}
    \E(X_i\mathbf{1}_{G}) & = \int_{\Omega^t\cap B_{R(t)}(z_i)} \E^*_z \Bigl[ \mathbf{1}_G \Pb_z [ \sigma_{R(t)}< t ]\Bigr] dz.   
\end{align*}
Arguing as in Corollary \ref{cor:scaling} (see Lemma \ref{lem:scaling without}), we see that on the event $G$, 
\begin{align*}
\Pb_z [ \sigma_{R(t)}< t ] & \le P_z [ \hat \sigma_1 < t R(t)^{-(2- \gamma^2/2)} e^{\gamma \sqrt{\log (1/t)}} N(t)^{10 \gamma} ]\\
& = P_z [ \hat \sigma_1 < t^{1+ o(1)}] 
\end{align*}
so that, using Lemma \ref{lem:tailbound},
\begin{align*}
    \E^*_z \Bigl[ \mathbf{1}_{G}\Pb_z [ \nu ( \tau_{R(t)} )< t ]\Bigr] & \le \P^*_z\otimes P_z [ \hat \sigma_1 < t^{1+ o(1)}] \le t^{K} 
\end{align*}
for any $K>0$ and for all $t$ sufficiently small, uniformly in $z \in B_{R(t)} (z_i)$ and $z_i \in \partial_{R(t)}$. Therefore, $\E( X_i \mathbf{1}_G) \le t^K$ for all $t$ sufficiently small. 
%
%
\end{proof}

We are now ready to finish the proof of Theorem \ref{T:heat}. 

\begin{proof}[Proof of Theorem \ref{T:heat}]
The upper bound on $\Kb_{\Omega}(t)$ follows from the first moment computation (Proposition \ref{prop:first moment Heat content fractal boundary}) and Markov's inequality. We therefore concentrate on the lower bound. Fix $\eps>0$. Note that by Lemma \ref{lem:comparison_local},
\begin{align*}
\P( \Kb_{\Omega} (t) \le t^{\Delta + \eps} ) & \le  \P ( \Kb_{\Omega,z_i} (t) \le t^{\Delta + \eps} , \text{ for all } 1\le i \le N(t))\\
& \le \P(G^c) + \P[t^{o(1)} \Kb^{h_i}_i(t^{1+ o(1)}) - X_i \le t^{\Delta + \eps} \text{ for all } 1\le i \le N(t) ; G]  .
\end{align*}
Note that by Markov's inequality, $\P( X_i \ge t^{\Delta +\eps} ;G) \le t^K$ for some $K>0$, so that by a union bound, and using the fact that $N(t) t^K \to 0$, 
$$
\P [ \max_{1\le i \le N(t)}  X_i > t^{\Delta + \eps}  ; G] \to 0. 
$$
Hence
\begin{align*}
    \P( \Kb_{\Omega} (t) \le t^{\Delta + \eps} ) & \le o(1) + \P [ \max_{1\le i \le N(t)} \Kb^{h_i}_i(t^{1+ o(1)}) \le 2 t^{\Delta + \eps + o(1)} ] 
\end{align*}
It therefore suffices to show: for all $\eps>0$,
\begin{align}\label{eq:goal_concl_heat}
\P [ \max_{1\le i \le N(t)} \Kb^{h_i}_i(t) \le  t^{\Delta + \eps } ] \to 0.
\end{align}
As all $\Kb_i^{h_i}$ are independent, we have
\begin{align}
\P [ \max_{1\le i \le N(t)} \Kb^{h_i}_i(t) \le  t^{\Delta + \eps } ] & =  \P[\Kb^{h_1}_1(t) \le  t^{\Delta + \eps }]^{N(t)} . 
\label{eq:heat_concl_prod}
\end{align}
On the other hand, by Lemma \ref{lem:comparison_local} again, for each $1\le i \le N(t)$,
\begin{align*}
     \P[\Kb^{h_i}_i(t) \le  t^{\Delta + \eps }] &\le \P(G^c) + \P(\Kb_{\Omega, z_i}(t^{1+ o(1)}) \le t^{\Delta + \eps + o(1)} )\\
     & \le N(t)^{-49}  + 1-t^{\epsilon(t)}R(t)^{\alpha_0^2/2}
\end{align*}
by \eqref{eq:paley-zygmund} and Lemma \ref{lem:G}. Plugging back into \eqref{eq:heat_concl_prod}, and using that $1- x \le e^{-x}$ for any $x \in \R$, we get
\begin{align*}
\P [ \max_{1\le i \le N(t)} \Kb^{h_i}_i(t) \le  t^{\Delta + \eps } ] & \le  [1- R(t)^{\alpha_0^2/2} t^{\epsilon(t)}+ N(t)^{-49} ]^{N(t)}\\
& \le \exp\left( - N(t) t^{\epsilon(t)} R(t)^{\alpha_0^2/2} + N(t)^{-48} \right) \to 0.
\end{align*}
In the last line, to conclude that the right hand side tends to zero, we used the fact that 
\begin{equation}\label{eq:R_N}
R(t)^{\alpha_0^2/2} t^{\epsilon(t)} N(t)  \to \infty.
\end{equation}
Let us justify why this holds. Indeed, given \eqref{eq:controlalpha0}, we have $$R(t)^{\alpha_0^2/2} N(t) t^{\epsilon(t)} \ge R(t)^{-\eta}\cdot  R(t)^{2 (1-\mathsf{x}) } N(t) t^{\epsilon(t)}=R(t)^{-\eta/2}t^{\eps(t)} \to \infty$$ 
for $\eta = 2(1- \mathsf{x}) - \alpha_0^2/2 >0$, and for the conclusion we used the fact that, in Assumption \ref{Assumption 2}, $N(t) \ge R(t)^{-2 ( 1- \mathsf{x}) + \zeta}$ with $\zeta = \eta/2$. (Note that this implicitly fixes the value of $\zeta$ to be used in the definition of $\partial_{R(t)}$ in Assumption \ref{Assumption 2}, as well as the choice of $R(t)$, taken so that $R(t)^{-\eta/2} \gg t^{ \epsilon(t)}$; this is possible as $\epsilon(t)$ does not depend on $R(t)$.) This completes the proof of \eqref{eq:R_N}. Altogether, the proof of Theorem \ref{T:heat} is therefore complete.  
\end{proof}

\begin{rem}
    The argument shows that it would have been sufficient to replace Assumption \ref{eq:assumption_lowerbound_local} by the following weaker bound 
    $N\ge  R^{ - 2(1- \sqrt{\mathsf{x}})^2 - \delta}$ for some $\delta>0$, or indeed the even weaker bound $N\ge R^{ - \alpha_0^2/2 - \delta}$ for some $\delta>0$. However, as this seems unnatural (compared to the assumption stated in \eqref{eq:assumption_lowerbound_local}, which just corresponds to the dimension of $\partial \Omega$) we left the assumption as such in Theorem \ref{T:heat}.
\end{rem}

\section{Heat trace asymptotics} \label{sec:3}

\label{S:proof_heat}

In this section we prove Theorem \ref{T:qc}, which deals with asymptotics of the (expected) heat trace 
$$
\mathbf{H}(t) = \int_\Omega \mathbf{p}^\Omega_t(x,x) \mu (dz).
$$
As already mentioned in the introduction, our starting point is the in-out decomposition \eqref{eq:decomposition}, which we recall here for convenience:
\begin{equation}\label{eq:decomposition2}
\mathbf{p}_t^\Omega (x,x) = \mathbf{p}_t^{\CC}(x,x) - \mathbf{p}_t^{\CC \setminus \Omega} (x,x) ,
\end{equation}
 where $\mathbf{p}_t^{\CC}(x,x)$ is the heat kernel in the full plane, and $\mathbf{p}_t^{\mathbb{C} \setminus \Omega} (x,x)$ can informally be viewed as the heat kernel from $x$ to $x$ in time $t$, restricted to trajectories that leave $\Omega$ before time $t$. In other words, 
$$
\mathbf{p}_t^{\CC \setminus \Omega} (x,x) = \mathbf{p}^{\CC}_t(x,x) \Pb^\gamma_{x \overset{t}{\to} x} (\sigma_\Omega < t)
$$
 $\Pb^\gamma_{x \overset{t}{\to} x}$ is the law of a $(\gamma-)$\textbf{Liouville Brownian loop} of duration $t>0$. By definition, the latter (and more generally the $\gamma$-Liouville Brownian bridge of duration $t$ from $x$ to $y$) is the unique reversible law on path space $C([0, t]; \CC)$ such that, for any $0 \le s < t$, 
 \begin{equation}\label{eq:defLBB}
\left(\frac{\mathrm{d}\Pb^\gamma_{x \overset{t}{\to} y}}{\mathrm{d} \Pb^\gamma_x } \right)|_{\mathcal{F}_s} = \frac{\pb_{t-s}^{\CC} (\mathbf{X}_s, y)}{\pb_t^{\CC}(x,y)}.
 \end{equation}
(Here, as before, $\Pb^\gamma_x$ is the law of a Liouville Brownian motion starting from $x$, and $\mathcal{F}_s = \sigma(\mathbf{X}_u, u \le s))$ is the canonical filtration on path space, and $\mathbf{X}_s (\omega) = \omega_s$ is the standard evaluation, or coordinate, map.) As already mentioned, the law of a Liouville loop or Liouville Brownian bridge can \emph{not} be simply described as the time-change of an ordinary Brownian bridge from $x$ to $y$. 


 Theorem \ref{T:heat} concerns the expected value of the heat trace. When we take the expectation in \eqref{eq:decomposition2}, we find
 \begin{equation}
     \E [ \Hb(t) ] = \int_\Omega \E^*_x [ \mathbf{p}_t^{\CC} (x,x) ] dx - \int_\Omega \E^*_x [ \mathbf{p}_t^{\CC \setminus \Omega} (x,x)] dx.
 \end{equation}
Our strategy will be to estimate these two terms (which we refer to as respectively the \textbf{bulk} and the \textbf{boundary} terms) separately:

\begin{itemize}
    \item (Bulk term). It is not hard to see that, for any $x \in \Omega$, (e.g., by adapting the proof of \cite{BW}) that $\E^*_x [ \mathbf{p}_t^{\CC} (x,x) ] \sim  \tfrac{c_\gamma}{t}$. In fact we will need something more quantitative, and show that the difference $t\E^*_x [ \mathbf{p}_t^{\CC} (x,x) ] - c_\gamma$ decays to 0 faster than any polynomial (FTP) as $t\to 0$.

    \item (Boundary term). On the other hand, the second term is (at least intuitively) related to the behaviour of the heat content. Indeed, roughly speaking, we can expect that 
    \begin{align*}
\E^*_x [ \mathbf{p}_t^{\CC \setminus \Omega} (x,x)] &\approx \E^*_x [ \mathbf{p}^{\CC}_t(x,x) \Pb_x( \nu(\tau_\Omega)< t) ]  \\
& \approx \text{const.} (1/t)\times  \P^*_x\otimes \Pb_x (\nu(\tau_\Omega) < t) 
    \end{align*}
    so that this second term should be roughly $(1/t) \E [ \Kb_\Omega(t)] = t^{ - (1- \Delta) + o(1)}$, by Theorem \ref{T:heat}. In making the above approximation, we identified the probability that a Liouville Brownian \emph{bridge} of duration $t$ leaves the domain $\Omega$, with the probability that a Liouville Brownian \emph{motion} leaves that domain in the same time, and we also pretended that the two terms in the expectation are independent. 
\end{itemize}

Combined together, these two asymptotics establish Theorem \ref{T:qc}. Although the above is quite convincing, making this precise requires some nontrivial arguments; in fact, we have to argue in a very indirect manner. 


Before starting the proof, we recall for future reference the ``bridge decomposition'', initially proved in \cite{RhodesVargas_spectral}, and which plays a crucial role in the analysis of the LQG heat kernel (playing a crucial role e.g. in \cite{HKPZ} and \cite{BW}). Indeed, while the heat kernel $\pb_t(x,y)$ is difficult to handle directly, the bridge decomposition gives a handle on \textbf{time-integrals} of the LQG heat kernel, relating it to the ordinary heat kernel. More precisely, the bridge decomposition is the following statement (see Theorem 3.4 in \cite{RhodesVargas_spectral}): if $G: [ 0, \infty) \to [0, \infty]$ is a nonnegative measurable function, then 
\begin{equation}\label{eq:bridge}
\int_0^\infty G(t) \pb_t(x,y) dt = \int_0^\infty \Eb_{x \overset{u}{\to} y} [ G( \nu(u))] p_u(x,y)du,
\end{equation}
where $p_u(x,y) = e^{ - |x-y|^2/(2u)} (2\pi u)^{-1}$ is the ordinary (Gaussian) heat kernel.

This is both natural (as Liouville Brownian motion is a time-change of ordinary Brownian motion, a fact which expresses the conformally invariant nature of Brownian motion in dimension two) and surprising, in that a Liouville bridge is \emph{not} the time-change of an ordinary Brownian bridge.

\subsection{Scale-invariance of the heat kernel on the quantum cone}

\label{SS:hkcone}

We start with some estimates which give a control on the $n$th moment of the (on-diagonal) heat kernel, showing in particular finiteness of this $n$th moment. The arguments will hinge on a delicate comparison between the heat kernel with respect to an environment given by a whole plane GFF (with a $\gamma$-log singularity at the origin), and a $\gamma$-quantum cone, respectively. Roughly speaking, the advantage of the former is that Gaussian computations are possible, while in the latter we can take advantage of some exact scale invariance properties, as will be seen later. 

We fix some notations. Let $h$ denote a realisation (under the law $\PP$) of a GFF on the whole plane, normalised so that it has zero unit circle average. For $x \in B_0(1/2)$, let $h^*_x = h + \gamma G(x, \cdot)$. Thus $h^*_x$ has law $\P^*_x$. When the reference point $x$ is clear we will simply write $h^*$. Associated to $h^*$ there is a heat-kernel, which we write as $\pb^*_t(z,w)$. To ease comparison with the case of the quantum cone below we recall the radial decomposition of this field (see Lemma 7.5 in \cite{BP}). Let $\Hrad(\CC)$ denote the completion, with respect to the ordinary Dirichlet inner product, of smooth functions with compact support in $\CC$, which are constant on each circle centered at the origin. Let $\Hcirc(\CC)$ denote the completion of those functions whose circle average on each circle centered at the origin is zero. We recall that we have the orthogonal decomposition $H^1_0(\CC) = \Hrad(\CC) \oplus \Hcirc(\CC)$. This leads to the independent decomposition of $h$:
$$
h = h_{\mathrm{rad}} + h_{\mathrm{circ}}
$$
where $h_{\mathrm{rad}}$ is a continuous function on $\CC$ which is radially symmetric, having the property that if $|z| =r = e^{-s}$ for some $s\in \R$ (thus $r>0$), then 
$$
h_{\mathrm{rad}} (z)= X_s;
$$
with $(X_s)_{s\in \R}$ a two-sided Brownian motion; that is, $X_0 = 0$, and $(X(s))_{s \in [0, \infty)}$ and $(X(-s))_{s \in [0, \infty)}$ two independent standard Brownian motions. Note that if $x = 0$ then $h^* = h^*_x$ has a similar description $h^* = h^*_{\mathrm{rad}} + h^*_{\mathrm{circ}}$, with radial part given by $h^*_{\mathrm{rad}}(z) = X^*_s$ whenever $|z| = e^{-s}$, 
\begin{equation}
    \label{eq:X*}
    X^*_s =X_s + \gamma (s\vee 0); \quad s\in \RR. 
\end{equation}
In order to make the comparison with the quantum cone easier, it is easier to parametrise the punctured plane $\CC \setminus\{0\}$ by the cylinder $\mathcal{S} = \RR \times (0, 2\pi)$ via the map $z \in \cS \mapsto e^{-z}$ (thereby identifying the lines $\{ \Im (z) = 0\}$ with $\{\Im(z) = 2\pi\}$). Applying the LQG change of coordinates, the fields $h$ and $h^*$ become, respectively, 
\begin{align}
\begin{cases}
\hrad(z) &= X_s - Qs  \\
\hrad^*(z) &= X^*_s - Qs = X_s - (Q - \gamma  1_{s \ge 0})s, 
\end{cases}
\quad z \in \cS, \Re(z) = s.
\end{align}
where, with a slight abuse of notation, we employ the same notation in for $\hrad$ in $\cS$ and in $\CC\setminus\{0\}$, and $X_s$ is a two-sided Brownian motion. Their circular part $\hcirc = \hcirc^*$  is independent of their radial part and can be chosen to be identical in both cases.

Now we carefully introduce the $\gamma$-quantum cone, and more generally the $\alpha$-quantum cone (for $0\le \alpha < Q$, corresponding to the ``thick'' regime), a field which we will denote by $\tilde h$, or $\tilde h^\alpha$ if we want to emphasise the dependence on $\alpha$. We again parameterise the punctured plane by $\cS$, and decompose it into a radial and angular components, $\tilde h = \thrad + \thcirc$. By definition of $\tilde h$,  $\thrad$ and $\thcirc$ are independent, with $\thcirc = \hcirc$ in law. As for the radial part, 
\begin{align}\label{eq:hrad_cone}
    \thrad (z) =: \tilde X_s - (Q-\alpha)s, \quad z \in \cS; \Re(z) = s
\end{align}
where $(\tilde X_s)_{s\ge 0}$ and $(\tilde X_{s})_{s\le 0}$ are independent, with laws given respectively by a Brownian motion motion, and another Brownian motion conditioned so that $\tilde X_{|s|} + (Q - \alpha) |s| \ge 0$ for all $s\le 0$. 

\textbf{Notations.} It will sometimes be useful to distinguish the different underlying environments for the heat kernel (be it the quantum cone $\tilde h$ or the whole-plane GFF $h^*$) either by emphasising this at the level of random variables (leading to heat kernels written respectively as $\tilde \pb_t(x,x)$ and $\pb^*_t(x,x)$) or at the level of the underlying laws for the field, which are denoted by $\tilde \P$ and $\P^*_x$ respectively. In the latter case (i.e., when the law is explicitly specified, say in an expectation) we simply write $\pb_t(x,x)$.



Our first observation is that, on the quantum cone, there is an \emph{exact} form of scale-invariance for the heat kernel, which leads to extremely useful scaling relations. For instance, we will see that we have the \emph{exact} relation:
$$
\EE [\mathbf{p}_t^{\cC}(x,x)] = \frac{c_\gamma}{t},
$$
Although scale-invariance for the area and boundary measures of the quantum cone is well known (see \cite{DuplantierMillerSheffield, BP}) the remarkable consequences that this has for the heat kernel seem to not have been noticed before. We believe this technique will have other uses in the future, so this step is interesting in its own right. More generally, the result below is valid for the notion of $\alpha$-quantum cone, although we will need it in this paper only for the case $\alpha = \gamma$.

\begin{prop}
    \label{P:scale_hk}
    Let $\alpha \in [0,Q)$  and let $\tilde h = \tilde h^\alpha$ be an $\alpha$-quantum cone, with associated heat kernel $\tilde \pb_t(z,w) = \tilde \pb^\alpha_t(z,w)$. Let $x = 0$. Then 
    \begin{equation}
\tilde \pb^\alpha_t(x,x) \overset{d}{=} \frac1t \tilde \pb^\alpha_1(x,x).        
    \end{equation}
\end{prop}

\begin{proof}
The idea is simply to exploit the scale-invariance of the quantum cone, which we first recall. To do so it is useful to indicate the explicit dependence of the area measure $\mu = \mu_{\tilde h}$ with respect to a specific realisation of the quantum cone, $\tilde h$. The scale-invariance of the quantum cone (Theorem 7.18 in \cite{BP}) states that the quantum surface described by $(\tilde h +C, \CC, 0, \infty)$ is identical in law to that described by $(\tilde h, \CC, 0, \infty)$, for any $C \in \RR$. Thus $\tilde h + C$ is another quantum cone (in an embedding which is not the usual ``unit circle embedding'', but embeddings are irrelevant as all quantities of interest here are LQG invariants). Note that 
$$
\mu_{\tilde h} (A) = e^{-\gamma C} \mu_{\tilde h+C}(A)
$$
for any Borel set $A$. Note also that if $(B_t, t \ge 0)$ is a fixed two-dimensional standard Brownian motion, then the quantum clocks with respect to $\tilde h$ and $\tilde h+C$ satisfy a similar relation:
$$
\nu_{\tilde h+C, B} (s) = e^{\gamma C} \nu_{\tilde h, B}(s); \quad  s \ge 0.
$$
Thus the corrsponding Liouville Brownian motions $Z = Z^{\tilde h,B}(s)$ associated to $\tilde h$ and Brownian path $B$ satisfy:
$$
Z^{\tilde h, B}(t) = Z^{\tilde h+C, B}(t e^{\gamma C}), \text{ for all } t\ge 0.
$$
Let $t>0$, and let us fix $C = \log(1/ t) / \gamma$, so $e^{\gamma C} = 1/t$. 
It follows that (using continuity of the heat kernel with respect to the space variable and the dominated convergence theorem)
\begin{align*}
\pb^{\tilde h}_t(x,x) &= \lim_{\eps \to 0}     \frac{\Pb_x (Z^{\tilde h,B}(t) \in B_x(\eps))}{\mu_{\tilde h} ( B_x(\eps))}\\
& = \lim_{\eps \to 0} \frac{\Pb_x (Z^{\tilde h+C,B}(1) \in B_x(\eps))}{t\mu_{\tilde h+C} ( B_x(\eps))}\\
& = \frac1t \pb^{\tilde h+C}_1(x,x),
\end{align*}
as desired. 
\end{proof}

Before continuing with the rest of the proof of Theorem \ref{T:heat} we illustrate the interest of this idea by resolving a conjecture, stated in \cite{BW}, about the annealed convergence of the heat kernel in $\Omega$, $\pb^{\Omega}_t(x,x)$ (i.e., with Dirichlet boundary conditions in $\Omega$) under $\P^*$. The conjecture was in fact stated for a Dirichlet GFF in $\Omega$ under the measure $\P^*$, rather than a full plane GFF under the measure $\P^*$, but as will be clear below this difference in boundary conditions of the field has no impact on the result and its proof.


\begin{cor}\label{cor:limit-law}
Consider the random variable $\pb^\Omega_t(x,x)$ under the law $\P^*_x$, where $x\in \Omega$ is arbitrary. Then 
\begin{equation}
t\pb^\Omega_t(x,x) \to X   
\end{equation}
in distribution as $t\to 0$. The random variable $X$ has the law of $\tilde \pb_1(0,0)$. 
\end{cor}

\begin{proof}
This follows directly from Proposition \ref{P:scale_hk}, the local convergence of the pointed quantum surface $(h+C, x)$ as $C \to \infty$ to a $\gamma-$quantum cone as $C\to \infty$ (in the sense of total variation, say -- see Theorem 7.18 in \cite{BP}), and the fact that $\mathbf{p}_t(x,x)$ is measurable with respect to the local environment near $x$ when $t \to 0$. Details are left to the reader; alternatively, we provide a complete argument (which yields quantitative bounds on the error) in Lemma \ref{L:comparison}.
\end{proof}

\subsection{Heat kernel moments}

\label{SS:moments}
\begin{prop}\label{prop:finite moments_new}
Let $x \in B_0(1/2)$. For all $n\ge 1$ and for all $\eps>0$ there exists $c_n$ such that,
under the rooted law $\P^*_x$ we have, for all $0<t<1$,
\begin{align*}
\EE^*_x\bigl[\bigl(\pb_t(x,x)\bigr)^n\bigr]\le c_n (1/t^{1+ \eps})^n.
    \end{align*}
    In particular, $\EE^*_x\bigl[\bigl(\pb_t(x,x)\bigr)^n\bigr]^{1/n} \le t^{-1 + o(1)}$ as $t\to 0$ (and is finite).

    A similar, more precise estimate holds in the case of a quantum cone: there exists $\tilde c_n \in (0, \infty)$ such that, for $x = 0$, we have the identity
    \begin{equation}
\tilde \EE [ (\pb_t(x,x))^n] = c_n (1/t)^{n}. 
    \end{equation}
    In particular, $\tilde \EE [ (\pb_t(x,x))^n]< \infty$ for all $n\ge 0$ and for all $t>0$, and $x = 0$. 
\end{prop}

\begin{proof}
    The proof is similar in both cases (of $\tilde \P$ and $\P^*_x$) and we do them in parallel. We will first treat the case $x = 0$ and then explain how to adapt the argument to the case of $\P^*_x$ with general $x \in B_0(1/2)$. 
    
    As the heat kernel is an LQG invariant (i.e., invariant under the LQG change of coordinates) we are free to use whichever parametrisation of the surface we wish; it will be convenient to switch between the whole plane model and the strip model, for the calculations below. As mentioned above, we will do so without indicating this in the notations. 
    
    Let $n\ge 1$. We recall that by the general theory of Dirichlet forms, $t\mapsto \mathbf{p}_t(x,x)$ is monotone decreasing in $t>0$, and for each fixed $t>0$, $y \in \CC \mapsto \mathbf{p}_t(x,y)$ is maximised at $y = x$. Thus
    $$
\mathbf{p}_t (x,x) \le \frac{2}{t}\int_{t/2}^t \mathbf{p}_u (x,x) du.  
    $$
Thus raising to the power $n\ge 1$ and  \textbf{bridge decomposition} \eqref{eq:bridge}, as well as Fubini's theorem, we obtain 
\begin{align*}
    & \EE^* [ \mathbf{p}_t(x,x)^n]\\
    \le & \frac{2^n}{t^n} \EE^*\left[\int_0^\infty \ldots \int_0^\infty \mathbf{1}_{u_1 \in [t/2,t]} \ldots \mathbf{1}_{u_n \in [t/2, t]} \mathbf{p}_{u_1}(x,x) \ldots \mathbf{p}_{u_n}(x,x) du_1 \ldots du_n\right]\\
     \le & \frac{2^n}{t^n} \int_0^\infty \ldots \int_0^\infty  \EE^*\otimes\Eb_{x\overset{u_1}{\rightarrow}x,\ldots x\overset{u_n}{\rightarrow}x} \left[ \mathbf{1}_{\nu^1 (u_1) \in [t/2,t] } \ldots \mathbf{1}_{\nu^n (u_n) \in [t/2,t]} \right] \frac1{(2\pi)^n u_1 \ldots u_n} du_1 \ldots du_n
\end{align*}
where $\Eb_{x\overset{u_1}{\rightarrow}x,\ldots x\overset{u_n}{\rightarrow}x} $ is the law of $n$ \emph{independent} ordinary Brownian bridges from $x$ to $x$, of respective durations $u_1, \ldots u_n$; and $\nu^1 (\cdot), \ldots, \nu^n(\cdot)$ are the quantum clocks accumulated by each of these bridges in the \emph{common} environment $\P^*_x$.

Applying Hölder's inequality, we get 
\begin{equation}\label{eq:HKintegral}
\EE^*_x [ \mathbf{p}_t(x,x)^n]\le \frac{1}{\pi^n t^n} \left( \int_0^\infty \frac1{u}  \PP^*_x\otimes\Pb_{x\overset{u}{\rightarrow}x}\bigl[\nu(u)\in[t/2,t]\bigr]^{\frac{1}{n}} du \right)^n.
\end{equation}

    In the case of the quantum cone, we recall that by scaling (Proposition \ref{P:scale_hk}) it suffices to take $t = 1$. The bound \eqref{eq:HKintegral} also holds with $\tilde \P$ replacing $\P^*_x$. 

Thus it suffices to control the integral on the right hand side of \eqref{eq:HKintegral} (and its analogue for the quantum cone in the case $t=1$), which we break into $u \in (0,t)$, $u\in (t,1)$, and $u \in (1, \infty)$. When $u \in (t,1)$ we simply bound the probability by 1 so the integral is at most $\log (1/t)$. 

So let us assume  $u>1$, and observe that, by monotonicity of $\nu$, 
\begin{align*}
    \PP^*_x\otimes\Pb_{x\overset{u}{\rightarrow}x}\bigl[\nu(u)\in[t/2,t]\bigr] & \le \PP^*_x\otimes\Pb_{x\overset{u}{\rightarrow}x}\bigl[\nu(u)\le t \bigr]\\
    & \le \PP^*_x\otimes\Pb_{x\overset{u}{\rightarrow}x}\bigl[\nu(u/2) \le t\bigr].
\end{align*}
Now, the law of a Brownian bridge of duration $u$, restricted to $[0,u/2]$, is absolutely continuous with respect to that of an ordinary Brownian motion over the same interval of time, with Radon--Nikodym derivative bounded by a universal constant $C>0$. Thus for any $q>0$, using also the monotonicity of $\nu$ and Markov's inequality,
\begin{align}
    \PP^*_x\otimes\Pb_{x\overset{u}{\rightarrow}x}\bigl[\nu(u)\in[t/2,t]\bigr] & \le  C \PP^*_x\otimes\Pb_x\bigl[\nu(u/2)\le t\bigr] \nonumber \\
    & \le C t^{q} \E^*_x \otimes \Eb_x [ \nu (u/2)^{-q} ]. \label{eq:HKbound1} 
\end{align}

We will thus need to control the negative moments of $\nu (u)$. We defer this to slightly later, explaining first how we take care of the contribution to the integral coming from $u<t$. 
In that case,
    \begin{align*}
        \PP^*_x\otimes\Pb_{x\overset{u}{\rightarrow}x}\bigl[t/2\le\nu(u) \le t \bigr]& \le\PP^*_x\otimes\Pb_{x\overset{u}{\to}x}[\nu(u)\ge t/2]\\
        & \le \PP^*_x\otimes\Pb_{x\overset{u}{\to}x}[\nu(u/2)\ge t/4] + \PP^*_x\otimes\Pb_{x\overset{u}{\to}x}[\nu(u/2, u)\ge t/4]
\end{align*}
where $\nu (u/2, u)$ is the quantum clock accumulated by the Brownian bridge during the interval $[u/2, u]$ (i.e., $\nu (u/2, u) = \nu(u) - \nu(u/2)$). By reversibility of the Brownian bridge, and by the same absolute continuity of Brownian bridge with respect to ordinary Brownian motion which was already used in \eqref{eq:HKbound1}, we get
\begin{align*}
        \PP^*_x\otimes\Pb_{x\overset{u}{\rightarrow}x}\bigl[t/2\le\nu(u) \le t \bigr] & \le 2C  \PP^*_x\otimes\Pb_{x}[\nu(u/2)\ge t/4]
        \\
        &\le2C \big(\PP^*_x\otimes\Pb_{x}[\nu(u/2\wedge\tau_1)\ge t/4]+\Pb_{x}[\tau_1< u/2]\big),
    \end{align*}
    where we recall that $\tau_1 = \inf\{ t>0: |X_t-x | > 1\}$ is the first time $X$ (the Brownian motion, here) leaves a ball of radius 1 around its starting point. Note that $\Pb_{x}[\tau_1< u/2] \to 0$ as $u \to 0$ FTP (in fact, exponentially fast in $(-1/u)$, by the explicit form of the Gaussian density function), so it suffices to concentrate on the first term. By Markov's inequality, we again have for all $q>0$,
    \begin{align} \label{eq:HKbound2}
        \PP^*_x\otimes\Pb_{x}[\nu(u/2\wedge\tau_1)\ge t/4] & \le (t/4)^{-q} \EE^*_x \otimes \Eb_x [ \nu (u/2 \wedge \tau_1)^q].
    \end{align}

Altogether, we get, for any choice of $q,q'>0$:
\begin{align}\label{eq:total_integral_moments}
     \int_0^\infty\frac1u 
        \PP^*_x\otimes\Pb_{x\overset{u}{\rightarrow}x}\bigl[\nu(u)\in[t/2,t]\bigr]^{\frac{1}{n}} du  & \le \log (1/t) + C  \nonumber \\
        & \quad + C_{q,n}t^{q/n} \int_1^\infty  \frac{1}{u} \E^*_x \otimes \Eb_x [ \nu(u/2)^{-q}]^{1/n} du \nonumber \\
        & \quad + t^{- q'/n}\int_0^t \frac1u \E^*_x \otimes \Eb_x [ \nu(u/2 \wedge \tau_1)^{q'}]^{1/n} du 
\end{align}

We now state the relevant controls over the moments needed to complete this proof, starting with the control at ``large'' scales ($u>1$):

\begin{lem}
    \label{L:negmoments}
    For all $\eps>0$, for all $q>0$ there exists a constant $C  = C( \gamma, q, \eps)$ such that for \emph{all} $u>0$, 
    $$
    \E^*_x \otimes \Eb_x [ \nu (u)^{-q} ] \le Cu^{\xi^*(-q)/2 + \eps}, 
    $$ 
    where $\xi^*(-q) = - q ( 2 - \gamma^2/2) + \gamma^2 q^2/2$. The same estimate applies to a quantum cone: for all $u>0$,
    $$
\tilde \E \otimes \Eb_x [ \nu (u)^{-q} ] \le Cu^{\xi^*(-q)/2 + \eps}
    $$
\end{lem}

The interest and potential difficulty here lies in the fact that we want a result valid for \emph{all} $u>0$, i.e., including at large scales. The function $\xi^*$ can be viewed as the $\gamma$-quantum cone multifractal spectrum (or that of a whole plane GFF with $\gamma$-log singularity at \emph{all} scales -- including large scales, unlike \eqref{eq:X*} where the logarithmic singularity is only present at scales less than unity). 

\begin{proof}
See Appendix \ref{A:proof_negmom}. \end{proof}

Thus, applying Lemma \ref{L:negmoments} with $ q >0$ small enough that $\xi^*(-q)>0$ (which is possible when $\gamma<2$), and plugging into \eqref{eq:total_integral_moments} we see that the integral in the second line is bounded by a constant (say $C$), so this entire term is bounded by $Ct^{1/n} \le C$. 

The control at small scales ($u<1$) is provided by the following lemma. 
\begin{lem}
 \label{L:posmoments}
 For all $q>0$ there exists a constant $C = C( \gamma, q)$ such that for all $0<u <1$, 
 $$
\EE^*_x \otimes \Eb_x [ \nu (u/2 \wedge \tau_1)^q] \le C u^{\xi^*(q)/2}
 $$
 where as before, $\xi^*(q) = (2- \gamma^2/2)q - \gamma^2q^2/2$. The same estimate applies to a $\gamma$-quantum cone. 
\end{lem}

Again, the proof is deferred to the appendix. Consequently, the integral in the third line of \eqref{eq:total_integral_moments} is convergent and thus (crudely) bounded for any $q'>0$, by some constant $C_{q',\gamma,n}$.
Altogether, this shows that the integral in \eqref{eq:total_integral_moments} is bounded by $C( 1+ \log (1/t) + t^{ - q'/n}) \le C t^{ - q'/n}$ for $t$ small enough. 
Thus, combining our estimates and recalling \eqref{eq:HKintegral},
 we see that 
\begin{align*}
    \EE^*_x [ \mathbf{p}_t(x,x)^n] & \le \frac{C}{t^n}  t^{ - q'}
\end{align*}
for $t$ sufficiently small. Since $q'>0$ is arbitrary, this concludes the proof of Proposition \ref{prop:finite moments_new}.
\end{proof}

\subsection{Quantitative bulk convergence }

\label{SS:bulk}

\def\cC{\mathcal{C}}

\emph{Overview.} In this section we are concerned with the (expected) ``bulk'' term in the in-out decomposition \eqref{eq:decomposition2}, namely $\EE [ \int_\Omega \mathbf{p}^{\CC}_t(x,x) \mu(dx) ]$. Recall that we already know (by an easy adaptation of the results in \cite{BW}) that 
\begin{equation}\label{eq:bulk_rough}
\EE [ \int_\Omega \mathbf{p}^{\CC}_t(x,x) \mu(dx) ] \sim \frac{c_\gamma}{t} |\Omega|,
\end{equation}
and our goal in this subsection is to show a quantitative form of convergence, namely, that the difference of the two terms above decays to zero faster than any polynomial (FTP) as $t \to 0$:
\begin{equation}
\label{eq:goal_bulk}
\left|\EE [ \int_\Omega \mathbf{p}^{\CC}_t(x,x) \mu(dx) ] - \frac{c_\gamma}{t} |\Omega|\right| \to 0, \quad \text{FTP}.
\end{equation} 

Rather than refining the techniques of \cite{BW} to obtain quantitative estimates that are sufficiently strong (which we feel would be quite difficult) we take an orthogonal approach, showing that the heat kernel associated to $h^*_x$ can be compared with very good quantitative error bound (FTP) to that associated with $\tilde h$. Using the scale invariance of heat kernel on the quantum cone of Proposition \ref{P:scale_hk}, this then gives quantitative (FTP) bound on the ``in'' term in the in-out decomposition. We start with the lemma allowing us to compare the two heat kernels. 



\begin{lem}\label{L:comparison}
Let $x \in B(0,1/2)$. Consider the fields $h^*_x$ (whose law is that of a GFF on the whole plane, normalised to have zero unit circle average, plus $\gamma G^{\CC} ( x, \cdot)$) and the $\gamma$-quantum cone $\tilde h $ centered at $x$ (i.e., $\tilde h (\cdot - x)$ is a $\gamma$-quantum cone). 
Let $\mathbf{p}^*_t(x,x)$ and $\tilde{\mathbf{p}}_t(x,x)$ denote the respective associated heat kernels. We can couple $h^*_x$ and $\tilde h$ in such a way that, uniformly in $x \in \Omega$,
$$
\E [ |\mathbf{p}^*_t(x,x) - \tilde{\mathbf{p}}_t(x,x) | ] \to 0, \quad \mathrm{FTP}.
$$
\end{lem}

\begin{proof}
We note that, by construction of the $\gamma$-quantum cone, the fields $\tilde h$ and $h^*_x$ can be made to agree exactly on the ball of radius 1/2, which we denote by $B$ in this proof. We assume that $h$ and $\tilde h$ are coupled in this way (for instance, start by sampling $h$ under $\P^*_x$, then sample $\tilde h$ according to the conditional law of $\tilde h$ given $\tilde h|_{B}$). Similar to the in-out decomposition, note that 
\begin{equation}\label{eq:inoutball1}
\mathbf{p}^*_t(x,x) = \mathbf{p}_t^B (x,x) + \mathbf{p}_t^{ \CC \setminus B} (x,x),
\end{equation}
where the first term consists of the heat kernel restricted to trajectories staying entirely in $B$, while the second denotes the heat kernel restricted to trajectories leaving $B$ before time $t$ (e.g., $\mathbf{p}_t^{ \CC \setminus B} (x,x) = \mathbf{p}_t (x,x) \Pb^\gamma_{ x \overset{t}{\to} x} ( \sigma_B< t )$). For notational ease we have removed the star from the two heat kernels on the right hand side.

Likewise, we may apply the same decomposition to $\tilde \pb_t(x,x)$, so that 
\begin{equation}\label{eq:inoutball2}
\tilde{\mathbf{p}}_t(x,x) = \tilde{\mathbf{p}}_t^B (x,x) + \tilde{\mathbf{p}}_t^{ \CC \setminus B} (x,x),
\end{equation}
Since $\tilde h$ and $h^*_x$ agree in $B$, we have 
$$
\tilde{\mathbf{p}}_t^B (x,x)  = \tilde{\mathbf{p}}_t^B (x,x) 
$$
so that 
$$
|\tilde{\mathbf{p}}_t(x,x) - \mathbf{p}^*_t(x,x)| \le \tilde{\mathbf{p}}_t^{ \CC \setminus B} (x,x) + {\mathbf{p}}_t^{ \CC \setminus B} (x,x).
$$
We will bound both terms separately, showing that they converge to 0 in $L^1(\P)$, FTP. The argument is the same for both terms, so let us start with the second one, say. 
Note that, by reversibility of the Liouville Brownian bridge (which is to say, by reversibility of Liouville Brownian motion with respect to $L^2 (\mu_h)$), letting $\sigma_B$ the first exit time of a given trajectory (which will either be Liouville Brownian motion or bridge below -- note that in the latter case we cannot express the Liouville bridge as a time-change of an ordinary Brownian bridge, so cannot write $\sigma_B = \nu (\tau_B)$ as we would for a Liouville Brownian motion): 
\begin{align*}
    \pb_t^{\CC \setminus B} (x,x) & = \Pb^\gamma_{x \overset{t}{\to} x} ( \sigma_B < t) \pb_t(x,x) \\
    & = 2\Pb^\gamma_{x \overset{t}{\to} x} ( \sigma_B < t/2) \pb_t(x,x) \\
    & = 2\Eb^\gamma_x \left[ 1_{\{\sigma_B< t/2\}} \pb_{t/2} ( \mathbf{Z}_{t/2} , x)\right]
\end{align*}
where in the last line we used the explicit construction of the law of Liouville Brownian bridge on $[0,t]$ (or rather, its restriction to $[0, t/2]$ as having the density $\pb_{t/2} (\cdot, x)/\pb_t(x,x)$ with respect to the law of Liouville Brownian motion $(\mathbf{Z}_t, t \ge 0)$. 

Now, by the general theory of Dirichlet forms, $\pb_s(\cdot, x)$ is maximised on the diagonal for any given $s>0$; thus
\begin{align}\label{eq:bridge_far}
 \pb_t^{\CC \setminus B} (x,x) \le 2\Pb_x^\gamma ( \sigma_B < t/2) \pb_{t/2} (x,x) .
\end{align}
Let us take expectation with respect to the field $h^*_x$, and use Cauchy-Schwarz: 
\begin{align*}
    \E^*_x [\pb_t^{\CC \setminus B} (x,x)  ] & \le 2 \E^*_x [ \Pb_x^\gamma ( \sigma_B < t/2)^2]^{1/2} \E^*_x[\pb_{t/2} (x,x)^2]^{1/2} .
\end{align*}
Concerning the first expectation, $\Pb_x^\gamma ( \sigma_B < t/2) \le 1$ so $\Pb_x^\gamma ( \sigma_B < t/2)^2 \le \Pb_x^\gamma ( \sigma_B < t/2)$, hence by Fubini's theorem 
\begin{align*}
    \E^*_x [\pb_t^{\CC \setminus B} (x,x)  ] & \le 2 \left[\P^*_x \otimes \Pb^\gamma_x ( \sigma_{B}< t/2) \right]^{1/2} \E^*_x[\pb_{t/2} (x,x)^2]^{1/2}\\
    & = 2\left[\P^*_x \otimes \Pb_x (\nu(\tau_B) < t/2) \right]^{1/2} \E^*_x[\pb_{t/2} (x,x)^2]^{1/2}.
\end{align*}
The first term tends to zero FTP by Lemma \ref{lem:tailbound}. The second tends to zero no faster than $t^{-1 + o(1)}$ by Proposition \ref{prop:finite moments_new}. Altogether, this shows that $\E^*_x [\pb_t^{\CC \setminus B} (x,x)  ]  \to 0$, FTP. The same reasoning applies to $\tilde \E [\pb_t^{\CC \setminus B} (x,x)  ] $ (note that Lemma \ref{lem:tailbound} applies also here to show that $\tilde \P \otimes \Pb_x (\nu(\tau_B) < t/2)$, since this estimate concerns only the behaviour of the field \emph{inside} $B$, where $\tilde h$ and $h^*_x$ agree). 
This concludes the proof of Lemma \ref{L:comparison}.
\end{proof}

\begin{rem}\label{rem:comparison2}
For a fixed $x \in \Omega$, the same argument can be used to compare $\pb_t^*(x,x)$ (which is the heat kernel for $h^x$ in the full plane) with $\pb_t^\Omega(x,x)$ (which is the heat kernel for $h^x$, but restricted to trajectories not leaving $\Omega$, and under the measure $\P_x^*$): we have $\E [ | \pb_t^*(x,x) - \pb_t^\Omega(x,x)| ]\to 0$, FTP. 
\end{rem}

An immediate consequence of this Lemma, together with the scale invariance of the heat kernel on the quantum cone is a strong (quantitative) estimate for the (expected) bulk term in the in-out decomposition \eqref{eq:decomposition2}, which proves our stated goal \eqref{eq:goal_bulk}.

\begin{lem}
 \label{L:bulk_quant}
 We have
 \begin{equation}\label{eq:bulk_quant}
 \left| \int_\Omega \E^*_x[\pb_t^{\CC}(x,x) ] dx - \frac{|\Omega|}{t} \tilde \E[\pb_1(0,0)] \right| \to 0, \quad \text{\emph{ FTP.}}
 \end{equation} 
 As a consequence, 
 \begin{equation}\label{eq:HK_cgamma}
 \tilde \E[\pb_1(0,0)] = c_\gamma = \frac{1}{\pi(2- \gamma^2/2)}
 \end{equation}
\end{lem}

\begin{proof}
The estimate \eqref{eq:bulk_quant} follows directly from Lemma \ref{L:comparison} together with the scale invariance of the heat kernel on the quantum cone (with $\alpha = \gamma$) proved in Proposition \ref{P:scale_hk} (and the fact that the expectation of $\tilde \pb_1(0,0)$ is finite by Proposition \ref{prop:finite moments_new}).
As for \eqref{eq:HK_cgamma}, this follows by comparing with \eqref{eq:bulk_rough}.
\end{proof}

We actually have all the tools to prove Theorem \ref{T:limit_law}.
\begin{proof}[Proof of Theorem \ref{T:limit_law}]
We have already seen in  Corollary \ref{cor:limit-law} that $t \pb_t^{\Omega}(x,x) \to X $ in law as $t \to 0$, where $X = \tilde \pb_1(0,0)$. In fact, from Remark \ref{rem:comparison2} and Proposition \ref{P:scale_hk} there is a coupling such that $\E [ | t \pb_t^{\Omega}(x,x) -  X| ] \to 0$, FTP. 

From Lemma  \ref{L:bulk_quant} we furthermore have $\E[X] = c_\gamma$, and from Proposition \ref{prop:finite moments_new} we have that $\E[X^n] < \infty$ for all $n\ge 0$. 
\end{proof}

\subsection{Boundary term}

\label{SS:boundary}

Taking into account the in-out decomposition \eqref{eq:decomposition2}, the quantitative control on the bulk (or ``in'') term in Lemma \ref{L:bulk_quant}, in order to prove Theorem \ref{T:qc}, it remains to show:
\begin{prop}\label{L:boundary}
\begin{equation}
\label{eq:goal_bd}
\E\left[ \int_\Omega \pb_t^{\CC \setminus \Omega} (x,x) \mu_h(dx) \right] = t^{ \Delta -1 + o(1)}.
\end{equation}
\end{prop}
We divide the proof in an upper bound and a lower bound. We start with the upper bound, which is a bit easier.

\begin{proof}[Proof of upper bound of \eqref{eq:goal_bd}] We start by applying Fubini's theorem and Girsanov's lemma, so that 
\begin{align*}
    \E\left[ \int_\Omega \pb_t^{\CC \setminus \Omega} (x,x) \mu_h(dx) \right] & = \int_\Omega \E^*_x\left[ \pb_t^{\CC \setminus \Omega} (x,x) \right] dx.
\end{align*}
Arguing as in \eqref{eq:bridge_far} (but with $\Omega$ instead of $B$), we have 
\begin{equation}
    \pb_t^{\CC\setminus \Omega} (x,x) \le  \pb_{t/2}(x,x)\Pb^\gamma_x ( \sigma_\Omega< t/2)
\end{equation}
    Therefore, taking expectation $\E^*_x$ and applying Hölder's inequality, we have for some arbitrary $n\ge 1$ and some constant $c_n$ depending only on $n\ge 1$ and $\gamma$, 
    \begin{align*}
        \E\left[ \int_\Omega \pb_t^{\CC \setminus \Omega} (x,x) \mu_h(dx) \right] & \le \int_\Omega \E^*_x [ \pb_{t/2} (x,x)^{n+1}]^{\tfrac{1}{n+1}} \left(\P^*_x \otimes \Pb_x^\gamma [ \sigma_\Omega < t/2]\right)^{\tfrac{n}{n+1}} dx
        \\
& \le c_nt^{ -1 +o(1)} \int_\Omega \left(\P^*_x \otimes \Pb_x^\gamma [ \sigma_\Omega < t/2]\right)^{\tfrac{n}{n+1}} dx\\
& \le c_n t^{-1+ o(1)} \left(\int_\Omega \P^*_x \otimes \Pb_x^\gamma [ \sigma_\Omega < t/2] dx\right)^{\tfrac{n}{n+1}}
    \end{align*}
where in the second line we used Proposition \ref{prop:finite moments_new}, and in the last line we used Jensen's inequality, since $x\mapsto x^{n/(n+1)}$ is concave.
The integral in the last line is nothing but $\E[ \Kb_\Omega ( t/2)]$. Therefore, applying Proposition \ref{prop:first moment Heat content fractal boundary}, we get 
\begin{align*}
        \E\left[ \int_\Omega \pb_t^{\CC \setminus \Omega} (x,x) \mu_h(dx) \right]
        \le c_n t^{-1+ o(1)} (t^{\Delta + o(1)})^{\tfrac{n}{n+1}}.
        \end{align*}
Since $n\ge 1$ can be chosen arbitrarily large, this proves
$$
 \E\left[ \int_\Omega \pb_t^{\CC \setminus \Omega} (x,x) \mu_h(dx) \right] \le t^{ -1 + \Delta + o(1)},
$$
which is the desired upper bound. 
\end{proof}

\begin{proof}[Proof of lower bound of \eqref{eq:goal_bd}] As in the proof of Proposition \ref{prop:first moment Heat content fractal boundary}, we get a lower bound on the expectation by restricting to points $x$ at the right distance to the boundary, and to an event $\cE = \cE_x$ concerning the behaviour of the GFF near $x$ similar to those already appearing in the analysis of the heat content, and which ought to provide the dominant contribution. However, the reason that the lower bound is harder is that there is no independence between the $\pb_t(x,x)$ and the above event $\cE_x$ which determines the behaviour of the GFF near $x$; on the contrary, these events may be \emph{a priori} very correlated. Since we want a lower bound, we cannot (as in the above upper bound) rely on Hölder's inequality. Thus, an argument is needed, to show that the conditional expectation of $\pb_t(x,x)$ remains of order $(1/t)$, even when we condition on the event $\cE_x$.

Let $\delta>0$, and introduce the following region (our notation here differs slightly from that used in the proof of Proposition \ref{prop:first moment Heat content fractal boundary}): 
        \begin{align*}
        \Omega^*_{\delta}&:=\Bigl\{x\in\Omega\Big\vert t^{\kappa_0(1+\delta)}\le d(x,\partial^*\Omega)\le 2t^{\kappa_0(1+\delta)}\Bigr\},
    \end{align*}
where we recall that $\partial^*\Omega$ is a connected component of the boundary of $\Omega$ of positive diameter, and we write $\kappa_0=\kappa(\alpha_0)$. 

We let $r = t^{\kappa_0(1+\delta)}$ and note that $\text{dist} (z, \partial^* \Omega) \asymp r$ whenever $z \in \Omega^*_{\delta}$. We choose a larger scale $r' = t^{\kappa_0 (1- \delta)}$, so that $r' /r = t^{ -2 \kappa_0 \delta} \to \infty$ as $t \to 0$. (More generally the proof will depend on a number of choices of scales which have not been optimised, but are sufficient for the proof). In particular, for $x \in \Omega_\delta$, the ball $B_r(x) \subset \Omega$ but a substantial portion of $B_{r'}(x)$ lies outside $\Omega$. We will condition the field on its circle average distance $r'$, as follows:
\begin{align*}
    \cE_{x} &:=\{ | h_{r'}(x) - \alpha_0\log(1/r') | \le 1 \}
\end{align*}

Then, obviously,
    \begin{align*}
        \int_\Omega\EE^*_x[\pb_t^{\CC \setminus\Omega}(x,x)]dx&\ge\int_{\Omega^*_{\delta}}\EE^*_x[\pb_t^{\CC \setminus \Omega}(x,x)\textbf{1}_{\cE_{x}}]dx =\int_{\Omega^*_{\delta}}\EE^*_x[\pb_t^{\CC \setminus \Omega}(x,x)\vert\cE_{x}]\PP^*_x[\cE_{x}]dx.
    \end{align*}
As in \eqref{eq:compare h to h'}, we have
   \begin{align}
       \PP^*_x[\cE_{x}] & = \exp \left( - \tfrac{(\gamma - \alpha_0)^2}{2} \log (1/r') (1+ o(1))\right) \nonumber \\
       &  = t^{\kappa_0(1-\delta)(\gamma-\alpha_0)^2/2+o(1)}, \label{eq:P*E_x}
   \end{align}
   so it remains to estimate $\EE^*[\pb_t^{\CC \setminus \Omega}(x,x)\vert\cE_{x}]$ for $x \in \Omega^*_{\delta}$, and more precisely we claim that it suffices to 
   show 
   \begin{equation}\label{goal:Econd_pt}
     \EE^*[\pb_t^{\CC \setminus \Omega}(x,x)\vert\cE_{x}] \ge   t^{-1 + o(1)},
   \end{equation} 
   uniformly over $x \in \Omega^*_{\delta}$. Indeed, if this holds, then 
\begin{align*}
\int_\Omega\EE^*_x[\pb_t^{\CC \setminus\Omega}(x,x)]dx & \ge t^{\kappa_0 (\gamma - \alpha_0)^2/2 -1 } | \Omega^*_\delta| \times t^{  \delta[ \kappa_0 (\gamma - \alpha_0)^2/2] + o(1)}\\
& \ge t^{2\mathsf{x}\kappa_0 + \kappa_0 ( \gamma - \alpha_0)^2/2 -1 + o(1)} \times t^{  \delta[  \kappa_0 (\gamma - \alpha_0)^2/2] + o(1)},
\end{align*}
where we used our assumption \eqref{eq:scaling_informal} to lower bound $|\Omega^*_\delta|$. Recalling from \eqref{eq:Delta_alt} that $\kappa_0 (2\mathsf{x}+  ( \gamma - \alpha_0)^2/2) = \Delta$, we get 
$$
\int_\Omega\EE^*_x[\pb_t^{\CC \setminus\Omega}(x,x)]dx\ge t^{\Delta -1} \times t^{ \delta[  \kappa_0 (\gamma - \alpha_0)^2/2] + o(1)}, 
$$
which proves the lower bound of \eqref{eq:goal_bd} since $\delta$ is arbitrary.

 We thus focus on \eqref{goal:Econd_pt}.
 Even though  $\pb_t^{\CC \setminus \Omega}(x,x)$ is far from independent from the event $\cE_{x}$, it is nevertheless reasonable to hope for such a lower bound in view of the a.s. \emph{uniform} (over the entire domain) lower bound on the on-diagonal heat kernel $\mathbf{p}_t(x,x)$ of the form $t^{-1+ o(1)}$ obtained in \cite{AK2016}. Of course, the latter is an a.s. bound, rather than a bound in expectation, so it cannot be directly used, but this suggests it is indeed reasonable to hope for such a conditional bound. 

To see \eqref{goal:Econd_pt} we first change the field $h^*_x$ to the $\gamma$-quantum cone. Arguing as in Lemma \ref{L:comparison}, the error we make in doing so tends to zero FTP:
$$
\E^*_x[ \pb_t^{\CC \setminus \Omega} (x,x)  | \cE_x ] = \tilde \E_x[ \pb_t^{\CC \setminus \Omega} (x,x)  | \cE_x ] + o_{\text{FTP}}(1),
$$
where $\tilde \E_x$ denotes the law of a $\gamma$-quantum cone translated so $x$ is at the origin. (Briefly, the only difference with the setting of Lemma \ref{L:comparison} is that here we are conditioning on $\cE_x$. Recall that Lemma \ref{L:comparison} relies on coupling $h^*_x$ and the translated quantum cone in such a way that they agree exactly in a ball of radius at least $1/2$ around $x$, hence the events $\cE_x$ are the same for both fields and thus the argument from that Lemma applies \emph{verbatim}.) It thus suffices to prove 
\begin{equation}\label{eq:goal_cond_pt_cone}
\tilde \E_x [ \pb_t^{\CC\setminus \Omega}(x,x) |\cE_x] \ge c t^{-1}
\end{equation}
for $t$ small enough, uniformly in $x \in \Omega^*_\delta$. Let us apply a scaling to the quantum cone, as in Proposition \ref{P:scale_hk}, i.e., consider the ``unit circle embedding'' $\bar h_x$ of $\tilde h_x + C$, where $\tilde h_x$ has the law $\tilde \P_x$, of the translated $\gamma$-quantum cone and $C = (1/\gamma) \log (1/t)$. In other words, we apply the LQG change of coordinate:
\begin{equation}
     \bar h_x(\cdot)  = \tilde h_x ( r_C\cdot) + C+ Q \log r_C,
\end{equation}
where $r_C = e^{ - s_C}$ is defined by 
$$
s_C = \inf\{ s \in \R: \tilde{X}_s - (Q- \gamma) s = -C\} 
$$
and where $\tilde X_s$ describes the radial part of the $\gamma$-quantum cone, as in \eqref{eq:hrad_cone} (i.e., the circle average at distance $\rho$ from the origin of $\tilde h$ is $\tilde X_{\log (1/ \rho)} + \gamma \log (1 / \rho)$).

As explained in Proposition \ref{P:scale_hk}, we have $\tilde \pb_t(x,x) = (1/t) \bar \pb_1(x,x)$, where $\tilde \pb$ and $\bar \pb$ denote the heat kernels associated respectively to $\tilde h_x$ and to $\bar h_x$. Furthermore, 
$$
\pb_t^{\CC \setminus \Omega}(x,x) = \frac{1}{t} \bar \pb^{\CC \setminus r_C^{-1} \Omega}_1(x,x)
$$
where $r_C^{-1} \Omega$ denotes, with a slight abuse of notation, the image of $\Omega$ under $z \mapsto (z-x)r_C^{-1} + x$. 
Furthermore, when we phrase the event $\cE_x$ in terms of $\bar h_x$ instead of in terms of $\tilde h_x$, 
$$
\bar \cE_x = \Bigl\{ \Big|(\bar h_x)_{r'/r_C}(x) - \alpha_0 \log (1/r') + C + Q\log r_C\Big| \le 1 \Bigr\},
$$
where $(\bar h_x)_{\rho}(x)$ denotes the circle average of $\bar h_x$ at distance $\rho$ from $x$. Thus, using the invariance of the heat kernel under LQG coordinate changes, 
\begin{equation}\label{eq:rescaled_hk}
\tilde \E_x [ \pb_t^{\CC \setminus \Omega} (x,x) | \cE_x ] = \frac1t \bar \E_x [ \pb_1^{ \CC \setminus r_C^{-1} \Omega } (x,x) | \bar \cE_x].
\end{equation}

We claim that:
\begin{itemize}
    \item (Large scales.) The effect of the conditioning becomes negligible as $t\to 0$, i.e., the conditional law of 
    $\bar h_x$ given $\cE_x$ converges (say in the sense of total variation locally on any compact) to that of the unconditional $\gamma$ quantum-cone $\tilde \P_x$:
    \begin{equation}\label{eq:simpl1}
        \bar \P_x ( \cdot | \bar \cE_x) \Rightarrow \tilde \P_x
    \end{equation}
    \item (Small scales.) The requirement that trajectories leave the scaled domain $r_C^{-1} \Omega$ in the computation of the heat kernel at time 1 also becomes irrelevant, since this domain shrinks to the origin $x$; in other words, 
    \begin{equation}\label{eq:simpl2}
        \pb_1^{\CC \setminus r_C^{-1} \Omega} (x,x) \to \pb_1^{\CC} (x,x)
    \end{equation}
\end{itemize}

Assuming these two facts, one sees that, using \eqref{eq:rescaled_hk} and Fatou's lemma, 
\begin{align}\nopagebreak
    \liminf_{t\to 0} t \tilde \E_x [ \pb_t^{ \CC \setminus \Omega} (x,x) | \cE_x] &\ge \liminf_{t\to 0} \bar \E_x[ \pb_1^{ \CC \setminus r_C^{-1} \Omega} (x,x) | \bar \cE_x] \nonumber \nopagebreak \\
    & \ge \tilde \E_x[ \pb^{\CC}_1(x,x) ] = c_\gamma \label{eq:Fatou}
\end{align}
where in the last line we used \eqref{eq:HK_cgamma}. So, this proves \eqref{eq:goal_cond_pt_cone} (in fact, a more precise version of \eqref{eq:goal_cond_pt_cone} in which there is no term of the form $t^{o(1)}$ and the constant in front is explicit, and is given by $c_\gamma$).

The proof of these two facts \eqref{eq:simpl1} and \eqref{eq:simpl2} is encapsulated by the following lemma. 

\begin{lem}
    \label{L:mapping}
    Given $\cE_x$, we have $r' /r_C \to \infty$ in probability as $t \to 0$, and $\bar \P_x ( \cdot | \cE_x ) \Rightarrow \tilde \P_x(\cdot)$ as $t \to 0$ (with $\Rightarrow$ denoting total variation convergence locally on \emph{any} compact set of $\CC$), uniformly in $x \in \Omega_\delta$.  
    
    Furthermore, we have $r /r_C \to 0$ in probability, and thus $d ( x, \partial^*(r_C^{-1} \Omega)) \to 0$. In particular,  $ \bar \pb_1^{\CC \setminus r_C^{-1} \Omega} (x,x) \to \bar \pb_1^{\CC} (x,x)$ in law.

\end{lem}

\begin{rem}
    For the final claim of this lemma it is crucial that the compact sets in the topology underlying the convergence in law $\Rightarrow$ are not just in $\CC \setminus \{x \}$ but in all of $\CC$, thereby allowing to identify the limiting geometry near $x$ itself. 
\end{rem}

\begin{proof}
    To prove this lemma it is preferrable to change coordinates to the cylinder $\cS$, as in \eqref{eq:hrad_cone}.
    \begin{figure}
        \centering
    \includegraphics[width=0.8\linewidth]{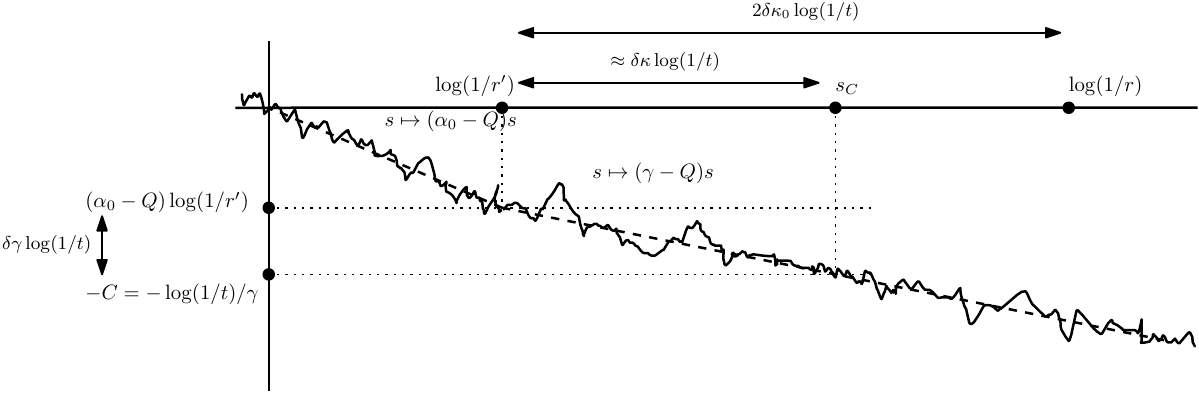}
        \caption{Quantum cone, conditioned on $\cE_x$, in cylindrical coordinates. The main observation in the proof of Lemma \ref{L:mapping} is that $s_C$ finds itself far away from both $\log (1/r)$ and $\log (1/r')$, so that the effect of both conditions (at short and large scales) disappear asymptotically.}
        \label{F:conditionedcone}
    \end{figure}
    
    In these coordinates, the conditioning $\cE_x$ for $h_x$ is simply that 
    $$
    \tilde \cE_x = \Bigl\{\Big|\tilde X_{\log (1/r') } - (\alpha_0 - Q) \log (1/r') \Big| \le 1 \Bigr\}.
    $$
    By the Markov property, the conditional behaviour of $\tilde X_s - (Q -\gamma ) s $ for $s \ge \log (1/r')$, given $\tilde \cE_x$, is that of a Brownian motion with drift $- (Q-\gamma)<0$, starting at time $\log (1/r')$ at a position which is within 1 of the value of $(\alpha_0 -Q) \log (1/r')$. 
    The key observation is that, given $\tilde \cE_x$, 
    \begin{equation} \label{eq:sC}
s_C - \log (1/r') \to +\infty        
    \end{equation}
    in probability as $t\to 0$ (which, in the original coordinate system of $\CC$, boils down to the claimed fact that $r'/r_C \to \infty$). To see \eqref{eq:sC}, note that 
    \begin{equation}\label{eq:C_r'}
C> ( Q - \alpha_0) \log (1/r')
    \end{equation}
    and more precisely 
    $$ 
    C - ( Q - \alpha_0) \log (1/r') = \frac{\delta}{ \gamma} \log (1/t) \to \infty
    $$
    since $r' = t^{\kappa_0 (1- \delta)}$ and $\kappa_0  = 1/ ( \gamma( Q - \alpha_0))$, by definition of $\kappa(\alpha)$ in \eqref{eq:definition of kappa}.
(Visually, the meaning of \eqref{eq:C_r'} is that the horizontal line $y = - C$ lies much below the line $y = (\alpha_0 - Q) \log (1/r')$, which is the value of $\tilde X$ at time $\log (1/r')$ given $\tilde \cE_x$.) 
Therefore, applying the Markov property, the difference $s_C - \log (1/r')$ (which is the first hitting time of some large negative value by a Brownian motion with negative drift) is to first order governed by the law of large numbers:
   \begin{equation}\label{eq:LLN}
\frac{ s_C - \log (1/r')}{\log (1/t)} \to  \frac{\delta}{\gamma( Q - \gamma)} 
    \end{equation}
    in probability, which certainly proves \eqref{eq:sC}. Applying the Markov property of $\tilde X$ (under $\tilde \P_x ( \cdot | \tilde \cE_x)$) at time $\log (1/r')$, and using the exact same proof as that of Theorem 7.11 in \cite{BP} we deduce that the law of $(\tilde X_{s_C + s} + C + ( \gamma - Q) s)_{s\in \RR} $, given $\tilde \cE_x$, converges (in the sense of total variation, on any interval of the form $[s_0, \infty) \subset \RR$ -- crucially we do not consider just compact intervals) to that of the unconditional law of $(\tilde X_s + (\gamma - Q) s)_{s\in \RR}$. After returning to the original coordinate system of $\CC$, this shows that the conditional law of $\bar h_x$, given $\bar \cE_x$, converges to the unconditional law of $\tilde \P_x$. 
    
    To conclude, it remains to prove that $r/ r_C \to 0$ in probability as $t \to 0$. However, this follows at once from \eqref{eq:LLN}, the fact that $\log (1/r) - \log (1/r') = 2\kappa_0 \delta \log (1/t)$, and the inequality 
    \begin{align}\label{eq:kappa and Q}
        2\kappa_0 >  \frac{1}{\gamma(Q - \gamma)}.
    \end{align}
    Indeed, recall that $\kappa_0=(\gamma(Q-\alpha_0))^{-1}$, so \eqref{eq:kappa and Q} is equivalent to $2(Q-\gamma)>Q-\alpha_0$, which can be reformulated as
    \begin{align*}
        2(Q-\gamma)>\sqrt{Q^2-4(1-\mathsf{x})}.
    \end{align*}
    Since $4(1-\mathsf{x})\le4$, it suffices to show (after squaring and rearranging) that
    \begin{align*}
        3Q^2-8Q\gamma+4\gamma^2+4>0.
    \end{align*}
    Using $8\gamma Q=16+4\gamma^2$, this is equivalent to $3Q-12>0$, which is true as $Q>2$. This shows that $r/ r_C \to 0$ in probability.

    This implies that the distance from $x$ to the  rescaled component of the boundary $\partial^* \Omega$, $\partial^* (r_C^{-1} \Omega)$, tends to 0 as $t\to0$, and so Liouville Brownian motion (and thus also Liouville Brownian bridge by absolute continuity), is guaranteed to leave $r_C^{-1} \Omega$ after any strictly positive amount of time with probability tending to 1 as $t\to 0$, since $\partial^* \Omega$ is connected and has positive diameter (so the diameter even tends to infinity after rescaling by $r_C^{-1}$, though it is not necessary to exploit this fact). Thus $$\bar \pb_1^{ \CC \setminus r_C^{-1} \Omega} (x,x) \to \bar \pb_1^{\CC} ( x,x)$$ 
    in law, as desired. This finishes the proof of Lemma \ref{L:mapping}. 
\end{proof}

As explained in \eqref{eq:Fatou}, this proves \eqref{eq:goal_cond_pt_cone}, and thus, \eqref{goal:Econd_pt}. In turn, this proves \eqref{eq:goal_bd} and 
 thus concludes the proof of Proposition \ref{L:boundary}.  
\end{proof}

Combining Proposition \ref{L:boundary}, Lemma \ref{L:bulk_quant}, and the in-out decomposition \eqref{eq:decomposition2}, the proof of Theorem \ref{T:qc} is also complete. \qed  

\appendix \label{appedix}

\section{Proof of Lemma \ref{L:negmoments}}
\label{A:proof_negmom}

Here we provide a proof of Lemma \ref{L:negmoments}; the easier case of positive moments in the unit ball (Lemma \ref{L:posmoments}) is left for the reader to check and in fact can be deduced directly from existing estimates such as Proposition 3.37 in \cite{BP}. The reason why Lemma \ref{L:negmoments} requires a new argument is that it pertains to negative moments \emph{at large times} for Brownian motion, corresponding to negative moments at large scales for the Liouville area measures. Furthermore, at large scales, both the quantum cone and the whole plane GFF under $\P^*_x$ have a behaviour which could cause the ``standard'' arguments (say in \cite{BP}) to break: indeed, on the one hand, beyond the unit ball the $\gamma$-quantum cone does not coincide with $\P^*_0$ and in particular is not Gaussian; while on the other hand, the change in behaviour of the Green function  $G(x, y)$ beyond $|y |=1$ (see \eqref{eq:Greens function of full plane GFF}) causes a change in drift which breaks scale invariance for the whole plane GFF under $\P^*_x$; see \eqref{eq:X*}. Thus careful arguments are needed. 

\begin{proof}[Proof of Lemma \ref{L:negmoments}] Our first observation is that it suffices to consider the case $x = 0$. In the case of the quantum cone this is directly assumed. In the case of the whole-plane GFF, recall that the whole plane GFF is invariant under all Möbius transforms up to additive constant (see, e.g., Lemma 6.26 in \cite{BP}). This implies that $h^*_x = h^*_0 - \Omega$ in law, where $\Omega$ is a Gaussian random variable with fixed variance (corresponding to the average value of $h^*_0$ on the circle of radius 1 around $-x$). The two random variables on the right hand side of this identity in law are not independent, but $\E[(e^{\gamma \Omega})^p]< \infty$ for all $p>0$, so by Hölder's inequality it suffices to prove the result with $h^*_0$ in place of $h^*_x$, as announced. 

The key observation is that both fields $h^*_0$ and $\tilde h$ can be coupled to a third field, which we will denote by $\hat h$, in such a way that
\begin{equation}\label{eq:domination}
h^*_0 \ge \hat h,\quad  \tilde h \ge \hat h
\end{equation}
(in the sense that the difference of the distributions is a nonnegative function), and where the field $\hat h$ can be simply written as 
\begin{equation}\label{eq: hat h}
\hat h (\cdot)  = h + \gamma \log (\tfrac{1}{|z|}); \quad \quad z \in \CC \setminus \{0\}.
\end{equation}
This is essentially obvious in the case of $h^*_0$, and amounts to the observation that 
$$h^*_0(z) =  h (z) + \gamma\log G(0, z)\ge  h + \gamma \log ( 1/ |z|), \quad \quad z \in \CC \setminus \{0\}.$$There is in fact equality in the unit disc $B_0(1)$, and outside the unit by definition of $G$ (see \eqref{eq:Greens function of full plane GFF}) we have $G(0, \cdot) = 0$ while $\log (1/|z|) \le 0$. 

In the case of $\tilde h$ this requires a bit more arguments and boils down to the fact that if $(X_t, t \ge 0)$ is a Brownian motion conditioned so that $X_t - (Q- \gamma ) t \ge 0$ for all $t \ge 0$, then $X_t $ stochastically dominates a Brownian motion. This can be seen either from an explicit computation of the generator, which has a positive drift (see, e.g., the discussion after Definition 7.7 in \cite{BP} where the generator is computed) or, more conceptually, can be seen from the fact that the conditioned process is a certain Doob $h$-transform, where the associated harmonic function, which is here very explicit as it is just an exponential function, is monotone increasing (see e.g. \cite{BenjaminiBerestycki} for related arguments). 

It thus suffices to prove the result with $\hat h$ instead of $h^*_x$ or $\tilde h$, and we will denote the corresponding measure by $\hat \mu$ and quantum clock by $\hat \nu(u), u \ge 0$.

Now, note that by conformal invariance of the whole plane GFF, $(h(zR)_{z\in \CC})$ is a full plane GFF, with a different additive normalisation. Thus $h(zR) - h_R(0)$ has the same law as $h$ (where $h_r(x)$ denote the circle average of $h$ at radius $r$ around $x$), and hence 
\begin{equation}\label{eq:scaling hat h}
\hat h(zR) {=} \hat h'(z) + \gamma\log (\frac{1}{R})+   h_R(0) 
\end{equation}
where $\hat h'$ has the same law as $\hat h$. Fix a Brownian motion $(B_t, t \ge 0)$ and for $t\ge 0$, let $B'_t = R B_{t/R^2}$. Then $(B'_t)_{t\ge 0}$ is also a standard Brownian motion. Let $\hat \nu (u) $ denote the quantum clock of $B$ with respect to $\hat h$ and let $\hat \nu'(u)$ denote the quantum clock of $B'$ with respect to $\hat h'$. Fix $u>0$. Then writing down explicitly the definition of the quantum clock, and noting that 
$$\hat h_\eps (zR) = ( \hat h')_{\eps/R} (z) + \gamma \log (1/R) + h_R(0),
$$
we see that 
\begin{equation}\label{eq:scaling clock}
\hat \nu (u) = e^{\gamma h_R(0) + \gamma^2 \log (1/R)} \hat \nu'( u/R^2) R^{2 + \gamma^2/2}.
\end{equation}
Thus if we suppose $u = R^2$ we find
$$
\hat \nu (u) = e^{ \gamma h_R(0)} \hat \nu'(1) R^{ 2- \gamma^2/2}.
$$
Raising to the power $-q$ and taking expectations on both sides,
\begin{align*}
\E[ \hat \nu (u)^{-q} ] & = R^{-q( 2- \gamma^2/2)} \E [ e^{ - q \gamma h_R(0)} \hat \nu'(1)^{-q} ].  
\end{align*}
Unlike in the exactly scale invariant case discussed e.g. in \cite[Theorem 3.27]{BP}, the two random variables inside the expectation in the right hand side above are not independent. Nevertheless, we can use Hölder's inequality: for every $m>1$, letting $m'$ be its Hölder conjugate (so $1/m  + 1/m' = 1$) we have
$$
\E[ \hat \nu (u)^{-q} ] \le u^{ - q (1 - \gamma^2/4)} \E[ e^{ - q \gamma m h_R(0)}]^{1/m} \E[ \hat \nu'(1)^{ -q m'}]^{1/m'}
$$
The first expectation can be computed exactly (and is equal  to $e^{ \gamma^2 m^2 q^2 \var (h_R(0)) /2}$). Since GMC admits negative moments of all order (Theorem 3.42 in \cite{BP}), the second expectation is some fixed number $C = C_{q,m, \gamma}$ which depends only on $q, m$ and of course $\gamma$. Since $\var (h_R(0)) = \log R + O(1) = (1/2) \log u + O(1)$, we see that for any $m >1$, 
\begin{align*}
    \E[ \hat \nu (u)^{-q} ]  & \le C u^{ - q (1 - \gamma^2/4) } (u^{1/2})^{q^2 m^2 \gamma^2/2}.
\end{align*}
By choosing $m>1$ as close to 1 as needed, we see that the exponent of $u$ above can be taken to be as close  to 
$$- q( 1- \gamma^2/4) + q^2 \gamma^2/4 = \xi^*(-q)/2$$ as desired. This concludes the proof of Lemma \ref{L:negmoments}. 
\end{proof}

\section{The heat content is intrinsic to LQG}
\label{App:K_intrinsic}

We verify that, as announced, the heat content is a quantity which is intrinsic to LQG (i.e., satisfies the LQG change of coordinate formula), even though it is not expected to be spectrally determined. Given a domain $\Omega$ and a field $h$ on $\Omega$, let us write $ \Kb^{(\Omega, h)}(t)$ for the heat content associated to the field $h$ on $\Omega$. 

\begin{prop}[Heat content is intrinsic] \label{P:intrinsic}
    Let $(\Omega,h)\sim(\Omega',h')$ be equivalent as quantum surfaces (i.e., in the sense of \cite{DS2011}). Then
        \begin{align}
        \Kb^{(\Omega, h)}(t) = \Kb_{}^{(\Omega',h')}(t)\quad \text{for all}\quad t\geq0.
    \end{align}
\end{prop}
\begin{proof}
    Let $\phi:\Omega\to\Omega'$ be a conformal isomorphism  and let $(B_s)_{s\ge0}$ be a planar Brownian motion starting at some $z\in\Omega$. Then the Liouville Brownian motion associated to $(\Omega,h)$ is given by $\mathbf{Z}_t = B_{\nu^{-1}(t)}$, where $\nu$ is the quantum clock, 
\begin{equation} \label{eq:quantum clock} 
    \nu(t) = \lim_{\eps \to 0} \eps^{\gamma^2/2}\int_0^t e^{\gamma h_\eps (B_s) } \mathrm{d} s .
    \end{equation}
    By Theorem 1.3 in \cite{Ber2015} the process $(\mathbf{Z}_t, t \ge0)$ satisfies the change of coordinate formulae and so is itself intrinsic to LQG: 
    $\phi(\mathbf{Z}_t)$ is a Liouville Brownian motion in $\Omega'$ associated to the field $h\circ \phi^{-1} + Q \log | ( \phi^{-1})'|$, where $Q = \tfrac{2}{\gamma} + \tfrac{\gamma}{2}$. This is precisely $h'$ since $(\Omega, h) $ and $(\Omega', h')$ are equivalent as quantum surfaces.

    Furthermore, $\phi(\mathbf{Z}_t)=: \mathbf{Z}'_t$ starts at $\mathbf{Z}'_0 = \phi(z)$. Thus, if $\tau_\Omega$ is the exit time by $B$ of the domain $\Omega$, and  $\sigma_\Omega$ (resp. $\sigma'_{\Omega'}$) is the exit time of $\Omega$ (resp. $\Omega'$) by $\mathbf{Z}$ (resp. $\mathbf{Z}'$) $\sigma_\Omega = \nu (\tau_\Omega)$ and $\sigma_\Omega = \sigma'_{\Omega'}$. Thus we find $\mathbf{P}_z ( \sigma_\Omega < t) = \mathbf{P}_{\phi(z)} ( \sigma'_{\Omega'} < t)$. It follows that the heat content on $\Omega$ associated to $h$ is equal to that on $\Omega'$ associated to $h'$, as desired. 
\end{proof}

\section{Example of different one-sided Minkowski dimensions}

\label{App:Minkowski}

As we have not easily found an example of a domain with different one-sided Minkowski dimensions, we will give one here, which in fact is even a Jordan domain. 

We start by describing the construction, which is a variation on the classical ``comb'' domain in which one removes an infinite sequence of parallel slits from a fixed nice domain, such as a square. Simply, we fatten the slits and add them above the square (rather than remove them from it), and let them have a height which decreases to zero as they accumulate towards a boundary point (taken to be a corner of the unit square). It will be important to assign precise rates to the various quantities involved (thickness of rectangles, height, spacing between rectangles).

\begin{figure}
        \centering
    \includegraphics[width=0.8\linewidth]{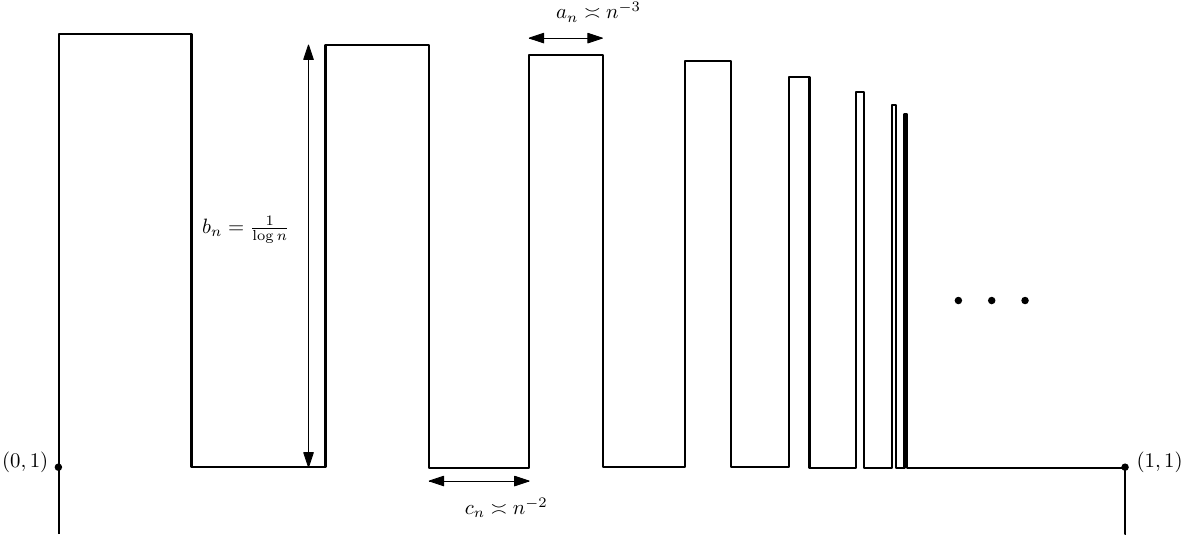}
        \caption{Example of a curve with different one-sided Minkowski dimensions on each side: the inner Minkowski dimension is $4/3$ whereas the outer one is $3/2$.}
        \label{F:differentMinkowskidimenions}
    \end{figure}

We start by describing the domain in words. Let $S = (0,1)^2$ denote the unit square. We will add to the top edge of the square a sequence of rectangles $R_n, n\ge 1$ (i.e., the lower edge of each rectangle will be contained in the segment $[ \mathbf i, \mathbf i +1]$, where $\mathbf i^2 =-1$), sequentially from left to right, as follows. The $n$th rectangle has height $b_n : = (1/\log (n+1))>0$, and thickness or width $a_n:=n^{-3}/2$. The spacing (i.e., distance) between $R_n$ and $R_{n+1}$ is $c_n:= 6/(\pi n^2)- a_n$. In other words, the difference between the bottom-left corners of $R_n$ and $R_{n+1}$ is $6 /(\pi n^2)$. 

We let $\ell_n$ denote the lower edge of $R_n$, which is horizontal segment of the form $[\lambda_n +\mathbf{i}, \lambda_n + a_n + \mathbf{i}]$, where $$\lambda_n = \sum_{i=1}^{n-1} 6/(\pi i^2)  $$ 
is the real part of the bottom-left corner of $R_n$. 
We then define the domain $\Omega$ to be
$$
\Omega = S \cup \left(\bigcup_{n\ge 1} R_n \right) \setminus \left( \bigcup_{n\ge 1} \ell_n\right)
$$
Note that the rectangles $R_n$ (say, the lower left corner, $\lambda_n + \bf i$) accumulate towards the point $\mathbf{ i} + 1$, since $\lim_{n\to \infty} \lambda_n = 1$.

We claim that $\Omega$ may be viewed as a Jordan domain. To do this, it will be useful to denote by $\tilde R_n$ the portion of $\partial \Omega$ which consists of $R_n \setminus \ell_n$ together with the portion of $ \partial S$ connecting $R_n$ and $R_{n+1}$, thus in total $\tilde R_n$ consists of four segments, three from $R_n$ and one from $\partial S$. 

We consider the curve $\eta$ that starts at time 0 from the point $\eta(0) = (0,1) = \mathbf{i}$, and traces $\partial \Omega$ counterclockwise as follows: for each $n\ge 1$, the curve traces $\tilde R_1$, followed by $\tilde R_2$, and so on. After visiting all rectangles, $\eta$ continues from $\bf i + 1$ and visits the remaining three sides of $S$, in counterclockwise order. We need to fix a parameterisation for $\eta:[0,4] \to \partial \Omega$. To this end we fix a positive summable series $\eps_n$ such that $\sum_{n\ge 1} \eps_n = 1$, and require $\eta$ to trace $\tilde R_n$ in time $\eps_n$, at constant speed. 
We set $\eta(1) = \mathbf{i} +1$ and let $\eta$ traverse the rest the remaining three sides of $\partial S$ at unit speed during the interval $[1,4]$.  

The above prescription assigns uniquely a value $\eta(t) \in \partial \Omega$ for each $0\le t \le 4$. The key observation is that $\eta$ is continuous and simple. 
To check continuity, note that only the continuity from the left at value $t=1$ is potentially problematic. However, if $t<1$, then $\eta(t)$ belongs to a uniquely defined $\tilde R_n$, say $\eta(t) \in \tilde R_{n(t)}$, with $n(t)$ uniquely determined by the fact $$\sum_{i=1}^{n(t)} \eps_i \le t < \sum_{i=1}^{n(t)+1} \eps_i.$$ 
Note that $n(t) \uparrow \infty$ as $t\to 1^-$. Hence, as $\text{diam}(\tilde R_n) \to 0$ and $d(\eta(1), \tilde R_n) \to 0$ as $n\to \infty$, we see that $\eta(t) \to \eta(1)$ as $t\to 1^-$. Thus $\eta$ is a continuous curve. 

Furthermore, $\eta$ is simple. Indeed let $0\le  s< t\le 4$. There are various cases to treat, but again the only potentially problematic case is when $t =1$ and $0\le s<1$. However,  if $s< 1$ then $\eta(s) \in \tilde R_{n(s)}$ so $\eta(s) \neq \eta(1)$. We deduce that $\eta$ is a Jordan curve, with $\partial \Omega = \eta([0,4])$.




\medskip The claim is now the following:

\begin{lem}[example of different one-sided Minkowski dimensions]\label{lem:example}
    The Jordan domain $\Omega$ has  inner Minkowski dimension $4/3$ and outer Minkowski dimension $3/2$.
\end{lem}
\begin{proof}
    We have to estimate the outer and inner $\eps$-neighborhood for $\eps>0$ small.\\

    \textbf{Inner neighborhood.} Without the rectangles, the segment from $(0,1)$ to $(1,1)$ would contribute an $\eps$-neighborhood of area exactly $\eps$. We now quantify the additional area contributed by each attached rectangle, distinguishing two regimes:\\

    Consider the area of the $\eps-$(inner) neighbourhood of $\tilde R_n$. Suppose first that $2\eps \le a_n$ (as shown in the first two rectangles in Figure \ref{F:explanation}). Compared to the flat case of $\partial S$, the amount that is added comes from the neighbourhood of the vertical sides of $R_n$ (their $\eps$-neighbourhood are distinct since $2\eps \le a_n$) and some rounding around the corners. Hence 
    $$ |\partial ( R_n \setminus \ell_n)_\eps| =2\eps b_n-2c\eps^2,$$ where we recall that $b_n$ is the rectangle’s height and the term with $c\in(0,1)$ accounts for the rounded corners near the two lower corners.
    
    Now suppose instead that  $2\eps > a_n$. Then the $\eps$-neighborhood fills the rectangle completely (see the third rectangle in Figure \ref{F:explanation}), 
    so  the added area is 
    $$|\partial ( R_n \setminus \ell_n)_\eps| =a_nb_n-\tilde ca_n^2,
    $$
    where $\tilde c\in(0,2)$ can be computed explicitly and again comes from the rounded corners.

    Note that $2\eps \le a_n$ holds precisely when $n\le\lfloor \eps^{-1/3}\rfloor$. Thus, the total area of the inner $\eps$-neighbourhood of this part of $\partial \Omega$
    can be obtained by summing the contribution of all rectangles:
    \begin{align*}
        I_\eps&=\eps+\sum_{n=1}^{\eps^{-1/3+o(1)}}(2\eps b_n-2\eps^2 c)+\sum_{n=\epsilon^{-1/3+o(1)}}^{\infty}a_n(b_n-a_n\tilde c)\\
        &=\eps+\sum_{n=1}^{\eps^{-1/3+o(1)}}\Bigl(\frac{2\eps}{\log(n)}-2\eps^2 c\Bigr)+\sum_{n=\eps^{-1/3+o(1)}}^{\infty}\frac{2}{n^3\log(n)}-\frac{4\tilde c}{n^6}\\
        &=\eps+\eps^{1-1/3+o(1)}+(\eps^{-1/3})^{-(2-1)+o(1)}=\eps^{2/3+o(1)},
    \end{align*}
    where $o(1)$ vanishes when $\eps\to0$. In conclusion, the inner Minkowski dimension is $2-\lim_{\eps\to0}\log(I_\eps)/\log(\eps^{-1})=2-2/3=4/3$.\\

    \textbf{Outer neighborhood.} Similar considerations apply, estimating the contribution of each rectangle. Now we must also account for interactions between neighboring rectangles, depending on whether they are sufficiently separated, corresponding to whether $2\eps\le c_n$ or not.

    If $2\eps\le c_n$, the gap between $R_n$ and $R_{n+1}$ is large enough that their outer $\eps$-neighborhoods do not overlap. The $n$-th rectangle then contributes an additional area $2\eps b_n-2c\eps^2$, with the negative term arising from the rounded corners (for some $c\in(0,1)$).

    If $2\eps>c_n$, the $\eps$-neighborhoods overlap in the gap. In this case, the extra area placed between the two rectangles is at most, say, $2b_n c_n \le C n^{-2}/\log (n)$.

    Combining these two regimes, we obtain the stated decomposition of the total contribution:
    \begin{align*}
        O_\eps&=\eps+\sum_{n=1}^{\eps^{-1/2+o(1)}}(\frac{2\eps}{\log(n)}-2\eps^2 c)+C \sum_{n=\eps^{-1/2+o(1)}}^{\infty}\frac{n^{-2+o(1)}}{\log(n)}\\
        &=\eps+\eps^{1-1/2+o(1)}+\eps^{-1(-1/2)+o(1)}\\
        &=\eps^{1/2+o(1)}
    \end{align*}
    so the outer Minkowski dimension is $2-\lim_{\eps\to0}\log(O_\eps)/\log(\eps)=2-1/2=3/2$. This completes the proof of Lemma \ref{lem:example}.
\end{proof}

\begin{figure}
        \centering
    \includegraphics[width=0.7\linewidth]{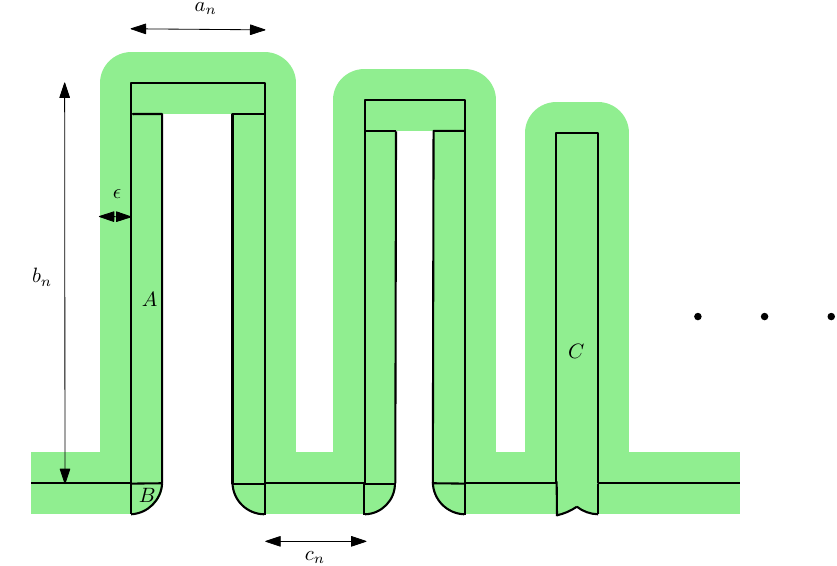}
        \caption{Idea of the proof of Lemma \ref{lem:example}.}
        \label{F:explanation}
    \end{figure}
\bibliographystyle{alpha}
\bibliography{main}

\end{document}